\documentclass[a4paper,reqno]{amsart}
\usepackage[a4paper,left=3.0cm,right=3.0cm,top=3.0cm,bottom=3.0cm]{geometry}
\usepackage{amsfonts}
\usepackage{amssymb}
\usepackage{cite}
\usepackage{fancyhdr}
\usepackage{hyperref}
\usepackage{enumitem}
\usepackage{amsmath}
\usepackage{makecell}
\usepackage{appendix}
\numberwithin{equation}{section}
\newtheorem{thm}{Theorem}[section]
\newtheorem{lemma}{Lemma}[section]
\newtheorem{prop}{Proposition}[section]
\newtheorem{Def}{Definition}[section]
\newtheorem{Rmk}{Remark}[section]
\newtheorem{claim}{Claim}[section]
\newtheorem{coro}{Corollary}[section]
\makeatletter
\newcommand{\lelan}{\left\langle}
\newcommand{\rilan}{\right\rangle}
\newcommand{\rmnum}[1]{\romannumeral #1}
\newcommand{\Rmnum}[1]{\expandafter\@slowromancap\romannumeral #1@}
\makeatother
\title[Energy critical damped wave equations]{Soliton resolution for the energy critical damped wave equations in the radial case}

\author[J. Gu]
{Jingyuan Gu}
\author[L. Zhao]
{Lifeng Zhao}

\address{School of Mathematical Sciences, University of Science and Technology of China, Hefei, Anhui, 230026, PR China}
\email{gjy4869@mail.ustc.edu.cn}

\address{School of Mathematical Sciences, University of Science and Technology of China, Hefei, Anhui, 230026, PR China}
\email{zhaolf@ustc.edu.cn}

\begin{document}
	
	\begin{abstract}
		We consider the energy-critical damped wave equation
		with radial initial data in dimensions six and higher. The equation admits a nontrivial radial stationary solution $W$, called the ground state, which is unique up to sign and scale. The damping term breaks the scaling invariance and destroys the channel of energy estimates. We therefore use the collision-interval method. We prove that any solution with bounded energy norm behaves asymptotically as a superposition of the modulated ground states and a radiation term. In particular, in the global case the solution converges to a pure multi-bubble due to the damping effect. 
	\end{abstract}
	\maketitle

	\section{Introduction}
	\subsection{Setting of the problem}
	We study the Cauchy problem for the damped focusing energy-critical wave equation
	\begin{equation}\label{DNLW}
	\begin{cases}
	u_{tt}-\Delta u+\alpha u_t=\left|u\right|^{\frac{4}{D-2}}u,& (t,x)\in I\times\mathbb{R}^D,\\
	(u(0),\partial_tu(0))=(u_0,u_1),& (u_0,u_1)\in H^1(\mathbb R^D)\times L^2(\mathbb R^D),
	\end{cases}
	\end{equation}
	where $D\geq 6$ is the underlying spatial dimension, $u=u(t,r)\in \mathbb R$ and $r=|x|\in (0,\infty)$ is the radial coordinate in $\mathbb R^D$. The damping parameter $\alpha>0$ is a constant. 
	
	The energy for (\ref{DNLW}) is given by
	\begin{equation*}
	E(u(t),\partial_tu(t)):=\int_0^{\infty}\frac{1}{2}\left[ (\partial_t u(t))^2+(\partial_r u(t))^2 -\frac{D-2}{2D}|u(t)|^{\frac{2D}{D-2}}\right]r^{D-1}dr.
	\end{equation*}
	A direct computation shows  
	\begin{equation}\label{energy decay}
	\frac{d}{dt}E(\boldsymbol{u}(t))=-\alpha \int_{\mathbb{R}^D}|\partial_t u|^2dx,
	\end{equation}
	which indicates that the energy decreases as time progresses. We use boldface to denote pairs of functions, for instance
\(\boldsymbol v=(v,\dot v)\). Then \eqref{DNLW} can be written as
	\begin{equation*}
	\partial_t \boldsymbol{u}(t)=\widetilde J\circ DE(\boldsymbol{u}(t)), \text{ }\boldsymbol{u}(T_0)=\boldsymbol{u}_0,
	\end{equation*}
	where 
	
	\[\widetilde{J}=\left(\begin{array}{ccc}
	0&1\\
	-1&-\alpha
	\end{array}\right)
	\textbf{, }
	DE(\boldsymbol{u}(t))=\left(\begin{array}{ccc}
	-\Delta u(t)-f(u)\\
	\partial_t u(t)
	\end{array}
	\right),
	\]  
	and $f(u)=|u|^{\frac{4}{D-2}}u.$
	We also introduce the energy space $\mathcal{E}$, which is defined as
	\begin{equation*}
	\left\|\boldsymbol{v}\right\|_{\mathcal{E}}^2:=\int_0^{\infty}\left[(\dot v(r))^2+(\partial_r v(r))^2+\frac{(v(r))^2}{r^2}\right]r^{D-1}dr<\infty.
	\end{equation*}
	Several works have been devoted to the Cauchy problem \eqref{DNLW}. The Strichartz estimates of (\ref{DNLW}) were proved by T. Watanabe in \cite{Wat} and T. Inui in \cite{Inui}. Moreover, Inui \cite{Inui} established the local well-posedness and global existence for small initial data in the space $H^1\times L^2$.
	\subsection{Statement of the main result}
		
		In order to state our main result, we introduce the scaling as
		\begin{equation*}
		\boldsymbol{u}(t,r)\to \boldsymbol{u}_{\lambda}(t,r):=(\lambda^{-\frac{D-2}{2}}u(t/\lambda,r/\lambda),\lambda^{-\frac{D}{2}}\dot{u}(t/\lambda,r/\lambda)),\text{ }\lambda>0.
		\end{equation*} 
		We emphasize that the damping term breaks the scaling invariance of equation (\ref{DNLW}).
		
		Equation (\ref{DNLW}) admits a stationary solution $\boldsymbol{W}(x):=(W(x),0)$ where 
		\begin{equation*}
		W(x):=\left(1+\frac{|x|^2}{D(D-2)}\right)^{-\frac{D-2}{2}}.
		\end{equation*}
		It is well known that $W(x)$ is the unique (up to sign, scaling and translation), non-negative and nontrivial $C^2$ solution to 
		\begin{equation*}
		-\Delta W(x)=|W(x)|^{\frac{4}{D-2}}W(x),\text{ }x\in \mathbb{R}^D.
		\end{equation*}
		For each $\lambda>0$, we write $\boldsymbol{W}_{\lambda}(r):=(\lambda^{-\frac{D-2}{2}}W(\lambda^{-1}r),0)$. We define multi-bubble configurations as follows. 
		\begin{Def}[Multi-bubble configuration]\label{Multi-bubble configuration}Given $M\in \mathbb N_0$, $\Vec{\iota}=(\iota_1,...,\iota_M)\in \left\{-1,1\right\}^M$ and an increasing sequence $\Vec{\lambda}=(\lambda_1,\lambda_2,...,\lambda_M)\in (0,\infty)^M$, a multi-bubble configuration is defined by the formula
			\begin{equation*}
			\boldsymbol{\mathcal{W}}(\Vec{\iota},\Vec{\lambda};r):=\sum_{j=1}^M\iota_j \boldsymbol{W}_{\lambda_j}(r).
			\end{equation*}
		\end{Def}
        
    The soliton resolution conjecture predicts that, along the forward evolution,
a bounded solution should decompose into a finite sum of coherent structures,
ordered by their scales, plus a dispersive or radiative remainder. In the
energy-critical setting the coherent structures are precisely the rescaled
ground states \(\pm W_\lambda\). Thus the natural asymptotic object is a
multi-bubble configuration, possibly accompanied by a radiation term in the
finite-time blow-up case. Our main result proves this picture for radial
solutions of the damped equation in dimensions \(D\ge6\).

		\begin{thm}[Soliton Resolution]\label{soliton resolution}Let $D \geq 6$ and let $\boldsymbol{u}(t)$ be a finite energy solution to (\ref{DNLW}) with initial data $\boldsymbol{u}(0)=\boldsymbol{u}_0\in \mathcal{E},$ defined on its maximal forward interval of existence $[0,T_+)$. Suppose that
			\begin{equation}\label{Type 2}
			\limsup_{t \to T_+}{\left\|\boldsymbol{u}(t)\right\|_{H^1\times L^2}}<\infty.    
			\end{equation}
			Then,
			
			(Global solution) if $T_+=\infty$, there exists a time $T_0>0$, an integer $N\geq 0$, continuous functions $\lambda_1(t),...,\lambda_N(t)\in C^0([T_0,T_+)),$ signs $\iota_1,...,\iota_N\in \left\{-1,1\right\}$ and $\boldsymbol{g}(t)\in \mathcal{E}$ defined by
			\begin{equation*}
			\boldsymbol{u}(t)=\sum_{j=1}^{N}\iota_j\boldsymbol{W}_{\lambda_j(t)}+\boldsymbol{g}(t),
			\end{equation*}
			such that
			\begin{equation*}
			\left\|\boldsymbol{g}(t)\right\|_{\mathcal{E}}+\sum_{j=1}^N \frac{\lambda_j(t)}{\lambda_{j+1}(t)}\to 0 \text{ as } t\to \infty,
			\end{equation*}
			where we use the convention that $\lambda_{N+1}(t)=t;$
			
			(Blow-up solution) if $T_+<\infty$, there exists a time $T_0<T_+$, a function $\boldsymbol{u}_0^*\in \mathcal{E}$, an integer $N \geq 1$, continuous functions $\lambda_1(t),...,\lambda_N(t)\in C^0([T_0,T_+)),$ signs $\iota_1,...,\iota_N\in \left\{-1,1\right\}$ and $\boldsymbol{g}(t)\in \mathcal{E}$ defined by
			\begin{equation*}
			\boldsymbol{u}(t)=\sum_{j=1}^{N}\iota_j\boldsymbol{W}_{\lambda_j(t)}+\boldsymbol{u}_0^*+\boldsymbol{g}(t),
			\end{equation*}
			such that
			\begin{equation*}
			\left\|\boldsymbol{g}(t)\right\|_{\mathcal{E}}+\sum_{j=1}^N \frac{\lambda_j(t)}{\lambda_{j+1}(t)}\to 0 \text{ as } t\to T_+,
			\end{equation*}
			where we use the convention that $\lambda_{N+1}=T_+-t$.
		\end{thm}
		
	\begin{Rmk}
The local theory for \eqref{DNLW} is formulated in \(H^1(\mathbb R^D)\times
L^2(\mathbb R^D)\). However, the soliton resolution statement is measured in the
energy topology \(\mathcal E\), which corresponds to the radial
\(\dot H^1(\mathbb R^D)\times L^2(\mathbb R^D)\) topology together with the
Hardy term. This is the natural topology for multi-bubble decompositions, since the quadratic part of the energy controls \(\nabla u\) and \(\partial_tu\), but
does not provide a coercive control of the full \(L^2\)-norm of \(u\).
\end{Rmk}
      Soliton resolution has been extensively studied for energy-critical dispersive
equations. For the focusing energy-critical wave equation, Duyckaerts, Kenig,
and Merle proved soliton resolution in all odd spatial dimensions \(D\ge3\)
\cite{DKM 2011,DKM 2020,DKM 2021,DKM 2023}. The four-dimensional case was
resolved by Duyckaerts, Kenig, Martel, and Merle \cite{DKMM 2021}, and the
six-dimensional case by Collot, Duyckaerts, Kenig, and Merle \cite{CDKM 2022}.
The general even-dimensional case was later studied in
\cite{CDKM 2022.1,CDKM 2023}. A central ingredient in these works is the
channel of energy method, which detects nontrivial radiation near the light cone.

Another related direction concerns the dynamics near the ground state and near
the threshold energy. The threshold scattering/blow-up dichotomy for the
focusing energy-critical wave equation was initiated by Kenig and Merle
\cite{KM 2008}. More refined descriptions near the ground state, including
center-stable manifolds and dynamics away from the ground state, were obtained
by Krieger--Nakanishi--Schlag \cite{KNS 2013,KNS 2015}. Related threshold
dynamics for nonlinear Klein--Gordon equations were developed by Nakanishi and
Schlag \cite{NS 2011,NS 2011 2}. Multi-soliton dynamics for damped
Klein--Gordon equations were studied in \cite{IN}.

The damped wave equation considered here has a different structure. The energy
decay identity and the vanishing result of Inui \cite{Inui} show that the
channel of energy mechanism is not suitable in this setting. We instead follow
the collision-interval strategy of Jendrej and Lawrie \cite{JL}, which proves
full soliton resolution without relying on exterior energy channels. Their
method grew out of the analysis of two-bubble dynamics \cite{JL 2018,JL 2019}
and was further developed for equivariant wave maps \cite{JL 2021}. In the
present paper we adapt this framework to the damped energy-critical wave
equation.
\subsection{The outline of the proof}
We describe the proof in three steps.

\medskip
\noindent
\textbf{Step 1: profile decomposition.}
The first ingredient is a profile decomposition adapted to the damped wave flow.
Since the damping term breaks the scaling symmetry, the decomposition differs from the profile decomposition for the free wave
equation in \cite{BG 1999}. There is at
most one fixed-scale damped profile, while all remaining nontrivial profiles appear
at vanishing scales. After rescaling around such a small scale, the damping
coefficient becomes \(\alpha\lambda_n\to0\), and the limiting profile is therefore
governed by the free energy-critical wave equation. This yields the linear and
nonlinear profile decompositions proved in Section~3.
The one-sided nature of the profile decomposition is useful in two ways. First, it
excludes the large-scale profiles which would be difficult to control in the damped
setting. Second, in the global case it gives an upper bound on the relevant scaling
parameters, which will later be used in the no-return argument.

\medskip
\noindent
\textbf{Step 2: radiation extraction and sequential soliton resolution.}
The profile decomposition is then used to obtain compactness along a sequence of
times after removing the exterior radiation. In the finite-time case, one extracts a
regular radiation term \(\boldsymbol u^*(t)\) outside the backward light cone of the
blow-up point. In the global case, the damping gives an additional decay mechanism:
the exterior radiation vanishes as \(t\to\infty\). These two facts are recorded as follows:

\begin{prop}[Properties of the radiation, finite-time case]\label{Properties of the radiation}Let $\boldsymbol{u}(t)\in \mathcal{E}$ be a finite energy solution to (\ref{DNLW}) on a finite interval $I_*:=[0,T_+],$ $T_+<\infty$ such that (\ref{Type 2}) holds. Then there exists a finite energy solution $\boldsymbol{u}^*(t)\in \mathcal{E}$ to (\ref{DNLW}) called the radiation, and a function $\rho: I_*\to (0,\infty)$ that satisfies
			\begin{equation*}
			\lim_{t\to T_+}\left((\rho(t)/(T_+-t))^{\frac{D-2}{2}}+\left\|\boldsymbol{u}(t)-\boldsymbol{u}^*(t)\right\|_{\mathcal{E}(\rho (t))}\right)=0.
			\end{equation*}
			Moreover, for any $\gamma \in (0,1)$,
			\begin{equation*}
			\left\|\boldsymbol{u}^*(t)\right\|_{\mathcal{E}(0,\gamma (T_+-t))}\to 0 \text{ as }t\to T_+.
			\end{equation*}
		\end{prop}
		In the global case, the radiation term will vanish as time goes to infinity due to the damping effect. In fact, we have
		\begin{prop}[Properties of the radiation, global case]\label{Properties of the radiation global case} Let $\boldsymbol{u}(t)\in \mathcal{E}$ be a finite energy solution to (\ref{DNLW}) on the interval $I_*:=[0, \infty)$ as above such that (\ref{Type 2}) holds. Then there exists a function $\rho: I_*\to (0,\infty)$ that satisfies
			\begin{equation}\label{Properties of the radiation gloabl formular}
			\lim_{t\to \infty}\left((\rho(t)/t)^{\frac{D-2}{2}}+\left\|\boldsymbol{u}(t)\right\|_{\mathcal{E}(\rho (t))}\right)=0.
			\end{equation}
		\end{prop}

After this reduction, the remaining compactness statement is proved by combining
the nonlinear profile decomposition with the localized virial argument. The damping
term is harmless in the finite-time compactness argument, while in the global case
the dissipation identity
\[
    \int_0^\infty \|\partial_tu(t)\|_{L^2}^2\,dt<\infty
\]
is used to obtain the required vanishing. This gives the sequential soliton
resolution:

\begin{thm}[Sequential soliton resolution]\label{Sequential soliton resolution}Let $\boldsymbol{u}(t)\in \mathcal{E}$ be a finite energy solution to (\ref{DNLW}) on an interval $I_*=[0,T_+)$, $T_+\leq \infty$ such that (\ref{Type 2}) holds. In the finite-time case let \(\boldsymbol u^*(t)\) be the radiation from
Proposition~\ref{Properties of the radiation}; in the global case set
\(\boldsymbol u^*(t)\equiv0\). Then there exists an integer $N\geq 0$, a sequence of times $t_n\to T_+$, a vector of signs $\Vec{\iota}\in \left\{-1,1\right\}^N$ and a sequence of scales $\Vec{\lambda}_n\in (0,\infty)^N$ such that
			\begin{equation*}
			\lim_{n\to \infty}\left(\left\|\boldsymbol{u}(t_n)-\boldsymbol{u}^*(t_n)-\boldsymbol{\mathcal{W}}(\Vec{\iota},\Vec{\lambda}_n)\right\|_{\mathcal{E}}+\sum_{j=1}^N\frac{\lambda_{n,j}}{\lambda_{n,j+1}}\right)=0, 
            \end{equation*} 
			where as above we use the convention $\lambda_{n,N+1}:=t_n$ in the global case and $\lambda_{n,N+1}:=T_+-t_n$ in the finite-time case.   
		\end{thm}

\medskip
\noindent
\textbf{Step 3: from sequential to full-time convergence.}
It remains to exclude the possibility that the solution leaves a small
neighborhood of the multi-bubble manifold after entering it along the sequence
\(t_n\). If full convergence failed, one could construct collision intervals:
on each such interval the solution is close to a multi-bubble configuration at
the endpoints but is separated from the multi-bubble manifold at some intermediate
time.

The no-return argument is carried out in Section~5.
First, a geometric exterior-interior decomposition separates the bubbles which
remain well described from the bubbles which actually participate in the collision.
This part uses finite speed of propagation, the exterior radiation estimate, and
the static multi-bubble modulation lemma, and is unchanged by the damping term.
Second, one proves modulation estimates for the interior bubbles. The damping
appears explicitly in the finite-dimensional system through terms such as
\(\beta_j'+\alpha\beta_j\), and these are treated by an exponential integrating
factor in the finite-time case and by the dissipation identity in the global case.
Finally, a localized virial functional is integrated over the collision interval.
In the finite-time case the damping term is absorbed by the weight
\(e^{\alpha(t-T_+)}\); in the global case the additional term
\(-\alpha\mathcal V(t)\) is controlled by the vanishing of the radiation and the
dissipation. This gives a contradiction and proves the full-time soliton
resolution.
\begin{Rmk}
The proof of Theorem~\ref{soliton resolution} is written for \(D\ge6\). The sequential
part of the argument, including the extraction of radiation and the sequential
multi-bubble decomposition, is not the main source of this restriction and can be
carried out in the natural range \(D\ge4\). The difference between \(D=5\) and
higher dimensions comes from the slower decay of \(W\) and the corresponding
low-dimensional nonlinear estimates in the profile decomposition and perturbation
arguments. These difficulties are similar to those treated in the undamped
energy-critical wave equation by the methods of \cite{JL}, and we expect the
case \(D=5\) to be accessible by incorporating those estimates. We do not include
this additional low-dimensional analysis here.

The four-dimensional case is more delicate. In \(D=4\), the ground state and its
scaling direction are not in \(L^2\). Thus the modulation variables used below to
describe the scale velocity, and in particular the corrected scale dynamics, are
not directly available. This would require a different treatment.
\end{Rmk}

The paper is organized as follows. Section~2 contains the
linear estimates, finite speed of propagation, virial identities, and the
multi-bubble tools used later. Section~3 proves the
linear and nonlinear profile decompositions for the damped equation. In
Section~4 we extract the radiation and prove the
sequential soliton resolution. Section~5 upgrades the
sequential convergence to full-time convergence by the collision-interval
no-return argument.

    \section{Preliminaries}
    \subsection{Strichartz estimates and decay estimates} We first introduce the linear propagators for the damped wave equation with a general damping coefficient. For $a\geq 0$, consider the linear equation
    \begin{equation}\label{linear dnlw}
        \partial_{tt}u-\Delta u+a\partial_t u=F.
    \end{equation}
    When $a>0$, we define the operator $\mathcal{K}_a(t)$ by 
	\[
    \widehat{\mathcal K_a(t)f}(\xi)
    :=
    e^{-\frac a2 t}
    \frac{\sin\bigl(t\sqrt{|\xi|^2-a^2/4}\bigr)}
    {\sqrt{|\xi|^2-a^2/4}}
    \widehat f(\xi),
    \qquad t\geq 0.
\]
Here and below the multiplier is understood by analytic continuation in the
low-frequency region \(|\xi|<a/2\), namely
\[
    \frac{\sin\bigl(t\sqrt{|\xi|^2-a^2/4}\bigr)}
    {\sqrt{|\xi|^2-a^2/4}}
    =
    \frac{\sinh\bigl(t\sqrt{a^2/4-|\xi|^2}\bigr)}
    {\sqrt{a^2/4-|\xi|^2}},
    \qquad |\xi|<a/2.
\]
For \(a=0\), we use the convention
\[
    \mathcal K_0(t)=\frac{\sin(t|\nabla|)}{|\nabla|}.
\]
For initial data \((u_0,u_1)\in H^1(\mathbb R^D)\times L^2(\mathbb R^D)\), we
denote by
\[
    S_a(t)(u_0,u_1)=(u(t),\partial_tu(t))
\]
the solution flow of the homogeneous linear equation
\[
    \partial_{tt}u-\Delta u+a\partial_tu=0.
\]
In terms of \(\mathcal K_a(t)\), the solution is given by
\[
    u(t)
    =
    (\partial_t+a)\mathcal K_a(t)u_0+\mathcal K_a(t)u_1.
\]
Indeed,
\[
    \mathcal K_a(0)=0,
    \qquad
    \partial_t\mathcal K_a(0)=\mathrm{Id},
\]
and hence $ u(0)=u_0,
    \qquad
    \partial_tu(0)=u_1.$ When \(a=1\), this agrees with the usual notation in \cite{Inui,InuiWakasugi2021}:
\[
    u(t)
    =
    \partial_t\mathcal K_1(t)u_0+\mathcal K_1(t)(u_0+u_1).
\]
In the inhomogeneous case, assume $u$ satisfies (\ref{linear dnlw})
on a time interval \(I\), then for \(t,t_0\in I\) with \(t\geq t_0\), we have the
Duhamel formula
\[
    u(t)
    =
    (\partial_t+a)\mathcal K_a(t-t_0)u(t_0)
    +
    \mathcal K_a(t-t_0)\partial_tu(t_0)
    +
    \int_{t_0}^t\mathcal K_a(t-s)F(s)\,ds.
\]
Equivalently, we have
\[
    S_a(t-t_0)(u(t_0),\partial_tu(t_0))
    =
    \left(
    (\partial_t+a)\mathcal K_a(t-t_0)u(t_0)
    +
    \mathcal K_a(t-t_0)\partial_tu(t_0),
    \ \partial_tu(t)
    \right),
\]
and the inhomogeneous contribution is generated by
\[
    \int_{t_0}^t\mathcal K_a(t-s)F(s)\,ds.
\]
With these notations, we arrive at the following definition.

	\begin{Def}[Solution]\label{solution}Let $T \in (0,\infty]$. We say that u is a solution to (\ref{DNLW}) on [0,T) if 
		u satisfies $(u,\partial_t u) \in C([0,T):H^1(\mathbb{R}^D)\times L^2(\mathbb{R}^D))$, $\langle \nabla\rangle^{\frac{1}{2}}u \in L_{t,x}^{2(D+1)/(D-1)}(I)$ and $u \in L_{t,x}^{2(D+1)/(D-1)}(I)$ for any compact interval $I \subset [0,T)$,	$(u(0),\partial_t u(0))=(u_0,u_1)$, and the Duhamel formula
        \begin{equation*}
             u(t)
    =
    (\partial_t+\alpha)\mathcal K_\alpha(t)u_0
    +
    \mathcal K_\alpha(t)u_1
    +
    \int_0^t\mathcal K_\alpha(t-s)(\left|u(s)\right|^{\frac{4}{D-2}}u(s))\,ds
        \end{equation*}
		for all $t \in [0,T) .$ We say that u is global if $T=\infty$.

	\end{Def}	
\begin{lemma}[Finite speed of propagation]
\label{lem:finite-speed}
Let \(u\) and \(v\) be two finite-energy solutions to \eqref{DNLW}
on a time interval \(I\). Let \(t_0\in I\), \(x_0\in\mathbb R^D\), and \(R>0\).
Assume that
\[
    (u(t_0),\partial_tu(t_0))
    =
    (v(t_0),\partial_tv(t_0))
    \qquad\text{on }B(x_0,R).
\]
Then $u(t,x)=v(t,x)$
for all \((t,x)\in I\times\mathbb R^D\) such that
\[
    |x-x_0|<R-|t-t_0|.
\]
Equivalently, the value of the solution in a space-time cone depends only on the
initial data inside the base of the cone.
\end{lemma}
\begin{proof}
This follows from the standard local energy identity for the difference
\(u-v\). The damping term is lower order and has the favorable sign in the local
energy estimate, while the nonlinearity is local in \(u\). Hence the usual
domain-of-dependence argument for semilinear wave equations applies.
\end{proof}

    Next, we denote the critical Strichartz norm by
\[
    \|u\|_{S_D(I)}
    :=
    \|u\|_{L_{t,x}^{\frac{2(D+1)}{D-2}}(I\times\mathbb R^D)}.
\]
For \(D\geq6\), we shall also use the following auxiliary norms. Set
\[
    \|u\|_{X(I)}
    :=
    \|u\|_{L_t^{\frac{D(D+1)}{D+2}}
    W_x^{\frac2D,\frac{2(D+1)}{D-1}}(I\times\mathbb R^D)},
\]
\[
    \|F\|_{X'(I)}
    :=
    \|F\|_{L_t^{\frac{D^2+D}{3D+2}}
    W_x^{\frac2D,\frac{2(D+1)}{D+3}}(I\times\mathbb R^D)},
\]
and
\[
    \|u\|_{W(I)}
    :=
    \|u\|_{L_t^{\frac{2(D+1)}{D-1}}
    B_{ \frac{2(D+1)}{D-1},2}^{1/2}(I\times\mathbb R^D)},
\]
\[
    \|F\|_{W'(I)}
    :=
    \|F\|_{L_t^{\frac{2(D+1)}{D+3}}
    B_{ \frac{2(D+1)}{D+3},2}^{1/2}(I\times\mathbb R^D)}.
\]
We also denote by \(S_1(I)\) the finite Besov-Strichartz norm used in the
high-dimensional energy-critical theory. More precisely, \(S_1(I)\) is a finite
maximum of norms of the form
\[
    L_t^q B_{r,2}^{1-\gamma(r)}(I\times\mathbb R^D),
    \qquad
    \frac1q=\frac{D-1}{2}\left(\frac12-\frac1r\right),
    \qquad
    \gamma(r)=\frac{D+1}{2}\left(\frac12-\frac1r\right),
\]
chosen so that
\[
    \|u\|_{W(I)}+\|u\|_{X(I)}+\|u\|_{Y(I)}
    \lesssim
    \|u\|_{S_1(I)}
\]
for the auxiliary Sobolev norm \(Y(I)\) appearing in the nonlinear estimates.
For \(D\geq6\), we define
\[
    \|u\|_{X_D(I)}
    :=
    \|u\|_{S_D(I)}
    +
    \|u\|_{X(I)}
    +
    \|u\|_{W(I)}
    +
    \|u\|_{S_1(I)},
\]
and
\[
    \|F\|_{N_D(I)}
    :=
    \|F\|_{X'(I)}
    +
    \|F\|_{W'(I)}.
\]
For \(4\leq D\leq5\), the same notation \(X_D(I)\) and \(N_D(I)\) will denote
the corresponding standard Strichartz solution and forcing spaces; in these
dimensions no exotic Besov component is needed.
We shall use the following estimates. For \(a=1\), the non-endpoint Strichartz
estimates were proved by Inui, and the wave-endpoint case was proved in \cite{InuiWakasugi2021}. For \(a>0\), the estimates follow by the scaling
\[
    v(s,y)=a^{-\frac{D-2}{2}}u(s/a,y/a).
\]
In what follows, the damping coefficient always satisfies \(0\leq a\leq\alpha\),
and the constants in the Strichartz estimates below are uniform in this range.
First, the homogeneous estimate reads
\[
    \|S_a(t-t_0)(u_0,u_1)\|_{X_D(I)}
    \lesssim
    \|(u_0,u_1)\|_{H^1\times L^2}.
\]
In particular,
\[
    \|S_a(t-t_0)(u_0,u_1)\|_{S_D(I)}
    \lesssim
    \|(u_0,u_1)\|_{H^1\times L^2}.
\]
The inhomogeneous estimate is
\[
    \left\|
    \int_{t_0}^t\mathcal K_a(t-s)F(s)\,ds
    \right\|_{X_D(I)}
    \lesssim
    \|F\|_{N_D(I)}.
\]
The nonlinear estimates used below are summarized as follows. Let $f(u)=|u|^{\frac4{D-2}}u.$
Then, for \(D\geq6\),
\[
    \|f(u)\|_{X'(I)}
    \lesssim
    \|u\|_{X(I)}^{\theta\frac4{D-2}+1}
    \|u\|_{S_1(I)}^{(1-\theta)\frac4{D-2}},
\]
and
\[
    \|f(u)\|_{W'(I)}
    \lesssim
    \|u\|_{X(I)}^{\theta\frac4{D-2}}
    \|u\|_{S_1(I)}^{(1-\theta)\frac4{D-2}+1},
\]
for a constant \(\theta=\theta(D)\in(0,1)\). Moreover, the difference estimate
\[
    \|f(u)-f(v)\|_{N_D(I)}
    \leq
    C\bigl(\|u\|_{X_D(I)},\|v\|_{X_D(I)}\bigr)
    \|u-v\|_{X(I)}
\]
holds on bounded \(X_D(I)\)-balls, and the constant is small whenever the relevant
critical \(S_D\)-norms are small.
\begin{Rmk}
The complete inhomogeneous Strichartz estimates contain an additional derivative
loss parameter \(\delta\), depending on the admissible pairs. We shall not record
the full table of \(\delta\), see \cite{Inui,InuiWakasugi2021}. In this paper, the estimates are used only through the
above \(X_D(I)\)-\(N_D(I)\) framework. The Besov-type estimates required in
dimensions \(D\geq6\) are absorbed into the definitions of \(X_D(I)\) and
\(N_D(I)\).
\end{Rmk}
\begin{Rmk}
The uniformity in \(a\in[0,\alpha]\) is only used for the Strichartz and
perturbative estimates above. We do not claim uniform \(L^p\)-\(L^q\) decay
estimates for all \(a\in[0,\alpha]\). This distinction will be important when
small-scale profiles are considered: after rescaling by \(\lambda_n\), the damping
coefficient becomes \(a_n=\alpha\lambda_n\), and hence \(a_n\in[0,\alpha]\) and
\(a_n\to0\) whenever \(\lambda_n\to0\).
\end{Rmk}
We next record the \(L^p\)-\(L^q\) decay estimates for the linear damped flow
with the fixed damping coefficient \(\alpha>0\). These estimates reflect the
diffusion phenomenon of the damped wave equation. Here \(P_{\leq \alpha}\) and \(P_{>\alpha}\) denote smooth Fourier cutoffs to the regions
\(|\xi|\lesssim \alpha\) and \(|\xi|\gtrsim \alpha\), respectively.
\begin{lemma}[\(L^p\)-\(L^q\) estimates (Theorem 1.1, \cite{MTMY})]
\label{Lp Lq decay}
Let \(1\leq q\leq p<\infty\), \(p\neq1\), and let \(s_1\leq s_2\). Set
\[
    \beta=\beta(p):=(D-1)\left|\frac12-\frac1p\right|.
\]
Then there exist constants \(C>0\), \(c_\alpha>0\), and \(\delta_p>0\), depending
on \(D,p,q,s_1,s_2\) and on the fixed damping coefficient \(\alpha\), such that
for all \(t>0\),
\[
\begin{aligned}
    \bigl\||\nabla|^{s_1}\mathcal K_\alpha(t)g\bigr\|_{L^p}
    \leq {}&
    C\langle t\rangle^{-\frac D2(\frac1q-\frac1p)-\frac{s_1-s_2}{2}}
    \bigl\||\nabla|^{s_2}P_{\leq\alpha}g\bigr\|_{L^q}
\\
    &\quad
    +C e^{-c_\alpha t}\langle t\rangle^{\delta_p}
    \bigl\||\nabla|^{s_1}P_{>\alpha}g\bigr\|_{H_p^{\beta-1}},
\end{aligned}
\]
and
\[
\begin{aligned}
    \bigl\||\nabla|^{s_1}\partial_t\mathcal K_\alpha(t)g\bigr\|_{L^p}
    \leq {}&
    C\langle t\rangle^{-\frac D2(\frac1q-\frac1p)-\frac{s_1-s_2}{2}-1}
    \bigl\||\nabla|^{s_2}P_{\leq\alpha}g\bigr\|_{L^q}
\\
    &\quad
    +C e^{-c_\alpha t}\langle t\rangle^{\delta_p}
    \bigl\||\nabla|^{s_1}P_{>\alpha}g\bigr\|_{H_p^{\beta-1}}.
\end{aligned}
\]
\end{lemma}
\begin{Rmk}
The constants in Lemma~\ref{Lp Lq decay} are not asserted to be uniform as
\(\alpha\to0\). This is consistent with the fact that the diffusion phenomenon is
a genuinely damped effect and disappears in the undamped limit. In the sequel,
this lemma is applied only to the original equation with the fixed coefficient
\(\alpha>0\).
\end{Rmk}

We now return to the perturbative theory associated with the Strichartz spaces
introduced above. Let \(\pi_1(u,u_t):=u\) denote the projection onto the first
component. This is the standard long-time perturbation theorem for the
energy-critical damped wave equation, written in a form that is uniform for
damping coefficients \(a\in[0,\alpha]\).
\begin{lemma}[Long-time perturbation]
\label{lem:long-time-perturbation}
Let \(D\geq4\), \(I=[t_0,t_1]\), and \(0\leq a\leq\alpha\). Let
\(\widetilde u\) be an approximate solution to
\[
    \partial_{tt}\widetilde u-\Delta \widetilde u
    +a\partial_t\widetilde u
    =
    f(\widetilde u)+e
\]
on \(I\times\mathbb R^D\), where $ f(\widetilde u)=|\widetilde u|^{\frac4{D-2}}\widetilde u.$ Assume that $\|\widetilde u\|_{X_D(I)}\leq M.$
Then there exist constants
\[
    \varepsilon_0=\varepsilon_0(M)>0,\qquad
    C=C(M)>0,\qquad
    c_D>0,
\]
independent of \(a\in[0,\alpha]\), with the following property. If
\[
    \left\|
    \pi_1 S_a(t-t_0)
    \bigl(
        u_0-\widetilde u(t_0),\
        u_1-\partial_t\widetilde u(t_0)
    \bigr)
    \right\|_{X_D(I)}
    +
    \|e\|_{N_D(I)}
    \leq \varepsilon
    \leq \varepsilon_0,
\]
then there exists a solution \(u\) to
\[
    \partial_{tt}u-\Delta u+a\partial_tu=f(u)
\]
on \(I\), with initial data $(u(t_0),\partial_tu(t_0))=(u_0,u_1)$, 
such that
\[
    \|u-\widetilde u\|_{X_D(I)}
    +
    \|(u-\widetilde u,\partial_tu-\partial_t\widetilde u)\|
        _{L_t^\infty(I;H^1\times L^2)}
    \leq
    C\varepsilon^{c_D}.
\]
\end{lemma}
\begin{proof}
We recall the standard argument, emphasizing the uniformity in
\(a\in[0,\alpha]\). Set $w:=u-\widetilde u.$
Then \(w\) solves
\[
    \partial_{tt}w-\Delta w+a\partial_t w
    =
    f(\widetilde u+w)-f(\widetilde u)-e,
\]
with initial data
\[
    (w(t_0),\partial_t w(t_0))
    =
    (u_0-\widetilde u(t_0),\
    u_1-\partial_t\widetilde u(t_0)).
\]
By the homogeneous and inhomogeneous Strichartz estimates stated above, together
with the nonlinear estimates in the \(X_D(I)\)-\(N_D(I)\) framework, there exists
\(\eta=\eta(M)>0\) such that on any subinterval \(J\subset I\) satisfying $\|\widetilde u\|_{X_D(J)}\leq \eta,$
one has
\[
\begin{aligned}
    \|w\|_{X_D(J)}
    \lesssim{}&
    \left\|
    \pi_1 S_a(t-\inf J)
    \bigl(w(\inf J),\partial_t w(\inf J)\bigr)
    \right\|_{X_D(J)}
    +
    \|e\|_{N_D(J)}
\\
    &\quad
    +
    o_\eta(1)\|w\|_{X_D(J)} .
\end{aligned}
\]
Choosing \(\eta\) sufficiently small, the last term can be absorbed. Since $\|\widetilde u\|_{X_D(I)}\leq M,$
the interval \(I\) can be divided into \(N=N(M)\) subintervals on which the above
smallness condition holds. Iterating the short-time estimate over these
subintervals gives
\[
    \|w\|_{X_D(I)}
    +
    \|(w,\partial_t w)\|_{L_t^\infty(I;H^1\times L^2)}
    \leq
    C(M)\varepsilon^{c_D}.
\]
All constants are uniform for \(a\in[0,\alpha]\), because the Strichartz estimates
and the nonlinear estimates used in the argument are uniform in this range.
\end{proof}
\begin{Rmk}
The uniformity with respect to \(a\in[0,\alpha]\) is essential only for the
perturbative arguments involving rescaled profiles. Indeed, if a profile is
rescaled by a factor \(\lambda_n\), then the damping coefficient becomes
\(a_n=\alpha\lambda_n\). Thus small-scale profiles correspond to \(a_n\to0\),
and the perturbation theory must be compatible with the limiting undamped
energy-critical wave equation.
\end{Rmk}
Finally, we recall the decay property for the damped flow with the fixed damping
coefficient \(\alpha>0\). This result will be used only in the global-in-time
analysis, where no uniformity as \(\alpha\to0\) is required.
\begin{lemma}[Decay of finite Strichartz solutions]
\label{lem:inui-decay}
Let \(u\) be a global solution to \eqref{DNLW}
on \([0,\infty)\times\mathbb R^D\). Assume that $\|u\|_{S_D([0,\infty))}<\infty .$
Then
\[
    \lim_{t\to\infty}
    \left(
    \|u(t)\|_{H^1}
    +
    \|\partial_tu(t)\|_{L^2}
    \right)
    =0.
\]
\end{lemma}
\begin{Rmk}
In Lemma~\ref{lem:inui-decay}, the damping coefficient is the fixed coefficient
\(\alpha\) of the original equation. Therefore the possible loss of uniformity in
the linear decay estimates as \(a\to0\) is irrelevant for this application.
\end{Rmk}
   
		\subsection{Virial identities} We have the following virial identities.
	\begin{lemma}[Virial identity]\label{Virial identity} Let $\boldsymbol{u}(t)\in\mathcal{E}$ be a solution to (\ref{DNLW}) on an open time interval $I$ and $\rho: I\to (0,\infty)$ a Lipschitz function. Then for almost all $t\in I$, a direct computation shows
    \begin{equation*}\label{virial identity 3}
	\begin{aligned}
	&\frac{d}{dt}\left\langle \partial_t u(t)\mid \chi_{\rho(t)}\left(r\partial_ru(t)+\frac{D-2}{2}u(t)\right)\right\rangle\\
    =&-\int_0^{\infty}(\partial_t u(t,r))^2\chi_{\rho(t)}(r)r^{D-1}dr
	-\alpha \left\langle \partial_t u(t)\mid \chi_{\rho(t)}\left(r\partial_ru(t)+\frac{D-2}{2}u(t)\right)\right\rangle\\
	&+\Omega_{1,\rho(t)}(\boldsymbol{u}(t))+\frac{D-2}{2}\Omega_{2,\rho(t)}(\boldsymbol{u}(t));
	\end{aligned}
	\end{equation*}
	and 
	\begin{equation*}
	\begin{aligned}
	&\frac{d}{dt}\left\langle \partial_t u(t)\mid \chi_{\rho(t)}\left(r\partial_ru(t)+\frac{D}{2}u(t)\right)\right\rangle\\
    =& -\int_0^{\infty}\left[ (\partial_r u(t,r))^2-|u(t,r)|^{\frac{2D}{D-2}} \right]\chi_{\rho(t)}(r)r^{D-1}dr\\
    &-\alpha \left\langle \partial_t u(t)\mid \chi_{\rho(t)}\left(r\partial_ru(t)+\frac{D}{2}u(t)\right)\right\rangle
	+\Omega_{1,\rho(t)}(\boldsymbol{u}(t))+\frac{D}{2}\Omega_{2,\rho(t)}(\boldsymbol{u}(t)),
	\end{aligned}
	\end{equation*}
		where 
		\begin{equation*}
		\begin{aligned}
		\Omega_{1,\rho(t)}(\boldsymbol{u}(t)):=&-\frac{\rho'(t)}{\rho(t)}\int_0^{\infty}\partial_t u(t,r)r\partial_r u(t,r)(r\partial_r \chi)(r/\rho(t))r^{D-1}dr\\
		&-\frac{1}{2}\int_0^{\infty}\left((\partial_t u(t,r))^2+(\partial_r u(t,r))^2\right)(r\partial_r \chi)(r/\rho(t))r^{D-1}dr\\
		&+\frac{1}{2}\int_0^{\infty}\frac{D-2}{D}|u(t,r)|^{\frac{2D}{D-2}}(r\partial_r \chi)(r/\rho(t))r^{D-1}\,dr,\\[0.4em]
		\Omega_{2,\rho(t)}(\boldsymbol{u}(t)):=&-\frac{\rho'(t)}{\rho(t)}\int_0^{\infty}\partial_t u(t,r)u(t,r)(r\partial_r\chi)(r/\rho(t))r^{D-1}\,dr\\
		&-\int_0^{\infty}\partial_r u(t,r)\frac{u(t,r)}{r}(r\partial_r \chi)(r/\rho(t))\,r^{D-1}\,dr.
		\end{aligned}
		\end{equation*} 
	\end{lemma}
	
	\subsection{Multi-bubble configurations} In this section we study properties of finite energy maps near a multi-bubble configuration. First, we define the infinitesimal generators of the \(\dot H^1\)-invariant dilations by $\Lambda$ and we denote the \(L^2\)-invariant scaling generator by \(\underline\Lambda\).
	\begin{equation*}
	\Lambda:=r\partial_r +\frac{D-2}{2},\qquad\underline{\Lambda}:=r\partial_r+\frac{D}{2}.
	\end{equation*}
	We have 
	\begin{equation*}
	\Lambda W(r)=\left(\frac{D-2}{D}-\frac{r^2}{2D}\right)\left(1+\frac{r^2}{D(D-2)}\right)^{-\frac{D}{2}}.
	\end{equation*}
	Note that both $W$ and $\Lambda W$ satisfy
	\begin{equation*}
	|W(r)|,|\Lambda W(r)|\simeq 1 \text{ if } r\leq 1, \text{ and }|W(r)|,|\Lambda W(r)|\simeq r^{-D+2} \text{ if } r\geq 1.
	\end{equation*}
	Next, we discuss the spectral properties. The operator $\mathcal{L}_{\mathcal{W}}$ obtained by linearization of (\ref{DNLW}) about an $M$-bubble configuration $\boldsymbol{\mathcal{W}}(\Vec{\iota},\Vec{\lambda})$ is given by 
	\begin{equation*}
	\mathcal{L}_{\mathcal{W}}:=D^2E_P(\boldsymbol{\mathcal{W}}(\Vec{\iota},\Vec{\lambda}))g=-\Delta g-f'(\boldsymbol{\mathcal{W}}(\Vec{\iota},\Vec{\lambda}))g
	\end{equation*}
	where $f(z):=|z|^{\frac{4}{D-2}}z$ and $f'(z)=\frac{D+2}{D-2}|z|^{\frac{4}{D-2}}$. Given $\boldsymbol{g}=(g,\dot{g})\in \mathcal{E}$,
	\begin{equation*}
	\left\langle D^2E(\boldsymbol{\mathcal{W}}(\Vec{\iota},\Vec{\lambda}))\boldsymbol{g}\mid \boldsymbol{g} \right\rangle=\int_0^{\infty}\left(\dot{g}(r)^2+(\partial_r g(r))^2-f'(\boldsymbol{\mathcal{W}}(\Vec{\iota},\Vec{\lambda}))g(r)^2\right)r^{D-1}dr.
	\end{equation*}
	In the one bubble case, we consider $\mathcal{W}=W_{\lambda}$  and use the notation,
	\begin{equation*}
	\mathcal{L}_{\lambda}=-\Delta-f'(W_{\lambda}).
	\end{equation*}
	In particular, we write $\mathcal{L}:=\mathcal{L}_1$. Importantly,
	\begin{equation*}
	\mathcal{L}(\Lambda W)=\frac{d}{d\lambda} |_{\lambda=1}(-\Delta W_{\lambda}-f(W_{\lambda}))=0.
	\end{equation*}
	Thus, if \(D\ge5\), \(\Lambda W\in L^2\) is the zero mode of \(\mathcal L\);
if \(D=4\), it is a threshold resonance. In fact, in the radial case  $\left\{f\in \dot{H}^1_{rad} : \mathcal{L}f=0 \right\}=\operatorname{span}\left\{\Lambda W\right\}$. In addition to this fact, it was shown in \cite{DM 08} that $\mathcal{L}$ has a unique negative simple eigenvalue that we denote by $-\kappa^2<0$ (with $\kappa>0$). We denote the associated eigenfunction by $\mathcal{Y}$ normalized in $L^2$ so that $\left\|\mathcal{Y}\right\|_{L^2}=1$. By elliptic regularity $\mathcal{Y}$ is smooth, and by Agmon estimates it decays exponentially. Using that $\mathcal{L}$ is symmetric we deduce that $\left\langle \mathcal{Y}\mid \Lambda W \right\rangle=0.$ Let 
	\begin{equation*}
	\boldsymbol{\mathcal{Y}}^-:=(\frac{1}{\kappa}\mathcal{Y},-\mathcal{Y}), \text{ }\boldsymbol{\mathcal{Y}}^+:=(\frac{1}{\kappa}\mathcal{Y},\mathcal{Y})\quad\text{and }{J}=\left(\begin{array}{ccc}
	0&1\\
	-1&0
	\end{array}\right).
	\end{equation*}
     Then we define
	\begin{equation*}
	\boldsymbol{\alpha}^-=\frac{\kappa}{2}J\boldsymbol{\mathcal{Y}}^+=\frac{1}{2}(\kappa\mathcal{Y},-\mathcal{Y}),\quad\boldsymbol{\alpha}^+:=-\frac{\kappa}{2}J\boldsymbol{\mathcal{Y}}^-=\frac{1}{2}(\kappa\mathcal{Y},\mathcal{Y}).
	\end{equation*}
	Recalling that 
	\begin{equation*}
	J \circ D^2E(\boldsymbol{W})=\left(\begin{array}{ccc}
	0&Id\\
	-\mathcal{L}&0
	\end{array}\right)
	\end{equation*}
	we see that 
	\begin{equation*}
	J \circ D^2E(\boldsymbol{W}) \boldsymbol{\mathcal{Y}}^-=-\kappa \boldsymbol{\mathcal{Y}}^-,\text{ and }J \circ D^2E(\boldsymbol{W}) \boldsymbol{\mathcal{Y}}^+=\kappa \boldsymbol{\mathcal{Y}}^+
	\end{equation*}
	and for all $\boldsymbol{h} \in \mathcal{E}$,
	\begin{equation*}
	\left\langle \boldsymbol{\alpha}^-\mid J \circ D^2E(\boldsymbol{W})\boldsymbol{h} \right\rangle=-\kappa \left\langle \boldsymbol{\alpha}^-\mid \boldsymbol{h}  \right\rangle,\text{ } \left\langle \boldsymbol{\alpha}^+\mid J \circ D^2E(\boldsymbol{W})\boldsymbol{h}  \right\rangle=\kappa \left\langle \boldsymbol{\alpha}^+\mid \boldsymbol{h}  \right\rangle.
	\end{equation*}
	We view $\boldsymbol{\alpha}^{\pm}$ as linear forms on $\mathcal{E}$ and note that $\left\langle \boldsymbol{\alpha}^-\mid \boldsymbol{\mathcal{Y}}^-\right\rangle=\left\langle \boldsymbol{\alpha}^+\mid \boldsymbol{\mathcal{Y}}^+\right\rangle=1$, $\left\langle \boldsymbol{\alpha}^-\mid \boldsymbol{\mathcal{Y}}^+\right\rangle=\left\langle \boldsymbol{\alpha}^+\mid \boldsymbol{\mathcal{Y}}^-\right\rangle=0$. For $\lambda>0$ similarly we define,
	\begin{equation*}
	\boldsymbol{\mathcal{Y}}^-_{\lambda}:=(\frac{1}{\kappa}\mathcal{Y}_{\lambda},-\mathcal{Y}_{\underline{\lambda}}), \text{ }\boldsymbol{\mathcal{Y}}^+_{\lambda}:=(\frac{1}{\kappa}\mathcal{Y}_{\lambda},\mathcal{Y}_{\underline{\lambda}})
	\end{equation*}
	and
	\begin{equation}\label{eq:alpha-pm}
	\boldsymbol{\alpha}^-_{\lambda}=\frac{\kappa}{2\lambda}J\boldsymbol{\mathcal{Y}}^+_{{\lambda}}=\frac{1}{2}(\frac{\kappa}{\lambda}\mathcal{Y}_{\underline{\lambda}},-\mathcal{Y}_{\underline{\lambda}}),\text{ }\boldsymbol{\alpha}^+_{{\lambda}}:=-\frac{\kappa}{2\lambda}J\boldsymbol{\mathcal{Y}}^-_{\lambda}=\frac{1}{2}(\frac{\kappa}{\lambda}\mathcal{Y}_{\underline{\lambda}},\mathcal{Y}_{\underline{\lambda}}).  
	\end{equation}
	With these scalings, we have $\left\langle \boldsymbol{\alpha}^-_{\lambda}\mid \boldsymbol{\mathcal{Y}}^-_{\lambda}\right\rangle=\left\langle \boldsymbol{\alpha}^+_{\lambda}\mid \boldsymbol{\mathcal{Y}}^+_{\lambda}\right\rangle=1$. We have
	\begin{equation*}
	J \circ D^2E(\boldsymbol{W}_{\lambda}) \boldsymbol{\mathcal{Y}}^-_{\lambda}=-\frac{\kappa}{\lambda}\boldsymbol{\mathcal{Y}}^-_{\lambda},\text{ and }J \circ D^2E(\boldsymbol{W}_{\lambda}) \boldsymbol{\mathcal{Y}}^+_{\lambda}=\frac{\kappa}{\lambda} \boldsymbol{\mathcal{Y}}^+_{\lambda}
	\end{equation*}
	and for all $\boldsymbol{h} \in \mathcal{E}$.
	\begin{equation*}
	\left\langle \boldsymbol{\alpha}^-_{\lambda}\mid J \circ D^2E(\boldsymbol{W}_{\lambda})\boldsymbol{h}  \right\rangle=-\frac{\kappa}{\lambda} \left\langle \boldsymbol{\alpha}^-_{\lambda}\mid \boldsymbol{h}  \right\rangle,\text{ } \left\langle \boldsymbol{\alpha}^+_{\lambda}\mid J \circ D^2E(\boldsymbol{W}_{\lambda})\boldsymbol{h}  \right\rangle=\frac{\kappa}{\lambda} \left\langle \boldsymbol{\alpha}^+_{\lambda}\mid \boldsymbol{h}  \right\rangle.
	\end{equation*}
	
	We next choose the test function used in the orthogonality conditions of the
static modulation lemma. If \(D\ge7\), we set $ Z:=\Lambda W .$
Then \(Z\in \dot H^{-1}\), and
\[
    \langle Z,\Lambda W\rangle=\|\Lambda W\|_{L^2}^2>0,
    \qquad
    \langle Z,Y\rangle=0 .
\]
Here the second identity follows from the symmetry of \(L\), since
\(L\Lambda W=0\) and \(LY=-\kappa^2Y\).
In dimensions \(4\le D\le6\), the function \(\Lambda W\) cannot be used as a
test function in the same way. We therefore fix once and for all
\(Z\in C_0^\infty(0,\infty)\) such that
\[
    \langle Z,\Lambda W\rangle>0,
    \qquad
    \langle Z,Y\rangle=0 .
\]
Such a choice is possible by density and the identity
\(\langle Y,\Lambda W\rangle=0\). For \(\lambda>0\), we denote by \(Z_\lambda\)
the corresponding \(L^2\)-scaled function.

\subsection{Static multi-bubble estimates}
The estimates in this subsection are purely elliptic and do not involve the
damping term. We state them in their natural range \(D\ge4\), following the
corresponding static tools for the energy-critical wave equation. In the proof of
Theorem~\ref{soliton resolution}, only the case \(D\ge6\) will be used.
We record the following localized coercivity estimate around the ground state.
\begin{lemma}[Localized coercivity around \(W\)]
\label{lem:localized-coercivity}
Let \(D\ge4\). There exist constants \(c\in(0,1/2)\) and \(C>0\) such that, for
all \(\boldsymbol g=(g,0)\in\mathcal E\),
\begin{equation*}
    \langle Lg,g\rangle
    \ge
    c\|\boldsymbol g\|_{\mathcal E}^2
    -
    C\langle Z,g\rangle^2
    -
    C\langle Y,g\rangle^2 .
\end{equation*}
Moreover, if \(R>0\) is sufficiently large, then
\begin{equation*}
\begin{aligned}
    &(1-2c)\int_0^R |\partial_rg(r)|^2r^{D-1}\,dr
    +
    c\int_R^\infty |\partial_rg(r)|^2r^{D-1}\,dr      \\
    &\qquad
    -
    \int_0^\infty f'(W(r))g(r)^2r^{D-1}\,dr
    \ge
    -
    C\langle Z,g\rangle^2
    -
    C\langle Y,g\rangle^2 .
\end{aligned}
\end{equation*}
\end{lemma}
The next estimate gives the energy expansion for a separated multi-bubble.
\begin{lemma}[Energy expansion for separated bubbles]
\label{lem:energy-expansion-separated-bubbles}
Let \(D\ge4\) and \(M\in\mathbb N\). For every \(\theta>0\), there exists
\(\eta>0\) with the following property. Let
\[
    \boldsymbol W(\boldsymbol\iota,\boldsymbol\lambda)
    :=
    \sum_{j=1}^M \iota_j\boldsymbol W_{\lambda_j}
\]
be an \(M\)-bubble configuration satisfying
\[
    \sum_{j=1}^{M-1}
    \left(\frac{\lambda_j}{\lambda_{j+1}}\right)^{\frac{D-2}{2}}
    \le\eta .
\]
Then
\begin{equation*}
\left|
    E(\boldsymbol W(\boldsymbol\iota,\boldsymbol\lambda))
    -
    M E(\boldsymbol W)
    +
    \frac{(D(D-2))^{\frac D2}}{D}
    \sum_{j=1}^{M-1}
    \iota_j\iota_{j+1}
    \left(\frac{\lambda_j}{\lambda_{j+1}}\right)^{\frac{D-2}{2}}
\right|
\le
\theta
\sum_{j=1}^{M-1}
\left(\frac{\lambda_j}{\lambda_{j+1}}\right)^{\frac{D-2}{2}} .
\end{equation*}
Moreover, there exists \(C>0\) such that, for all
\(\boldsymbol g=(g,0)\in\mathcal E\),
\begin{equation*}
    \left|
        \langle DE(\boldsymbol W(\boldsymbol\iota,\boldsymbol\lambda)),
        \boldsymbol g\rangle
    \right|
    \le
    C\|\boldsymbol g\|_{\mathcal E}
    \sum_{j=1}^{M-1}
    \left(\frac{\lambda_j}{\lambda_{j+1}}\right)^{\frac{D-2}{2}} .
\end{equation*}
\end{lemma}
We next define the static distance to the \(M\)-bubble manifold. For
\(\boldsymbol v\in\mathcal E\), set
\begin{equation*}
    d_M(\boldsymbol v)
    :=
    \inf_{\boldsymbol\iota,\boldsymbol\lambda}
    \left(
        \left\|
            \boldsymbol v
            -
            \sum_{j=1}^M\iota_j\boldsymbol W_{\lambda_j}
        \right\|_{\mathcal E}^2
        +
        \sum_{j=1}^{M-1}
        \left(\frac{\lambda_j}{\lambda_{j+1}}\right)^{\frac{D-2}{2}}
    \right)^{1/2},
\end{equation*}
where the infimum is taken over
\[
    \boldsymbol\iota=(\iota_1,\ldots,\iota_M)\in\{-1,1\}^M,
    \qquad
    \boldsymbol\lambda=(\lambda_1,\ldots,\lambda_M)\in(0,\infty)^M .
\]
\begin{lemma}[Static modulation lemma]
\label{lem:static-modulation}
Let \(D\ge4\) and \(M\in\mathbb N\). There exist \(\eta>0\) and \(C>0\) with the
following property. Let \(0<\theta<1\), and let \(\boldsymbol v\in\mathcal E\)
satisfy
\[
    d_M(\boldsymbol v)\le\eta,
    \qquad
    E(\boldsymbol v)\le M E(\boldsymbol W)+\theta^2 .
\]
Then there exist unique signs and scales 
\[
    \boldsymbol\iota=(\iota_1,\ldots,\iota_M)\in\{-1,1\}^M,\qquad \boldsymbol\lambda=(\lambda_1,\ldots,\lambda_M)\in(0,\infty)^M,
\]
and a remainder \(\boldsymbol g=(g,\dot g)\in\mathcal E\) such that
\begin{equation*}
    \boldsymbol v
    =
    \sum_{j=1}^M\iota_j\boldsymbol W_{\lambda_j}
    +
    \boldsymbol g,
    \qquad
    \langle Z_{\lambda_j},g\rangle=0,
    \quad 1\le j\le M.
\end{equation*}
Moreover,
\begin{equation*}
    d_M(\boldsymbol v)^2
    \le
    \|\boldsymbol g\|_{\mathcal E}^2
    +
    \sum_{j=1}^{M-1}
    \left(\frac{\lambda_j}{\lambda_{j+1}}\right)^{\frac{D-2}{2}}
    \le
    C d_M(\boldsymbol v)^2 .
\end{equation*}
If
\[
    a_j^\pm:=\langle \boldsymbol\alpha_{\lambda_j}^\pm,\boldsymbol g\rangle,
    \qquad
    S:=\{j\in\{1,\ldots,M-1\}:\iota_j=\iota_{j+1}\},
\]
then
\begin{equation*}
    \|\boldsymbol g\|_{\mathcal E}^2
    +
    \sum_{j\notin S}
    \left(\frac{\lambda_j}{\lambda_{j+1}}\right)^{\frac{D-2}{2}}
    \le
    C\max_{j\in S}
    \left(\frac{\lambda_j}{\lambda_{j+1}}\right)^{\frac{D-2}{2}}
    +
    C\max_{1\le i\le M,\ \pm}|a_i^\pm|^2
    +
    C\theta^2 .
\end{equation*}
\end{lemma}
Finally, we record the leading adjacent-bubble interaction.
\begin{lemma}[Adjacent-bubble interaction]
\label{lem:adjacent-bubble-interaction}
Let \(D\ge4\) and \(M\in\mathbb N\). For every \(\theta>0\), there exists
\(\eta>0\) with the following property. Let
\(\boldsymbol W(\boldsymbol\iota,\boldsymbol\lambda)\) be an \(M\)-bubble
configuration satisfying
\[
    \sum_{j=0}^{M}
    \left(\frac{\lambda_j}{\lambda_{j+1}}\right)^{\frac{D-2}{2}}
    \le\eta,
    \qquad
    \lambda_0:=0,\quad \lambda_{M+1}:=\infty .
\]
Define the nonlinear interaction error
\begin{equation*}
    \mathcal I(\boldsymbol\iota,\boldsymbol\lambda)
    :=
    f\left(\sum_{j=1}^M\iota_j W_{\lambda_j}\right)
    -
    \sum_{j=1}^M\iota_j f(W_{\lambda_j}).
\end{equation*}
Then, for \(1\le j\le M\),
\begin{equation*}
\begin{aligned}
&\left|
    \left\langle
        \Lambda W_{\lambda_j},
        \mathcal I(\boldsymbol\iota,\boldsymbol\lambda)
    \right\rangle
    -
    \iota_{j+1}
    \frac{D-2}{2D}(D(D-2))^{\frac D2}
    \left(\frac{\lambda_j}{\lambda_{j+1}}\right)^{\frac{D-2}{2}}
\right. \\
&\left.
    \qquad\qquad
    +
    \iota_{j-1}
    \frac{D-2}{2D}(D(D-2))^{\frac D2}
    \left(\frac{\lambda_{j-1}}{\lambda_j}\right)^{\frac{D-2}{2}}
\right|  \\
&\qquad\le
    \theta
    \left[
        \left(\frac{\lambda_{j-1}}{\lambda_j}\right)^{\frac{D-2}{2}}
        +
        \left(\frac{\lambda_j}{\lambda_{j+1}}\right)^{\frac{D-2}{2}}
    \right],
\end{aligned}
\end{equation*}
where the terms with \(j=0\) or \(j=M+1\) are understood to be zero.
\end{lemma}

	\section{Profile decomposition}\label{profile decomposition}
    \subsection{Linear profile decomposition}
We first introduce the notation for the energy-critical scaling. For
\(\lambda>0\), define
\[
    D_\lambda f(x)
    :=
    \lambda^{-\frac{D-2}{2}}f\left(\frac{x}{\lambda}\right),
    \qquad
    \dot D_\lambda g(x)
    :=
    \lambda^{-\frac D2}g\left(\frac{x}{\lambda}\right).
\]
Thus \(D_\lambda\) preserves the \(\dot H^1\)-norm and
\(\dot D_\lambda\) preserves the \(L^2\)-norm. We denote by \(S_0(t)\) the free
linear wave flow associated with
\[
    \partial_{tt}u-\Delta u=0.
\]
Throughout this section we assume \(D\geq5\).
\begin{prop}[Radial linear profile decomposition]
\label{LPF for DNLW}
Let  $\{(v_{0,n},v_{1,n})\}_{n\geq1}$
be a bounded sequence in
\(H^1_{\mathrm{rad}}(\mathbb R^D)\times L^2_{\mathrm{rad}}(\mathbb R^D)\).
After passing to a subsequence, there exist \(J_0\in\{0,1,\ldots,\infty\}\) and,
for each \(1\leq j<J_0\), a nonzero profile of one of the following two types.

\medskip
\noindent
\textup{(i) Damped profile.}
There is at most one profile of this type. It is given by
\[
    (\phi^j,\psi^j)
    \in
    H^1_{\mathrm{rad}}(\mathbb R^D)\times L^2_{\mathrm{rad}}(\mathbb R^D),
\]
and we set
\[
    (V_{0,n}^j,V_{1,n}^j):=(\phi^j,\psi^j).
\]
\medskip
\noindent
\textup{(ii) Small-scale wave profiles.}
There exist $  (\phi^j,\psi^j)
    \in
    \dot H^1_{\mathrm{rad}}(\mathbb R^D)\times L^2_{\mathrm{rad}}(\mathbb R^D)$, a sequence of scales $ \lambda_n^j\to0,$
and a sequence of rescaled times \(s_n^j\in[0,\infty)\) such that, after absorbing
finite limits into the profile, either
\[
    s_n^j\equiv0
    \qquad\text{or}\qquad
    s_n^j\to+\infty.
\]
Let \(U_L^j\) be the free linear wave with initial data
\[
    (U_L^j(0),\partial_tU_L^j(0))=(\phi^j,\psi^j).
\]
Fix \(0<\theta<1\), and set
\[
    P_n^j:=P_{>(\lambda_n^j)^\theta}.
\]
Then the corresponding profile at time \(t=0\) is defined by
\[
    V_{0,n}^j
    :=
    D_{\lambda_n^j}P_n^j U_L^j(-s_n^j),
    \qquad
    V_{1,n}^j
    :=
    \dot D_{\lambda_n^j}P_n^j \partial_tU_L^j(-s_n^j).
\]
\medskip
For every \(1\leq J<J_0\), we have the decomposition
\[
    v_{0,n}
    =
    \sum_{j=1}^J V_{0,n}^j+w_{0,n}^J,
    \qquad
    v_{1,n}
    =
    \sum_{j=1}^J V_{1,n}^j+w_{1,n}^J.
\]
The parameters of distinct small-scale wave profiles are asymptotically
orthogonal: if \(j\neq k\) are both of type \textup{(ii)}, then
\[
    \frac{\lambda_n^j}{\lambda_n^k}
    +
    \frac{\lambda_n^k}{\lambda_n^j}
    +
    \frac{|t_n^j-t_n^k|}{\lambda_n^j}
    \to+\infty,\qquad \text{where }t_n^j=\lambda_n^js_n^j. 
\]
Moreover, for every fixed \(J\), the energy norms decouple:
\[
\begin{aligned}
    \|v_{0,n}\|_{H^1}^2+\|v_{1,n}\|_{L^2}^2
    =
    \sum_{j=1}^J
    \left(
        \|V_{0,n}^j\|_{H^1}^2+\|V_{1,n}^j\|_{L^2}^2
    \right)
    +
    \|w_{0,n}^J\|_{H^1}^2+\|w_{1,n}^J\|_{L^2}^2
    +o_n(1).
\end{aligned}
\]
The remainders are asymptotically orthogonal to all previously extracted profiles.
If \(j\leq J\) is the damped profile, then
\[
    (w_{0,n}^J,w_{1,n}^J)
    \rightharpoonup0
    \qquad
    \text{in }H^1\times L^2.
\]
If \(j\leq J\) is a small-scale wave profile, then
\[
    S_0(s_n^j)
    \left(
        D_{\lambda_n^j}^{-1}w_{0,n}^J,\
        \dot D_{\lambda_n^j}^{-1}w_{1,n}^J
    \right)
    \rightharpoonup0
    \qquad
    \text{in }\dot H^1\times L^2.
\]
Finally, the damped linear evolution of the remainder vanishes in the critical
Strichartz norm:
\[
    \lim_{J\to J_0}
    \limsup_{n\to\infty}
    \left\|
        \pi_1S_\alpha(t)(w_{0,n}^J,w_{1,n}^J)
    \right\|_{S_D([0,\infty))}
    =0.
\]
\end{prop}
\begin{Rmk}
The lack of scaling invariance is responsible for the one-sided nature of
Proposition~\ref{LPF for DNLW}, namely the upper bound on the scales
\(\lambda_n^j\). Similar one-sided profile decompositions appear in other
scaling-broken problems, for instance nonlinear Schr\"odinger equations on curved
or product spaces \cite{IPS,IS,HP, HPTV, CGYZ, YZ}, equations with potentials \cite{LZ}, and
Klein--Gordon or Schr\"odinger equations with additional lower-order terms
\cite{Nak, KSV, CGM, TVZ, CML,CMY, MXZ}.
\end{Rmk}

The proof of Proposition~\ref{LPF for DNLW} follows the usual induction scheme
for profile decompositions. The point specific to the damped equation is the
extraction of the first profile. A nontrivial Strichartz norm for the linear
damped flow yields concentration at a dyadic frequency and at a forward time.
The damping excludes concentration at infinite physical time, while at small
scales the rescaled damping coefficient tends to zero and the limiting profile is
governed by the free wave equation. We first isolate this mechanism through a
refined Strichartz inequality and a weak concentration lemma.

\begin{lemma}[Refined Strichartz inequality]
\label{lem:refined-strichartz-damped}
Let \(I\subset[0,\infty)\) be a time interval. For any
\[
    (f,g)\in H^1(\mathbb R^D)\times L^2(\mathbb R^D),
\]
we have
\[
    \|\pi_1S_\alpha(t)(f,g)\|_{S_D(I)}
    \lesssim
    \|(f,g)\|_{H^1\times L^2}^{\frac{D-2}{D-1}}
    \left(
        \sup_{N\in 2^{\mathbb Z}}
        N^{-\frac{D-2}{2}}
        \|P_N\pi_1S_\alpha(t)(f,g)\|_{L^\infty_{t,x}(I\times\mathbb R^D)}
    \right)^{\frac1{D-1}} .
\]
\end{lemma}
\begin{proof}
Set $ u(t):=\pi_1S_\alpha(t)(f,g)$. We also write
\[
    p:=\frac{2(D+1)}{D-2},
    \qquad
    q:=\frac{2(D+1)}{D-1}.
\]
Thus \(S_D(I)=L^p_{t,x}(I\times\mathbb R^D)\). By the Littlewood--Paley square
function estimate and the standard dyadic expansion, we have
\[
    \|u\|_{L^p_{t,x}}^p
    \lesssim
    \sum_{N\geq M}
    \int_{I\times\mathbb R^D}
    |P_Nu|^q |P_Mu|^{p-q}\,dxdt .
\]
Using the \(L^\infty_{t,x}\)-norm for the lower frequency factor, we obtain
\[
\begin{aligned}
    \|u\|_{L^p_{t,x}}^p
    &\lesssim
    \sum_{N\geq M}
    \|P_Nu\|_{L^q_{t,x}}^q
    \|P_Mu\|_{L^\infty_{t,x}}^{p-q}.
\end{aligned}
\]
Let
\[
    A:=
    \sup_{L\in2^{\mathbb Z}}
    L^{-\frac{D-2}{2}}
    \|P_Lu\|_{L^\infty_{t,x}}.
\]
Then $ \|P_Mu\|_{L^\infty_{t,x}}
    \leq
    A M^{\frac{D-2}{2}},$
and hence
\[
\begin{aligned}
    \|u\|_{L^p_{t,x}}^p
    &\lesssim
    A^{p-q}
    \sum_{N\geq M}
    M^{\frac{D-2}{2}(p-q)}
    \|P_Nu\|_{L^q_{t,x}}^q .
\end{aligned}
\]
Since
\[
    \frac{D-2}{2}(p-q)
    =
    \frac{D+1}{D-1}
    =
    \frac q2,
\]
summing over \(M\leq N\) gives
\[
    \|u\|_{L^p_{t,x}}^p
    \lesssim
    A^{p-q}
    \sum_N
    N^{q/2}\|P_Nu\|_{L^q_{t,x}}^q .
\]
It remains to bound the last sum by the energy norm of the initial data. For the
pair \((q,q)\), the homogeneous Strichartz estimate has regularity $ \gamma=\frac12.$ Therefore, applying the homogeneous Strichartz estimate to the dyadic piece
\(P_Nu\), we get
\[
\begin{aligned}
    \|P_Nu\|_{L^q_{t,x}}
    &\lesssim
    \|\langle\nabla\rangle^{1/2}P_N f\|_{L^2}
    +
    \|\langle\nabla\rangle^{-1/2}P_N(\alpha f+g)\|_{L^2}.
\end{aligned}
\]
Multiplying by \(N^{1/2}\), and using that \(\alpha>0\) is fixed, we obtain
\[
    N^{1/2}\|P_Nu\|_{L^q_{t,x}}
    \lesssim
    \langle N\rangle \|P_Nf\|_{L^2}
    +
    \|P_Ng\|_{L^2}.
\]
Consequently, since \(q>2\),
\[
\begin{aligned}
    \sum_N
    N^{q/2}\|P_Nu\|_{L^q_{t,x}}^q
    &\lesssim
    \sum_N
    \left(
        \langle N\rangle \|P_Nf\|_{L^2}
        +
        \|P_Ng\|_{L^2}
    \right)^q
\\
    &\lesssim
    \left(
        \sum_N
        \langle N\rangle^2\|P_Nf\|_{L^2}^2
        +
        \sum_N
        \|P_Ng\|_{L^2}^2
    \right)^{q/2}
\\
    &\lesssim
    \|(f,g)\|_{H^1\times L^2}^q.
\end{aligned}
\]
Combining the above estimates yields
\[
    \|u\|_{L^p_{t,x}}^p
    \lesssim
    A^{p-q}
    \|(f,g)\|_{H^1\times L^2}^q.
\]
Finally, since
\[
    \frac q p=\frac{D-2}{D-1},
    \qquad
    \frac{p-q}{p}=\frac1{D-1},
\]
taking the \(p\)-th root gives
\[
    \|u\|_{S_D(I)}
    \lesssim
    \|(f,g)\|_{H^1\times L^2}^{\frac{D-2}{D-1}}
    A^{\frac1{D-1}},
\]
which is the desired estimate.
\end{proof}
The refined Strichartz inequality shows that a non-small critical norm forces a
large dyadic \(L^\infty_{t,x}\) component. The next lemma converts this dyadic
concentration into an actual weak profile. In doing so, we distinguish three
frequency regimes. The low-frequency regime is excluded by the \(H^1\)-bound,
the fixed-frequency regime gives a damped profile, and the high-frequency regime
gives a small-scale free wave profile.
\begin{lemma}[Weak concentration]
\label{lem:weak-concentration}
Let $\{(f_n,g_n)\}$ be a bounded sequence in
\(H^1_{\mathrm{rad}}(\mathbb R^D)\times L^2_{\mathrm{rad}}(\mathbb R^D)\). Assume
that
\[
    \|(f_n,g_n)\|_{H^1\times L^2}\leq A
\]
and
\[
    \limsup_{n\to\infty}
    \|\pi_1S_\alpha(t)(f_n,g_n)\|_{S_D([0,\infty))}
    \geq \varepsilon>0.
\]
Then, after passing to a subsequence, one of the following two alternatives holds.
\medskip
\noindent
\textup{(i) Fixed-scale damped concentration.}
There exists a nonzero pair
\[
    (\phi,\psi)\in H^1_{\mathrm{rad}}(\mathbb R^D)\times L^2_{\mathrm{rad}}(\mathbb R^D)
\]
such that, after absorbing a finite concentration time into the definition of the
profile,
\[
    (f_n,g_n)\rightharpoonup(\phi,\psi)
    \qquad
    \text{weakly in }H^1\times L^2.
\]
\medskip
\noindent
\textup{(ii) Small-scale wave concentration.}
There exist a sequence of scales $\lambda_n\to0,$
a sequence of rescaled times \(s_n\in[0,\infty)\), and a nonzero pair
\[
    (\phi,\psi)\in \dot H^1_{\mathrm{rad}}(\mathbb R^D)\times L^2_{\mathrm{rad}}(\mathbb R^D)
\]
such that
\[
    S_0(s_n)
    \left(
        D_{\lambda_n}^{-1}f_n,\,
        \dot D_{\lambda_n}^{-1}g_n
    \right)
    \rightharpoonup
    (\phi,\psi)
    \qquad
    \text{weakly in }\dot H^1\times L^2.
\]
Moreover, after passing to a further subsequence, either
\[
    s_n\equiv0
    \qquad\text{or}\qquad
    s_n\to+\infty.
\]
\end{lemma}
\begin{proof}
Set
\[
    u_n(t):=\pi_1S_\alpha(t)(f_n,g_n).
\]
By Lemma~\ref{lem:refined-strichartz-damped}, we have
\[
    \varepsilon
    \lesssim
    A^{\frac{D-2}{D-1}}
    \left(
        \sup_{N\in2^{\mathbb Z}}
        N^{-\frac{D-2}{2}}
        \|P_Nu_n\|_{L^\infty_{t,x}([0,\infty)\times\mathbb R^D)}
    \right)^{\frac1{D-1}} .
\]
Hence there exist dyadic numbers \(N_n\in2^{\mathbb Z}\), times \(t_n\geq0\), and
radii \(r_n\geq0\) such that
\begin{equation}\label{ineq:refine stri give}
     N_n^{-\frac{D-2}{2}}
    |P_{N_n}u_n(t_n,r_n)|
    \gtrsim
    c(\varepsilon,A)>0.  
\end{equation}
We first record two standard consequences of radiality and of the damped
propagator. Since \(u_n(t)\) is radial, the radial Bernstein estimate implies that
\[
    N^{-\frac{D-2}{2}}|P_N h(r)|
    \lesssim
    (1+Nr)^{-\frac{D-1}{2}}
    \|h\|_{\dot H^1}
\]
for radial \(h\). Applying this estimate to \(h=u_n(t_n)\), and using the
boundedness of the linear flow in the energy space, (\ref{ineq:refine stri give}) implies
\begin{equation}\label{ineq:N and r}
  N_n r_n \lesssim 1.   
\end{equation}
Thus the concentration point remains in a bounded region after rescaling by the
frequency \(N_n\).
We next exclude the low-frequency case \(N_n\to0\). Indeed, by Bernstein and the
energy bound,
\[
    N_n^{-\frac{D-2}{2}}\|P_{N_n}u_n(t_n)\|_{L^\infty}
    \lesssim
    N_n
    \|(u_n(t_n),\partial_tu_n(t_n))\|_{H^1\times L^2}
    \lesssim
    N_n A,
\]
which tends to \(0\) if \(N_n\to0\). This contradicts (\ref{ineq:refine stri give}). Hence, after passing
to a subsequence, either
\[
    N_n\sim1
    \qquad\text{or}\qquad
    N_n\to+\infty.
\]
Assume first that \(N_n\sim1\). We claim that \(t_n\) is bounded. Indeed, suppose
by contradiction that \(t_n\to+\infty\). We first prove the following fixed
frequency decay estimate. Let \(2<p<\infty\). Then there exist constants
\(\sigma=\sigma(D,p)>0\), \(C>0\), and \(c_\alpha>0\), depending on the fixed
damping coefficient \(\alpha\), such that for every dyadic \(N\sim1\),
\begin{equation}\label{ineq:fixed frequency decay}
    \|P_N\pi_1S_\alpha(t)(f,g)\|_{L^\infty}
    \leq
    C\langle t\rangle^{-\sigma}
    \|(f,g)\|_{H^1\times L^2}
    +
    Ce^{-c_\alpha t}\langle t\rangle^{\delta_p}
    \|(f,g)\|_{H^1\times L^2}. 
\end{equation}
To see this, recall that
\[
    \pi_1S_\alpha(t)(f,g)
    =
    \partial_t\mathcal K_\alpha(t)f
    +
    \mathcal K_\alpha(t)(\alpha f+g).
\]
By Bernstein's inequality, since \(N\sim1\),
\[
    \|P_N h\|_{L^\infty}
    \lesssim
    \|P_Nh\|_{L^p}.
\]
Applying Lemma \ref{Lp Lq decay} with \(q=2\), \(s_1=s_2=0\), we obtain
\[
\begin{aligned}
    \|P_N\mathcal K_\alpha(t)(\alpha f+g)\|_{L^p}
    \lesssim{}&
    \langle t\rangle^{-\frac D2(\frac12-\frac1p)}
    \|P_{\leq\alpha}P_N(\alpha f+g)\|_{L^2}
\\
    &+
    e^{-c_\alpha t}\langle t\rangle^{\delta_p}
    \|P_{>\alpha}P_N(\alpha f+g)\|_{H_p^{\beta-1}},
\end{aligned}
\]
and
\[
\begin{aligned}
    \|P_N\partial_t\mathcal K_\alpha(t)f\|_{L^p}
    \lesssim{}&
    \langle t\rangle^{-\frac D2(\frac12-\frac1p)-1}
    \|P_{\leq\alpha}P_N f\|_{L^2}
\\
    &+
    e^{-c_\alpha t}\langle t\rangle^{\delta_p}
    \|P_{>\alpha}P_N f\|_{H_p^{\beta-1}} .
\end{aligned}
\]
Since \(N\sim1\) and \(\alpha>0\) is fixed, all Sobolev norms appearing on the
right-hand side are controlled by $\|(f,g)\|_{H^1\times L^2}$.
Thus (\ref{ineq:fixed frequency decay}) follows, for instance with
\[
    \sigma=\frac D2\left(\frac12-\frac1p\right)>0.
\]
Applying (\ref{ineq:fixed frequency decay}) to \((f_n,g_n)\), and using
\[
    \|(f_n,g_n)\|_{H^1\times L^2}\leq A,
\]
we get
\[
    \|P_{N_n}u_n(t_n)\|_{L^\infty}
    =
    \|P_{N_n}\pi_1S_\alpha(t_n)(f_n,g_n)\|_{L^\infty}
    \to0,
\]
because \(N_n\sim1\) and \(t_n\to+\infty\). Since \(N_n^{-\frac{D-2}{2}}\sim1\),
this contradicts the concentration lower bound (\ref{ineq:refine stri give}). Therefore \(t_n\) must be bounded. Since the linear damped flow is continuous on \(H^1\times L^2\) on
finite time intervals, the sequence $S_\alpha(t_n)(f_n,g_n)$
is bounded in \(H^1\times L^2\). Passing to a subsequence, we may assume
\[
    S_\alpha(t_n)(f_n,g_n)
    \rightharpoonup
    (\Phi,\Psi)
    \qquad
    \text{weakly in }H^1\times L^2.
\]
The lower bound (\ref{ineq:refine stri give}), together with (\ref{ineq:N and r}), implies $(\Phi,\Psi)\neq(0,0)$.
Absorbing the finite time $t_n$ into the definition of the profile gives the fixed-scale damped concentration.

It remains to consider the high-frequency case \(N_n\to+\infty\). Set
\[
    \lambda_n:=N_n^{-1}\to0,
    \qquad
    s_n:=\frac{t_n}{\lambda_n}.
\]
We first show that \(t_n\) is bounded. For \(N\gg_\alpha1\), the explicit
high-frequency formula for the damped propagator and Bernstein's inequality give
\[
    N^{-\frac{D-2}{2}}
    \|P_N\pi_1S_\alpha(t)(f,g)\|_{L^\infty}
    \lesssim
    e^{-\frac{\alpha}{2}t}
    \|(f,g)\|_{H^1\times L^2}.
\]
If \(t_n\to+\infty\), this estimate contradicts the concentration lower bound.
Thus \(t_n\) is bounded. Passing to a subsequence, we may assume \(t_n\to T\in[0,\infty)\). Define
\[
    F_n:=D_{\lambda_n}^{-1}f_n,
    \qquad
    G_n:=\dot D_{\lambda_n}^{-1}g_n.
\]
Then \((F_n,G_n)\) is bounded in \(\dot H^1\times L^2\), and the scaling relation
gives
\[
    D_{\lambda_n}^{-1}u_n(\lambda_ns)
    =
    \pi_1S_{\alpha\lambda_n}(s)(F_n,G_n).
\]
In particular,
\[
    D_{\lambda_n}^{-1}u_n(t_n)
    =
    \pi_1S_{\alpha\lambda_n}(s_n)(F_n,G_n).
\]
On unit frequencies, since \(t_n\) is bounded and \(\lambda_n\to0\),
\[
    P_1\pi_1S_{\alpha\lambda_n}(s_n)(F_n,G_n)
    -
    e^{-\frac{\alpha t_n}{2}}
    P_1\pi_1S_0(s_n)(F_n,G_n)
    \to0
\]
after pairing with any fixed Schwartz test function. Since
\(e^{-\alpha t_n/2}\to e^{-\alpha T/2}>0\), the concentration lower bound yields
a nontrivial weak limit for the free wave shifted sequence. Hence, after passing
to a subsequence,
\[
    S_0(s_n)(F_n,G_n)
    \rightharpoonup
    (\phi,\psi)
    \qquad
    \text{weakly in }\dot H^1\times L^2,
\]
with \((\phi,\psi)\neq(0,0)\).
If \(s_n\to s_\infty<\infty\), we replace
\[
    (\phi,\psi)
    \quad\text{by}\quad
    S_0(-s_\infty)(\phi,\psi)
\]
and may assume \(s_n\equiv0\). Otherwise, after passing to a subsequence,
\(s_n\to+\infty\). This proves the small-scale wave concentration and completes
the proof.
\end{proof}
The preceding lemma only produces a weak limit in the natural profile coordinates.
We now turn this weak concentration into a genuine profile which can be subtracted
from the original sequence. In the small-scale case, the profile belongs naturally
to \(\dot H^1\times L^2\), and therefore we insert a low-frequency cut-off before
rescaling it back to the \(H^1\times L^2\) level. This also gives the
Pythagorean expansion of the energy and the weak orthogonality of the new
remainder.
\begin{lemma}[Extraction of one profile]
\label{lem:one-profile-extraction}
Let $\{(f_n,g_n)\}$ satisfy the condition in Lemma \ref{lem:weak-concentration}.
Then, after passing to a subsequence, there exists a sequence
\[
    (V_{0,n},V_{1,n})\in H^1_{\mathrm{rad}}(\mathbb R^D)\times
    L^2_{\mathrm{rad}}(\mathbb R^D)
\]
of one of the following two forms.

\medskip
\noindent
\textup{(i) Fixed-scale damped profile.}
There exists 
\[
    (\phi,\psi)\in H^1_{\mathrm{rad}}(\mathbb R^D)\times
    L^2_{\mathrm{rad}}(\mathbb R^D),
    \qquad
    (\phi,\psi)\neq(0,0),
\]
such that
\[
    (V_{0,n},V_{1,n})=(\phi,\psi).
\]
\medskip
\noindent
\textup{(ii) Small-scale wave profile.}
There exist $ \lambda_n\to0$, $s_n\in[0,\infty),$
and
\[
    (\phi,\psi)\in \dot H^1_{\mathrm{rad}}(\mathbb R^D)\times
    L^2_{\mathrm{rad}}(\mathbb R^D),
    \qquad
    (\phi,\psi)\neq(0,0),
\]
such that, if \(U_L\) is the free linear wave satisfying
\[
    (U_L(0),\partial_tU_L(0))=(\phi,\psi),
\]
then, for a fixed \(0<\theta<1\), setting $P_n:=P_{>\lambda_n^\theta}$, 
we define
\[
    V_{0,n}:=D_{\lambda_n}P_nU_L(-s_n),
    \qquad
    V_{1,n}:=\dot D_{\lambda_n}P_n\partial_tU_L(-s_n).
\]
Moreover, after passing to a further subsequence, either
\[
    s_n\equiv0
    \qquad\text{or}\qquad
    s_n\to+\infty.
\]
\medskip
In both cases, there exists \(c=c(\varepsilon,A)>0\) such that
\[
    \liminf_{n\to\infty}
    \left(
        \|V_{0,n}\|_{H^1}+\|V_{1,n}\|_{L^2}
    \right)
    \geq c.
\]
Let
\[
    (\widetilde f_n,\widetilde g_n)
    :=
    (f_n,g_n)-(V_{0,n},V_{1,n}).
\]
Then
\[
\begin{aligned}
    \|f_n\|_{H^1}^2+\|g_n\|_{L^2}^2
    =
    \|V_{0,n}\|_{H^1}^2+\|V_{1,n}\|_{L^2}^2
    +
    \|\widetilde f_n\|_{H^1}^2+\|\widetilde g_n\|_{L^2}^2
    +o_n(1).
\end{aligned}
\]
Moreover, the new remainder is orthogonal to the extracted profile in the
following sense. In case \textup{(i)},
\[
    (\widetilde f_n,\widetilde g_n)
    \rightharpoonup0
    \qquad
    \text{weakly in }H^1\times L^2.
\]
In case \textup{(ii)},
\[
    S_0(s_n)
    \left(
        D_{\lambda_n}^{-1}\widetilde f_n,\,
        \dot D_{\lambda_n}^{-1}\widetilde g_n
    \right)
    \rightharpoonup0
    \qquad
    \text{weakly in }\dot H^1\times L^2.
\]
\end{lemma}
\begin{proof}
By Lemma~\ref{lem:weak-concentration}, after passing to a subsequence, either a
fixed-scale damped concentration or a small-scale wave concentration occurs.
We first consider the fixed-scale case. After absorbing the finite concentration
time into the definition of the profile, we have
\[
    (f_n,g_n)\rightharpoonup(\phi,\psi)
    \qquad
    \text{weakly in }H^1\times L^2,
\]
where $(\phi,\psi)\neq(0,0)$.
Set
\[
    (V_{0,n},V_{1,n})=(\phi,\psi),
    \qquad
    (\widetilde f_n,\widetilde g_n)
    =
    (f_n,g_n)-(\phi,\psi).
\]
Then
\[
    (\widetilde f_n,\widetilde g_n)\rightharpoonup0
    \qquad
    \text{weakly in }H^1\times L^2.
\]
Consequently,
\[
\begin{aligned}
    \|f_n\|_{H^1}^2+\|g_n\|_{L^2}^2
    =
    \|\phi\|_{H^1}^2+\|\psi\|_{L^2}^2
    +
    \|\widetilde f_n\|_{H^1}^2+\|\widetilde g_n\|_{L^2}^2
    +o_n(1).
\end{aligned}
\]
The quantitative lower bound follows from the quantitative lower bound in
Lemma~\ref{lem:weak-concentration}. This proves the lemma in the fixed-scale case.
We now turn to the small-scale case. Lemma~\ref{lem:weak-concentration} gives
\[
    \lambda_n\to0,\qquad s_n\in[0,\infty),
\]
and
\begin{equation}\label{S0 samll wave weak convergence}
   S_0(s_n)
    \left(
        D_{\lambda_n}^{-1}f_n,\,
        \dot D_{\lambda_n}^{-1}g_n
    \right)
    \rightharpoonup
    (\phi,\psi)
    \qquad
    \text{weakly in }\dot H^1\times L^2,   
\end{equation}
where $(\phi,\psi)\neq(0,0)$.
Let \(U_L\) be the free linear wave with data \((\phi,\psi)\), and set
\[
    P_n:=P_{>\lambda_n^\theta},
    \qquad 0<\theta<1.
\]
We define
\[
    V_{0,n}:=D_{\lambda_n}P_nU_L(-s_n),
    \qquad
    V_{1,n}:=\dot D_{\lambda_n}P_n\partial_tU_L(-s_n).
\]
First, we check that the profile belongs to \(H^1\times L^2\). Since
\(P_n=P_{>\lambda_n^\theta}\), Bernstein gives
\[
    \|P_n h\|_{L^2}
    \lesssim
    \lambda_n^{-\theta}\|\nabla h\|_{L^2}.
\]
Therefore
\begin{equation}\label{DPU vanish}
   \begin{aligned}
    \|D_{\lambda_n}P_nU_L(-s_n)\|_{L^2}
    &=
    \lambda_n\|P_nU_L(-s_n)\|_{L^2}
\\
    &\lesssim
    \lambda_n^{1-\theta}
    \|\nabla U_L(-s_n)\|_{L^2}
    \to0.
\end{aligned} 
\end{equation}
Moreover, since \(P_n\to \mathrm{Id}\) strongly on
\(\dot H^1\times L^2\), and since the free wave flow is unitary on
\(\dot H^1\times L^2\), we have
\begin{equation}\label{V energy decouple}
    \begin{aligned}
    \|V_{0,n}\|_{H^1}^2+\|V_{1,n}\|_{L^2}^2
    &=
    \|\nabla P_nU_L(-s_n)\|_{L^2}^2
    +
    \|P_n\partial_tU_L(-s_n)\|_{L^2}^2
    +o_n(1)
\\
    &=
    \|P_n\phi\|_{\dot H^1}^2+\|P_n\psi\|_{L^2}^2+o_n(1)
\\
    &=
    \|\phi\|_{\dot H^1}^2+\|\psi\|_{L^2}^2+o_n(1).
\end{aligned}
\end{equation}
Here the \(o_n(1)\) in the first line comes from the \(L^2\)-part of the
\(H^1\)-norm, which is negligible by (\ref{DPU vanish}).
We next compute the cross term. Using the scaling invariance of
\(\dot H^1\times L^2\), the unitarity of \(S_0(t)\), and the commutation of
\(P_n\) with \(S_0(t)\), we get
\[
\begin{aligned}
&\left\langle
    (f_n,g_n),
    (V_{0,n},V_{1,n})
\right\rangle_{\dot H^1\times L^2}
\\
&\quad =
\left\langle
    \left(
        D_{\lambda_n}^{-1}f_n,\,
        \dot D_{\lambda_n}^{-1}g_n
    \right),
    \left(
        P_nU_L(-s_n),\,
        P_n\partial_tU_L(-s_n)
    \right)
\right\rangle_{\dot H^1\times L^2}
\\
&\quad =
\left\langle
    S_0(s_n)
    \left(
        D_{\lambda_n}^{-1}f_n,\,
        \dot D_{\lambda_n}^{-1}g_n
    \right),
    P_n(\phi,\psi)
\right\rangle_{\dot H^1\times L^2}.
\end{aligned}
\]
By (\ref{S0 samll wave weak convergence}) and the strong convergence
\[
    P_n(\phi,\psi)\to(\phi,\psi)
    \qquad
    \text{in }\dot H^1\times L^2,
\]
we obtain
\begin{equation}\label{strong converge of cross term}
    \left\langle
    (f_n,g_n),
    (V_{0,n},V_{1,n})
\right\rangle_{\dot H^1\times L^2}
\to
\|\phi\|_{\dot H^1}^2+\|\psi\|_{L^2}^2.
\end{equation}
The \(L^2\)-part of the \(H^1\)-inner product of the first components is
negligible. Indeed, by (\ref{DPU vanish}) and the boundedness of \(f_n\) in \(L^2\),
\begin{equation}\label{fn cross V}
  \langle f_n,V_{0,n}\rangle_{L^2}=o_n(1).   
\end{equation}
Combining (\ref{V energy decouple}), (\ref{strong converge of cross term}), and (\ref{fn cross V}), we find
\[
\begin{aligned}
    \left\langle
        (f_n,g_n),
        (V_{0,n},V_{1,n})
    \right\rangle_{H^1\times L^2}
    =
    \|V_{0,n}\|_{H^1}^2+\|V_{1,n}\|_{L^2}^2+o_n(1).
\end{aligned}
\]
Therefore, with the definition of $(\widetilde f_n,\widetilde g_n)$
we obtain the Pythagorean expansion
\[
\begin{aligned}
    \|f_n\|_{H^1}^2+\|g_n\|_{L^2}^2
    =
    \|V_{0,n}\|_{H^1}^2+\|V_{1,n}\|_{L^2}^2
    +
    \|\widetilde f_n\|_{H^1}^2+\|\widetilde g_n\|_{L^2}^2
    +o_n(1).
\end{aligned}
\]
It remains to prove the weak orthogonality of the new remainder. By definition,
\[
\begin{aligned}
&S_0(s_n)
    \left(
        D_{\lambda_n}^{-1}\widetilde f_n,\,
        \dot D_{\lambda_n}^{-1}\widetilde g_n
    \right)
\\
&\quad =
S_0(s_n)
    \left(
        D_{\lambda_n}^{-1}f_n,\,
        \dot D_{\lambda_n}^{-1}g_n
    \right)
-
S_0(s_n)
    \left(
        P_nU_L(-s_n),\,
        P_n\partial_tU_L(-s_n)
    \right)
\\
&\quad =
S_0(s_n)
    \left(
        D_{\lambda_n}^{-1}f_n,\,
        \dot D_{\lambda_n}^{-1}g_n
    \right)
-
P_n(\phi,\psi).
\end{aligned}
\]
The first term converges weakly to \((\phi,\psi)\) by (\ref{S0 samll wave weak convergence}), while the second
term converges strongly to \((\phi,\psi)\). Hence
\[
    S_0(s_n)
    \left(
        D_{\lambda_n}^{-1}\widetilde f_n,\,
        \dot D_{\lambda_n}^{-1}\widetilde g_n
    \right)
    \rightharpoonup0
    \qquad
    \text{weakly in }\dot H^1\times L^2.
\]
Finally, the quantitative lower bound
\[
    \liminf_{n\to\infty}
    \left(
        \|V_{0,n}\|_{H^1}+\|V_{1,n}\|_{L^2}
    \right)
    \geq c(\varepsilon,A)
\]
follows from (\ref{V energy decouple}) and the quantitative nontriviality of
\((\phi,\psi)\) obtained in Lemma~\ref{lem:weak-concentration}. This completes
the proof.
\end{proof}
With these preparatory lemmas, we now turn to the proof of Proposition \ref{LPF for DNLW}.
\begin{proof}[Proof of Proposition~\ref{LPF for DNLW}]
We argue by induction, following the standard profile decomposition scheme. Set
\[
    (w_{0,n}^0,w_{1,n}^0):=(v_{0,n},v_{1,n}).
\]
Suppose that, for some \(J\geq0\), profiles $ \{(V_{0,n}^j,V_{1,n}^j)\}_{1\leq j\leq J}$
and a remainder \((w_{0,n}^J,w_{1,n}^J)\) have been constructed so that
\[
    (v_{0,n},v_{1,n})
    =
    \sum_{j=1}^J (V_{0,n}^j,V_{1,n}^j)
    +
    (w_{0,n}^J,w_{1,n}^J),
\]
with the energy decoupling and the corresponding weak orthogonality relations
for all previously extracted profiles.
If
\[
    \limsup_{n\to\infty}
    \left\|
        \pi_1S_\alpha(t)(w_{0,n}^J,w_{1,n}^J)
    \right\|_{S_D([0,\infty))}
    =0,
\]
then the construction stops. Otherwise, by Lemma~\ref{lem:one-profile-extraction},
applied to the sequence \((w_{0,n}^J,w_{1,n}^J)\), after passing to a further
subsequence, we obtain a new profile
\[
    (V_{0,n}^{J+1},V_{1,n}^{J+1})
\]
of either the fixed-scale damped type or the small-scale wave type. We then define
\[
    (w_{0,n}^{J+1},w_{1,n}^{J+1})
    :=
    (w_{0,n}^J,w_{1,n}^J)
    -
    (V_{0,n}^{J+1},V_{1,n}^{J+1}).
\]
Lemma~\ref{lem:one-profile-extraction} gives
\[
\begin{aligned}
    \|w_{0,n}^J\|_{H^1}^2+\|w_{1,n}^J\|_{L^2}^2
    =
    \|V_{0,n}^{J+1}\|_{H^1}^2
    +
    \|V_{1,n}^{J+1}\|_{L^2}^2
    +
    \|w_{0,n}^{J+1}\|_{H^1}^2
    +
    \|w_{1,n}^{J+1}\|_{L^2}^2
    +o_n(1).
\end{aligned}
\]
Iterating this identity yields, for each fixed \(J\),
\[
\begin{aligned}
    \|v_{0,n}\|_{H^1}^2+\|v_{1,n}\|_{L^2}^2
    =
    \sum_{j=1}^J
    \left(
        \|V_{0,n}^j\|_{H^1}^2+\|V_{1,n}^j\|_{L^2}^2
    \right)
    +
    \|w_{0,n}^J\|_{H^1}^2+\|w_{1,n}^J\|_{L^2}^2
    +o_n(1).
\end{aligned}
\]
We next verify the orthogonality of the parameters. First, there can be at most
one fixed-scale damped profile. Indeed, after such a profile has been extracted,
the new remainder converges weakly to zero in \(H^1\times L^2\), and hence
Lemma~\ref{lem:weak-concentration} cannot produce another nontrivial fixed-scale
weak limit from this remainder.
Now consider two small-scale profiles indexed by \(j\neq k\), with parameters
\[
    \lambda_n^j,\ s_n^j
    \qquad\text{and}\qquad
    \lambda_n^k,\ s_n^k.
\]
If their parameters were not asymptotically orthogonal, then, after passing to a
subsequence, the two scales would be comparable and the rescaled time centers
would remain at bounded distance. In that case, the profile
\((V_{0,n}^k,V_{1,n}^k)\) would have a nonzero weak limit in the coordinate frame
of the \(j\)-th profile, namely after applying
\[
    S_0(s_n^j)
    \left(
        D_{\lambda_n^j}^{-1}\cdot,\,
        \dot D_{\lambda_n^j}^{-1}\cdot
    \right).
\]
This contradicts the weak orthogonality of the remainder obtained at the stage
when the \(j\)-th profile was extracted. Hence the small-scale parameters satisfy
\[
    \frac{\lambda_n^j}{\lambda_n^k}
    +
    \frac{\lambda_n^k}{\lambda_n^j}
    +
    \frac{|\lambda_n^j s_n^j-\lambda_n^k s_n^k|}{\lambda_n^j}
    \to+\infty.
\]
The same argument gives the weak orthogonality of the final remainder to each
previously extracted profile.
It remains to prove the vanishing of the linear evolution of the remainder. Let
\(J_0\) be the maximal number of profiles extracted by the above procedure. If
\(J_0<\infty\), then the construction stops precisely when
\[
    \limsup_{n\to\infty}
    \left\|
        \pi_1S_\alpha(t)(w_{0,n}^{J_0},w_{1,n}^{J_0})
    \right\|_{S_D([0,\infty))}
    =0.
\]
If \(J_0=\infty\), suppose by contradiction that the remainders do not vanish in
the critical Strichartz norm. Then there exists \(\varepsilon_*>0\) such that,
for infinitely many \(J\),
\[
    \limsup_{n\to\infty}
    \left\|
        \pi_1S_\alpha(t)(w_{0,n}^J,w_{1,n}^J)
    \right\|_{S_D([0,\infty))}
    \geq \varepsilon_* .
\]
Applying Lemma~\ref{lem:one-profile-extraction} at each such step yields profiles
whose \(H^1\times L^2\) norms are bounded from below by a positive constant
depending only on \(\varepsilon_*\) and on the original energy bound. This
contradicts the energy decoupling, since
\[
    \sum_{j=1}^\infty
    \limsup_{n\to\infty}
    \left(
        \|V_{0,n}^j\|_{H^1}^2+\|V_{1,n}^j\|_{L^2}^2
    \right)
    \leq
    \limsup_{n\to\infty}
    \left(
        \|v_{0,n}\|_{H^1}^2+\|v_{1,n}\|_{L^2}^2
    \right)
    <\infty.
\]
Therefore
\[
    \lim_{J\to J_0}
    \limsup_{n\to\infty}
    \left\|
        \pi_1S_\alpha(t)(w_{0,n}^J,w_{1,n}^J)
    \right\|_{S_D([0,\infty))}
    =0.
\]
This completes the proof.
\end{proof}

\subsection{Nonlinear profile decomposition} In this subsection, we establish the nonlinear profile decomposition for (\ref{DNLW}).  Let $\{(v_{0,n},v_{1,n})\}$
be a bounded sequence in
\(H^1_{\mathrm{rad}}(\mathbb R^D)\times L^2_{\mathrm{rad}}(\mathbb R^D)\), and let
\[
    (v_{0,n},v_{1,n})
    =
    \sum_{j=1}^J (V_{0,n}^j,V_{1,n}^j)
    +
    (w_{0,n}^J,w_{1,n}^J)
\]
be the linear profile decomposition given by Proposition~\ref{LPF for DNLW}.
For each linear profile, we define the corresponding nonlinear profile as follows.

\medskip
\noindent
\textup{(i) Fixed-scale damped profile.}
If \(j\) is the fixed-scale damped profile, namely
\[
    (V_{0,n}^j,V_{1,n}^j)=(\phi^j,\psi^j),
    \qquad
    (\phi^j,\psi^j)\in H^1_{\mathrm{rad}}\times L^2_{\mathrm{rad}},
\]
let \(U^j\) be the solution of
\[
    \partial_{tt}U^j-\Delta U^j+\alpha\partial_tU^j=f(U^j),
    \qquad
    f(u)=|u|^{\frac4{D-2}}u,
\]
with
\[
    (U^j(0),\partial_tU^j(0))=(\phi^j,\psi^j).
\]
Set
\[
    U_n^j(t,x):=U^j(t,x).
\]
\medskip
\noindent
\textup{(ii) Small-scale wave profile.}
Assume that \(j\) is a small-scale wave profile. Let \(U_L^j\) be the free
linear wave satisfying
\[
    (U_L^j(0),\partial_tU_L^j(0))=(\phi^j,\psi^j),
    \qquad
    (\phi^j,\psi^j)\in\dot H^1_{\mathrm{rad}}\times L^2_{\mathrm{rad}}.
\]
If \(s_n^j\equiv0\), let \(U^j\) be the nonlinear wave solution satisfying
\[
    \partial_{tt}U^j-\Delta U^j=f(U^j),
    \qquad
    (U^j(0),\partial_tU^j(0))=(\phi^j,\psi^j).
\]
If \(s_n^j\to+\infty\), let \(U^j\) be the nonlinear wave solution scattering to
\(U_L^j\) as \(t\to-\infty\), namely
\[
    \lim_{t\to-\infty}
    \left\|
        (U^j(t),\partial_tU^j(t))
        -
        (U_L^j(t),\partial_tU_L^j(t))
    \right\|_{\dot H^1\times L^2}
    =0.
\]
For a fixed \(0<\theta<1\), set
\[
    P_n^j:=P_{>(\lambda_n^j)^\theta}.
\]
The small-scale nonlinear profile contribution at the original scale is defined
by
\[
    U_n^j(t,x)
    :=
    (\lambda_n^j)^{-\frac{D-2}{2}}
    P_n^j U^j\left(
        \frac{t}{\lambda_n^j}-s_n^j,
        \frac{x}{\lambda_n^j}
    \right).
\]
Equivalently,
\[
    (U_n^j(0),\partial_tU_n^j(0))
    =
    \left(
        D_{\lambda_n^j}P_n^jU^j(-s_n^j),
        \dot D_{\lambda_n^j}P_n^j\partial_tU^j(-s_n^j)
    \right).
\]

\begin{prop}[Nonlinear profile decomposition]
\label{prop:nonlinear-profile}
Let \(I_n\subset[0,\infty)\) be a sequence of time intervals containing \(0\).
Assume that the nonlinear profiles exist on the corresponding time intervals in
the following sense.
For each small-scale wave profile \(j\), set
\[
    I_n^j:=(\lambda_n^j)^{-1}I_n-s_n^j .
\]
We assume that there exists an interval \(I^j\) such that
\[
    I_n^j\subset I^j
\]
for all sufficiently large \(n\), that \(U^j\) is defined on \(I^j\), and that
\[
    \|U^j\|_{X_D(I^j)}<\infty.
\]
For the fixed-scale damped profile, if it exists, we assume that \(U^j\) is
defined on \(I_n\) and
\[
    \limsup_{n\to\infty}\|U^j\|_{X_D(I_n)}<\infty.
\]
Define
\[
    u_n^J(t):=
    \sum_{j=1}^J U_n^j(t)
    +
    \pi_1S_\alpha(t)(w_{0,n}^J,w_{1,n}^J).
\]
Then, for every fixed \(J<J_0\),
\begin{equation}
    (u_n^J(0),\partial_tu_n^J(0))-(v_{0,n},v_{1,n})
    \to0
    \qquad
    \text{in }H^1\times L^2.
    \label{eq:nonlinear-profile-initial-error}
\end{equation}
Moreover, if
\[
    e_n^J:=\partial_{tt}u_n^J-\Delta u_n^J+\alpha\partial_tu_n^J-f(u_n^J),
\]
then
\begin{equation}
    \lim_{J\to J_0}\limsup_{n\to\infty}
    \|e_n^J\|_{N_D(I_n)}
    =0.
    \label{eq:nonlinear-profile-error}
\end{equation}
Consequently, if \(u_n\) is the exact solution to
\[
    \partial_{tt}u_n-\Delta u_n+\alpha\partial_tu_n=f(u_n),
    \qquad
    (u_n(0),\partial_tu_n(0))=(v_{0,n},v_{1,n}),
\]
then
\begin{equation}
\begin{aligned}
    \lim_{J\to J_0}\limsup_{n\to\infty}
    \Bigl[
        \|u_n-u_n^J\|_{X_D(I_n)}
        +
        \|(u_n-u_n^J,\partial_tu_n-\partial_tu_n^J)\|
            _{L_t^\infty(I_n;H^1\times L^2)}
    \Bigr]
    =0.
\end{aligned}
    \label{eq:nonlinear-profile-approximation}
\end{equation}
\end{prop}
\begin{proof}
\noindent
\textbf{Step 1. Initial data matching.}
We first prove
\[
    (u_n^J(0),\partial_tu_n^J(0))
    -
    (v_{0,n},v_{1,n})
    \to0
    \qquad
    \text{in }H^1\times L^2
\]
for every fixed \(J<J_0\).
By the linear profile decomposition,
\[
    (v_{0,n},v_{1,n})
    =
    \sum_{j=1}^J (V_{0,n}^j,V_{1,n}^j)
    +
    (w_{0,n}^J,w_{1,n}^J).
\]
Moreover,
\[
    \pi_1S_\alpha(0)(w_{0,n}^J,w_{1,n}^J)=w_{0,n}^J,
    \qquad
    \partial_t\pi_1S_\alpha(0)(w_{0,n}^J,w_{1,n}^J)=w_{1,n}^J.
\]
Therefore it suffices to prove, for each fixed profile \(j\), that
\[
    (U_n^j(0),\partial_tU_n^j(0))
    -
    (V_{0,n}^j,V_{1,n}^j)
    \to0
    \qquad
    \text{in }H^1\times L^2.
\]
If \(j\) is the fixed-scale damped profile, then
\[
    (V_{0,n}^j,V_{1,n}^j)=(\phi^j,\psi^j),\qquad(U^j(0),\partial_tU^j(0))=(\phi^j,\psi^j).
\]
Thus
\[
    (U_n^j(0),\partial_tU_n^j(0))
    =
    (V_{0,n}^j,V_{1,n}^j).
\]
Now suppose that \(j\) is a small-scale wave profile. Recall the definition of $P_n^j$
and
\[
    (V_{0,n}^j,V_{1,n}^j)
    =
    \left(
        D_{\lambda_n^j}P_n^jU_L^j(-s_n^j),
        \dot D_{\lambda_n^j}P_n^j\partial_tU_L^j(-s_n^j)
    \right),
\]
whereas
\[
    (U_n^j(0),\partial_tU_n^j(0))
    =
    \left(
        D_{\lambda_n^j}P_n^jU^j(-s_n^j),
        \dot D_{\lambda_n^j}P_n^j\partial_tU^j(-s_n^j)
    \right).
\]
If \(s_n^j\equiv0\), then
\[
    (U^j(0),\partial_tU^j(0))
    =
    (U_L^j(0),\partial_tU_L^j(0)),
\]
and the difference is identically zero. If \(s_n^j\to+\infty\), then by the
scattering condition as \(t\to-\infty\),
\[
    \left\|
        (U^j(-s_n^j),\partial_tU^j(-s_n^j))
        -
        (U_L^j(-s_n^j),\partial_tU_L^j(-s_n^j))
    \right\|_{\dot H^1\times L^2}
    \to0.
\]
Set
\[
    \eta_n^j:=U^j(-s_n^j)-U_L^j(-s_n^j),
    \qquad
    \zeta_n^j:=\partial_tU^j(-s_n^j)-\partial_tU_L^j(-s_n^j).
\]
Then
\[
    \|\eta_n^j\|_{\dot H^1}+\|\zeta_n^j\|_{L^2}\to0.
\]
Since \(P_n^j\) is bounded on \(\dot H^1\times L^2\), and since the
energy-critical scaling preserves \(\dot H^1\times L^2\), we obtain
\[
    \|\nabla D_{\lambda_n^j}P_n^j\eta_n^j\|_{L^2}
    +
    \|\dot D_{\lambda_n^j}P_n^j\zeta_n^j\|_{L^2}
    \leq
    \|\eta_n^j\|_{\dot H^1}+\|\zeta_n^j\|_{L^2}
    \to0.
\]
For the \(L^2\)-part of the first component, Bernstein gives
\[
    \|P_n^j h\|_{L^2}
    \lesssim
    (\lambda_n^j)^{-\theta}\|\nabla h\|_{L^2}.
\]
Therefore 
\[
\begin{aligned}
    \|D_{\lambda_n^j}P_n^j\eta_n^j\|_{L^2}
    =
    \lambda_n^j\|P_n^j\eta_n^j\|_{L^2}
\lesssim
    (\lambda_n^j)^{1-\theta}\|\eta_n^j\|_{\dot H^1}
    \to0.
\end{aligned}
\]
Thus
\[
    (U_n^j(0),\partial_tU_n^j(0))
    -
    (V_{0,n}^j,V_{1,n}^j)
    \to0
    \qquad
    \text{in }H^1\times L^2.
\]
Since \(J\) is fixed, summing over \(1\leq j\leq J\) proves the initial data
matching (\ref{eq:nonlinear-profile-initial-error}).
Having matched the initial data, we next show that the function \(u_n^J\) is an
approximate solution to the damped equation. The fixed-scale nonlinear profile is
an exact solution of the damped equation, and the linear remainder solves the
homogeneous damped equation. Hence the only error produced by a single profile
comes from the small-scale wave profiles. 

\medskip
\noindent
\textbf{Step 2. Error generated by one small-scale wave profile.}
Fix a small-scale profile \(j\). To simplify notation, write
\[
    \lambda_n=\lambda_n^j,\qquad
    s_n=s_n^j,\qquad
    P_n=P_n^j,\qquad
    U=U^j.
\]
Set $ \tau=\frac{t}{\lambda_n}-s_n,
    $ and $
    y=\frac{x}{\lambda_n}. $
Then
\[
    U_n^j(t,x)=\lambda_n^{-\frac{D-2}{2}}P_nU(\tau,y).
\]
Since \(U\) solves
\[
    \partial_{\tau\tau}U-\Delta_yU=f(U),
\]
and since \(P_n\) commutes with \(\partial_\tau\) and \(\Delta_y\), we have
\[
    \partial_{tt}U_n^j-\Delta U_n^j
    =
    \lambda_n^{-\frac{D+2}{2}}
    P_nf(U)(\tau,y).
\]
Moreover,
\[
    f(U_n^j)(t,x)
    =
    \lambda_n^{-\frac{D+2}{2}}
    f(P_nU)(\tau,y)\text{ , and }\quad \alpha\partial_tU_n^j(t,x)
    =
    \lambda_n^{-\frac{D+2}{2}}
    \alpha\lambda_n P_n\partial_\tau U(\tau,y).
\]
Therefore the damped equation error of the \(j\)-th small-scale profile is
\[
\begin{aligned}
    e_{n,j}
    &:=
    \partial_{tt}U_n^j-\Delta U_n^j+\alpha\partial_tU_n^j-f(U_n^j)
\\
    &=
    \lambda_n^{-\frac{D+2}{2}}
    \left[
        P_nf(U)-f(P_nU)
        +
        \alpha\lambda_n P_n\partial_\tau U
    \right]
    \left(
        \frac{t}{\lambda_n}-s_n,
        \frac{x}{\lambda_n}
    \right).
\end{aligned}
\]
Let $I_n^j:=\lambda_n^{-1}I_n-s_n$
be the corresponding time interval in the profile variables. By the critical
scaling of the forcing norm \(N_D\), it is enough to prove
\[
    \|P_nf(U)-f(P_nU)\|_{N_D(I_n^j)}
    +
    \alpha\lambda_n\|P_n\partial_tU\|_{N_D(I_n^j)}
    \to0.
\]
We first handle the truncation error. By the assumption above, \(I_n^j\subset I^j\) for all sufficiently large \(n\)
and \(\|U\|_{X_D(I^j)}<\infty\). Hence
\[
    P_nU\to U
    \qquad
    \text{in }X_D(I_n^j).
\]
Indeed, it is enough to prove the convergence on \(I^j\). If \(I^j\) is
unbounded, this follows from the absolute continuity of the \(X_D\)-norm and the
strong convergence of Littlewood--Paley cutoffs on compact time intervals. Using
\[
    P_nf(U)-f(P_nU)
    =
    P_n\bigl(f(U)-f(P_nU)\bigr)
    -
    P_{\leq\lambda_n^\theta}f(U),
\]
we estimate the first term by the nonlinear difference estimate in the
\(X_D\)-\(N_D\) framework:
\[
    \|P_n(f(U)-f(P_nU))\|_{N_D(I_n^j)}
    \lesssim
    C(\|U\|_{X_D(I_n^j)},\|P_nU\|_{X_D(I_n^j)})
    \|U-P_nU\|_{X_D(I_n^j)}.
\]
This tends to \(0\).  Similarly, since \(f(U)\in N_D(I^j)\), we have
\[
    P_{\leq(\lambda_n^j)^\theta}f(U)\to0
    \qquad
    \text{in }N_D(I_n^j).
\]
Therefore
\begin{equation}
   \|P_nf(U)-f(P_nU)\|_{N_D(I_n^j)}\to0.
    \label{eq:cutoff-error-small}  
\end{equation}
It remains to estimate the damping error. We use the following lower-order
estimate, which is a direct consequence of the definitions of the spaces
\(X_D,N_D\), Bernstein's inequality on the support of \(P_{>\lambda^\theta}\),
and the Strichartz estimates for the wave equation: there exist constants
\(\kappa>0\) and \(C>0\), depending only on \(D\) and \(\theta\), such that for
every interval \(J\),
\begin{equation*}
   \lambda\|P_{>\lambda^\theta}\partial_t U\|_{N_D(J)}
    \leq
    C\lambda^\kappa
    \left(
        \|(U,\partial_tU)\|_{L_t^\infty(J;\dot H^1\times L^2)}
        +
        \|U\|_{X_D(J)}
    \right).
    \label{eq:lower-order-damping-estimate}  
\end{equation*}
Applying this estimate with \(\lambda=\lambda_n\) and \(J=I_n^j\), and using the
assumed finite profile norm, we get
\begin{equation}
   \alpha\lambda_n\|P_n\partial_tU\|_{N_D(I_n^j)}
    \to0.
    \label{eq:damping-error-small}  
\end{equation}
Combining \eqref{eq:cutoff-error-small} and \eqref{eq:damping-error-small}, we
obtain
\[
    \|e_{n,j}\|_{N_D(I_n)}\to0.
\]
Thus every small-scale wave profile is, after truncation and rescaling, an
approximate solution to the original damped equation. This completes the analysis
of the only profile-level error which is not already built into the definition of
the nonlinear profiles.

We now pass from the error of each individual profile to the error of the full
approximate solution. Recall that
\[
    u_n^J
    =
    \sum_{j=1}^J U_n^j+r_n^J,
    \qquad
    r_n^J(t):=\pi_1S_\alpha(t)(w_{0,n}^J,w_{1,n}^J).
\]
The linear remainder \(r_n^J\) solves the homogeneous damped wave equation.
Moreover, every fixed-scale damped nonlinear profile solves the nonlinear damped
equation exactly, while Step 2 shows that each small-scale wave profile solves it
up to an error which tends to zero in \(N_D(I_n)\). Hence
\[
\begin{aligned}
    e_n^J
    &:=
    \partial_{tt}u_n^J-\Delta u_n^J+\alpha\partial_tu_n^J-f(u_n^J)
\\
    &=
    \sum_{j=1}^J f(U_n^j)
    -
    f\left(\sum_{j=1}^J U_n^j+r_n^J\right)
    +
    \sum_{\substack{1\leq j\leq J\\ j\ \mathrm{small\ scale}}} e_{n,j},
\end{aligned}
\]
where
\[
    \lim_{n\to\infty}\|e_{n,j}\|_{N_D(I_n)}=0
\]
for every fixed small-scale profile \(j\). Thus it remains to show that the
nonlinear interaction term is small.

\medskip
\noindent
\textbf{Step 3. Orthogonality and nonlinear interactions.}
We claim that, for every fixed \(J<J_0\),
\[
    \left\|
        f\left(\sum_{j=1}^J U_n^j\right)
        -
        \sum_{j=1}^J f(U_n^j)
    \right\|_{N_D(I_n)}
    \to0.
\]
This follows from the asymptotic orthogonality of the parameters. More precisely,
for \(j\neq k\), the profiles \(U_n^j\) and \(U_n^k\) are separated either in scale
or in the rescaled time variables. Hence all mixed products which appear in the
nonlinear estimates vanish. In the high-dimensional case \(D\geq6\), these mixed
terms are estimated in the same spaces as in the nonlinear estimates of \cite{InuiWakasugi2021}, namely in the \(X'(I_n)\) and \(W'(I_n)\) components of
\(N_D(I_n)\), using the \(X(I_n)\), \(Y(I_n)\), \(W(I_n)\), and \(S_1(I_n)\)
controls of the profiles. Thus, for every pair \(j\neq k\),
\[
    \|f(U_n^j+U_n^k)-f(U_n^j)-f(U_n^k)\|_{N_D(I_n)}
    \to0.
\]
Since \(J\) is fixed, summing over finitely many pairs gives
\[
    \left\|
        f\left(\sum_{j=1}^J U_n^j\right)
        -
        \sum_{j=1}^J f(U_n^j)
    \right\|_{N_D(I_n)}
    \to0.
\]
It remains to include the linear remainder \(r_n^J\). By the homogeneous
Strichartz estimates and the energy decoupling, the sequence \(r_n^J\) is bounded
in \(X_D(I_n)\) uniformly in \(n\), for each fixed \(J\). On the other hand, the
linear profile decomposition gives
\[
    \lim_{J\to J_0}
    \limsup_{n\to\infty}
    \|r_n^J\|_{S_D(I_n)}
    =
    0.
\]
Using the nonlinear difference estimate in the \(X_D\)-\(N_D\) framework, we
therefore obtain
\[
\begin{aligned}
\left\|
    f\left(\sum_{j=1}^J U_n^j+r_n^J\right)
    -
    f\left(\sum_{j=1}^J U_n^j\right)
\right\|_{N_D(I_n)}
\leq
    C\left(
        \sum_{j=1}^J\|U_n^j\|_{X_D(I_n)}
        +
        \|r_n^J\|_{X_D(I_n)}
    \right)
    o_{J}(1),
\end{aligned}
\]
where
\[
    o_J(1)\to0
    \qquad
    \text{as }J\to J_0,
\]
after taking \(\limsup_{n\to\infty}\). More explicitly, the small factor comes
from the critical Strichartz norm of \(r_n^J\), while the remaining factors are
bounded by the assumptions on the nonlinear profiles and by the linear estimates
for the remainder.
Combining the preceding estimates, (\ref{eq:nonlinear-profile-error}) is proved.

\medskip
\noindent
\textbf{Step 4. Perturbation argument.}
By Step 1,
\[
    (u_n^J(0),\partial_tu_n^J(0))
    -
    (v_{0,n},v_{1,n})
    \to0
    \qquad
    \text{in }H^1\times L^2
\]
for every fixed \(J\). Moreover, by the assumed profile bound, the homogeneous
Strichartz estimate for the linear remainder, and the energy decoupling, we have
\[
    \limsup_{n\to\infty}\|u_n^J\|_{X_D(I_n)}<\infty
\]
for every fixed \(J\). The equation error satisfies
\[
    \lim_{J\to J_0}
    \limsup_{n\to\infty}
    \|e_n^J\|_{N_D(I_n)}
    =0.
\]
Therefore Lemma~\ref{lem:long-time-perturbation} applies to the approximate
solution \(u_n^J\). It follows that the exact solution \(u_n\) with initial data
\[
    (u_n(0),\partial_tu_n(0))=(v_{0,n},v_{1,n})
\]
exists on \(I_n\), for \(n\) sufficiently large after \(J\) is fixed, and satisfies (\ref{eq:nonlinear-profile-approximation}).
This completes the nonlinear profile decomposition.

\end{proof}

	\section{Sequential soliton resolution}
The goal of this section is to prove Theorem~\ref{Sequential soliton resolution}.
We first prove a sequential compactness lemma, then extract the radiation term,
and finally combine these ingredients with the non-concentration of energy in the
self-similar region.
\subsection{Sequential compactness lemma}
Define the localized distance to the multi-bubble manifold by
\[
\boldsymbol{\delta}_R(\boldsymbol u)
:=
\inf_{M,\boldsymbol\iota,\boldsymbol\lambda}
\left(
    \left\|
        (u-\mathcal W(\boldsymbol\iota,\boldsymbol\lambda),\partial_tu)
    \right\|_{\mathcal E(r\le R)}^2
    +
    \sum_{j=1}^M
    \left(\frac{\lambda_j}{\lambda_{j+1}}\right)^{\frac{D-2}{2}}
\right)^{1/2},
\]
where the infimum is taken over
\(M\in\{0,1,2,\ldots\}\),
\(\boldsymbol\iota\in\{-1,1\}^M\), and
\(\boldsymbol\lambda\in(0,\infty)^M\). We use the convention
\(\lambda_{M+1}=R\).

We shall use two versions of the compactness argument, corresponding to bounded
and unbounded time intervals.

\begin{lemma}[Sequential compactness lemma]
\label{lem:sequential-compactness}
Let \(D\geq5\). Let \(\rho_n>0\), and let \(u_n\) be a sequence of
solutions to \((1.1)\) on the time intervals \([0,\rho_n]\) such that
\begin{equation*}
    \limsup_{n\to\infty}\sup_{t\in[0,\rho_n]}
    \|\boldsymbol{u}_n(t)\|_{H^1\times L^2}<\infty.
\end{equation*}
Assume that one of the following two alternatives holds.

\medskip
\noindent
\textup{Case I.} The sequence \(\{\rho_n\}\) is bounded, and there exists
\(R_n\to+\infty\) such that
\begin{equation}
    \lim_{n\to\infty}
    \frac1{\rho_n}
    \int_0^{\rho_n}\int_0^{\rho_nR_n}
    |\partial_tu_n(t,r)|^2r^{D-1}\,drdt=0.
    \label{eq:seq-comp-kinetic-finite}
\end{equation}
\medskip
\noindent
\textup{Case II.} We have \(\rho_n\to+\infty\), and there exists
\(R_n\to+\infty\) such that
\begin{equation}
    \lim_{n\to\infty}
    \int_0^{\rho_n}\int_0^{\rho_nR_n}
    |\partial_tu_n(t,r)|^2r^{D-1}\,drdt=0.
    \label{eq:seq-comp-kinetic-infinite}
\end{equation}
Then, after passing to a subsequence, there exist
\[
    t_n\in[0,\rho_n],
    \qquad
    r_n\leq R_n,
    \qquad
    r_n\to+\infty,
\]
such that
\[
    \lim_{n\to\infty}\delta_{\rho_nr_n}(u_n(t_n))=0.
\]
\end{lemma}

\begin{proof}
We first prove the lemma in Case II. The proof in Case I differs only in the
initial choice of the two endpoint times and will be explained at the end.

\medskip
\noindent
\textbf{Step 1. A localized virial identity between two good endpoint times.}
We first choose two times at which the localized kinetic energy is sufficiently
small. From \eqref{eq:seq-comp-kinetic-infinite}, the kinetic energy on the first
and the last third of \([0,\rho_n]\) tends to zero. Hence, choosing
\(R_{1,n}\to+\infty\), \(R_{1,n}\leq R_n\), sufficiently slowly, we can find
\[
    \sigma_n\in\left[0,\frac{\rho_n}{3}\right],
    \qquad
    \tau_n\in\left[\frac{2\rho_n}{3},\rho_n\right],
\]
such that
\begin{equation}
    \rho_nR_{1,n}
    \int_0^{\rho_nR_n}
    |\partial_tu_n(\sigma_n,r)|^2r^{D-1}\,dr
    \to0,
    \label{eq:sigma-kinetic-small}
\end{equation}
and
\begin{equation}
    \rho_nR_{1,n}
    \int_0^{\rho_nR_n}
    |\partial_tu_n(\tau_n,r)|^2r^{D-1}\,dr
    \to0.
    \label{eq:tau-kinetic-small}
\end{equation}
Recall the definition of $\underline{\Lambda}$
and define the localized scaling functional
\[
    \mathcal M_{n,R}(t):=
    \int_0^\infty
    \partial_tu_n(t,r)\underline{\Lambda}u_n(t,r)
    \chi_R(r)r^{D-1}\,dr,
    \qquad
    \chi_R(r):=\chi(r/R).
\]
Here \(\chi\) is a fixed radial cut-off satisfying \(\chi=1\) on
\(|x|\leq1/2\) and \(\chi=0\) on \(|x|\geq1\). By Cauchy--Schwarz, the uniform
energy bound, and \eqref{eq:sigma-kinetic-small}--\eqref{eq:tau-kinetic-small},
we have
\[
    \mathcal M_{n,\rho_nR_{1,n}}(\sigma_n)\to0,
    \qquad
    \mathcal M_{n,\rho_nR_{1,n}}(\tau_n)\to0.
\]
The localized Jia--Kenig virial identity (see \cite{JK 2017}) gives, for \(t\in[\sigma_n,\tau_n]\),
\[
\begin{aligned}
    \frac{d}{dt}\mathcal M_{n,\rho_nR_{1,n}}(t)
    =
    \int_0^\infty
    &\Bigg[
        \left(
            \partial_r^2u_n+\frac{D-1}{r}\partial_ru_n
            -|u_n|^{\frac4{D-2}}u_n
        \right)\underline{\Lambda}u_n
\\
    &\quad
        -\alpha\partial_tu_n\underline{\Lambda}u_n
        +\partial_tu_n
        \left(r\partial_r\partial_tu_n+\frac D2\partial_tu_n\right)
    \Bigg]
    \chi_{\rho_nR_{1,n}}r^{D-1}\,dr .
\end{aligned}
\]
Integrating this identity from \(\sigma_n\) to \(\tau_n\), and using the
vanishing of the endpoint functionals, we obtain
\begin{equation}
    \int_{\sigma_n}^{\tau_n}
    \mathcal I_n(t)\,dt=o_n(1),
    \label{eq:average-virial-small}
\end{equation}
where
\[
\begin{aligned}
    \mathcal I_n(t):=
    \int_0^\infty
    &\Bigg[
        \left(
            \partial_r^2u_n+\frac{D-1}{r}\partial_ru_n
            -|u_n|^{\frac4{D-2}}u_n
        \right)\underline{\Lambda}u_n
\\
    &\quad
        -\alpha\partial_tu_n\underline{\Lambda}u_n
        +\partial_tu_n
        \left(r\partial_r\partial_tu_n+\frac D2\partial_tu_n\right)
    \Bigg](t,r)
    \chi_{\rho_nR_{1,n}}r^{D-1}\,dr .
\end{aligned}
\]
\medskip
\noindent
\textbf{Step 2. Selection of a good time.}
We now choose a time at which the kinetic energy and the localized virial
functional have the desired pointwise properties. We use the elementary
selection Lemma~3.4 in \cite{JL}, applied to
\[
    f_n(t):=\int_0^{\rho_nR_n}|\partial_tu_n(t,r)|^2r^{D-1}\,dr
\]
and to the function
\[
    g_n(t):=-\mathcal I_n(t).
\]
The assumptions of that lemma follow from \eqref{eq:seq-comp-kinetic-infinite}
and \eqref{eq:average-virial-small}. Hence there exists
\(t_n\in[\sigma_n,\tau_n]\) such that
\begin{equation}
    \int_0^{\rho_nR_n}|\partial_tu_n(t_n,r)|^2r^{D-1}\,dr\to0,
    \label{eq:seq-good-time-kinetic}
\end{equation}
and, more generally,
\begin{equation}
    \lim_{n\to\infty}
    \sup_{I\subset[\sigma_n,\tau_n]}
    \frac1{|I|}
    \int_I\int_0^{\rho_nR_n}
    |\partial_tu_n(t,r)|^2r^{D-1}\,drdt=0.
    \label{eq:seq-maximal-kinetic}
\end{equation}
Moreover, for every sequence \(\widetilde R_n\leq R_{1,n}\) with
\(\widetilde R_n\to+\infty\), the localized virial functional satisfies
\begin{equation}
\begin{aligned}
0\geq
\limsup_{n\to\infty}
\Bigg(
-\int_0^\infty
    &\left[
        \left(
            \partial_r^2u_n+\frac{D-1}{r}\partial_ru_n
            -|u_n|^{\frac4{D-2}}u_n
        \right)\underline{\Lambda}u_n
        -\alpha\partial_tu_n\underline{\Lambda}u_n
\right.\\
&\left.
        +\partial_tu_n
        \left(r\partial_r\partial_tu_n+\frac D2\partial_tu_n\right)
    \right](t_n,r)
    \chi_{\rho_n\widetilde R_n}(r)r^{D-1}\,dr
\Bigg).
\end{aligned}
    \label{eq:seq-good-time-virial}
\end{equation}
The time \(t_n\) fixed in this step
will be the time at which the profile decomposition is applied.

\medskip
\noindent
\textbf{Step 3. Truncation and linear profile decomposition at the good time.}
The kinetic estimates obtained in Step~2 are localized in the region
\(r\leq \rho_nR_n\). We first pass to a truncated sequence for which these
estimates become global in space.
Choose sequences \(r_{2,n}\) and \(A_n\) such that
\[
    1\ll A_n\ll r_{2,n},\qquad
    A_nr_{2,n}\ll R_n,\qquad
    r_{2,n}\to+\infty.
\]
In Case II, since \(\rho_n\to+\infty\), this also gives
\[
    R_n^\ast:=\rho_nr_{2,n}\to+\infty.
\]
By the annular pigeonhole argument, see the proof of Lemma 3.1 from \cite{JL}, and the uniform energy bound, we may further
assume that
\begin{equation}
    \|\vec u_n(t_n)\|_{\mathcal E(A_n^{-1}\rho_nr_{2,n},\,A_n\rho_nr_{2,n})}
    \to0.
    \label{eq:step3-annulus-small}
\end{equation}
Set
\[
    R_n^\ast:=\rho_nr_{2,n}.
\]
Let \(\chi_n\) be a smooth radial cut-off such that
\[
    \chi_n(r)=1\quad\text{for }r\leq R_n^\ast,\qquad
    \chi_n(r)=0\quad\text{for }r\geq 2R_n^\ast,
\]
and
\[
    |\partial_r\chi_n|\lesssim (R_n^\ast)^{-1}.
\]
For all large \(n\), the transition region of \(\chi_n\) is contained in
\[
    A_n^{-1}R_n^\ast<r<A_nR_n^\ast.
\]
Define
\[
    \vec{\widetilde u}_n(t_n)
    :=
    \bigl(
        \chi_nu_n(t_n),\chi_n\partial_tu_n(t_n)
    \bigr),
\]
and let \(\widetilde u_n\) be the damped wave solution with this initial data at
time \(t_n\).
By \eqref{eq:step3-annulus-small} and Hardy's inequality, the cut-off error is
\(o_n(1)\) in the energy norm. Hence
\begin{equation*}
    \vec{\widetilde u}_n(t_n)=\vec u_n(t_n)
    \qquad\text{for }r\leq R_n^\ast,
\end{equation*}
and
\begin{equation*}
    \|\vec{\widetilde u}_n(t_n)\|_{\mathcal E(r\geq R_n^\ast)}
    =o_n(1).
\end{equation*}
We next pass the localized kinetic estimates to \(\widetilde u_n\). By finite
speed of propagation, \(\widetilde u_n=u_n\) in the cone
\[
    |t-t_n|+r<R_n^\ast.
\]
The region where the two solutions may differ is contained in the domain of
dependence of the annulus
\[
    A_n^{-1}R_n^\ast<r<A_nR_n^\ast.
\]
Since
\[
    A_n\ll r_{2,n},\qquad A_nr_{2,n}\ll R_n,
\]
this domain of dependence remains inside \(r\leq\rho_nR_n\) for
\(t\in[\sigma_n,\tau_n]\). Its contribution is \(o_n(1)\) by
\eqref{eq:step3-annulus-small} and the local energy estimate. Therefore the
kinetic estimates from Step~2 imply
\begin{equation}
    \|\partial_t\widetilde u_n(t_n)\|_{L^2(\mathbb R^D)}\to0,
    \label{eq:step3-global-velocity-small}
\end{equation}
and
\begin{equation}
    \lim_{n\to\infty}
    \sup_{I\subset[\sigma_n,\tau_n]}
    \frac1{|I|}
    \int_I\int_0^\infty
    |\partial_t\widetilde u_n(t,r)|^2r^{D-1}\,drdt
    =0.
    \label{eq:step3-global-maximal-kinetic}
\end{equation}
We now apply Proposition~\ref{LPF for DNLW} to the bounded sequence
\[
    \vec{\widetilde u}_n(t_n)\in H^1_{\rm rad}\times L^2_{\rm rad}.
\]
After passing to a subsequence, for every fixed \(J\) we have
\begin{equation}
    \vec{\widetilde u}_n(t_n)
    =
    \sum_{j=1}^J(\widetilde V_{0,n}^j,\widetilde V_{1,n}^j)
    +
    (\widetilde w_{0,n}^J,\widetilde w_{1,n}^J).
    \label{eq:step3-truncated-profile-decomp}
\end{equation}
The orthogonality, energy decoupling, and remainder smallness are those of
Proposition~\ref{LPF for DNLW}. In particular,
\[
    \lim_{J\to J_0}\limsup_{n\to\infty}
    \left\|
        \pi_1S_\alpha(t)(\widetilde w_{0,n}^J,\widetilde w_{1,n}^J)
    \right\|_{S_D([0,\infty))}
    =0.
\]
Although the small-scale limiting profiles are homogeneous wave profiles in
\(\dot H^1\times L^2\), the decomposition is applied to a bounded sequence in
\(H^1\times L^2\); the \(H^1\)-realization of each small-scale profile is the
one constructed in Proposition~\ref{LPF for DNLW}, with the high-frequency
cut-off \(P_n^j=P_{>(\lambda_n^j)^\theta}\).
Compared with the free wave decomposition used in \cite{JL}, the damped profile
decomposition has two simplifications: the physical-time escaping branch is
absent by the decay of the damped flow, and the large-scale branch is absent
because the scaling is one-sided. Hence the only possible nontrivial profiles in
\eqref{eq:step3-truncated-profile-decomp} are the fixed-scale damped profile and
small-scale wave profiles with
\[
    \lambda_n^j\to0,\qquad
    s_n^j=0
    \quad\text{or}\quad
    s_n^j\to+\infty.
\]
\textbf{Step 4. Exclusion of small-scale profiles with escaping wave time.}
We now rule out the small-scale profiles in
\eqref{eq:step3-truncated-profile-decomp} for which \(s_n^j\to+\infty\).
Suppose, toward a contradiction, that such a nonzero profile exists. Since the
physical-time escaping branch is absent in Proposition~\ref{LPF for DNLW},
after passing to a subsequence we may assume
\begin{equation*}
    \lambda_n^js_n^j\to T_j\in[0,\infty).
\end{equation*}
We shall prove that
\[
    (\phi^j,\psi^j)=(0,0),
\]
contradicting the nontriviality of the profile.
We will use the following claim.
\begin{claim}
\label{claim:profile-stationarity}
Let \(\widetilde u_n\) be a sequence of radial solutions to the damped equation
on intervals containing \(t_n\), with uniformly bounded energy. Assume that
\begin{equation}
    \lim_{n\to\infty}
    \sup_{I\subset[\sigma_n,\tau_n]}
    \frac1{|I|}
    \int_I
    \|\partial_t\widetilde u_n(t)\|_{L^2}^2\,dt
    =0.
    \label{eq:claim-global-kinetic-small}
\end{equation}
Let a nonlinear profile in the damped profile decomposition of
\(\vec{\widetilde u}_n(t_n)\) be considered on a compact profile-time interval
\(K\), and assume that the corresponding physical time windows are contained in
\([\sigma_n,\tau_n]\) for all large \(n\). Then the associated nonlinear profile
is stationary on \(K\). 
\end{claim}
An analogous result was established for wave equation, see \cite{DKM 2012}. 
Before proving the claim, we finish Step~4. The estimate
\eqref{eq:step3-global-maximal-kinetic} gives
\eqref{eq:claim-global-kinetic-small}. Moreover, since
\(\lambda_n^js_n^j\to T_j<\infty\), and \(t_n\) was chosen in Step~2 away from
the endpoints of \([\sigma_n,\tau_n]\) by a distance tending to infinity, all
compact profile windows are contained in \([\sigma_n,\tau_n]\) for large \(n\).
Claim~\ref{claim:profile-stationarity} then implies that the nonlinear wave
profile \(U^j\) associated with the escaping small-scale profile is stationary.
But by the definition of the nonlinear profile in the case \(s_n^j\to+\infty\),
\(U^j\) scatters backward to the free wave \(U_L^j\). A nonzero stationary
finite-energy solution cannot scatter on a half-line. Hence \(U^j\equiv0\), and
therefore
\[
    (\phi^j,\psi^j)=(0,0),
\]
contradicting the nontriviality of the profile. Thus no nonzero small-scale
profile with \(s_n^j\to+\infty\) occurs.
\begin{proof}[Proof of Claim~\ref{claim:profile-stationarity}]
We prove the claim first for a small-scale wave profile. By translating time,
we may assume \(t_n=0\). Let \(U^j\) be the associated nonlinear wave profile.
For a compact interval \(K\) contained in its lifespan, set
\[
    U_n^j(s,y)
    :=
    (\lambda_n^j)^{\frac{D-2}{2}}
    \widetilde u_n\bigl(\lambda_n^j(s_n^j+s),\lambda_n^jy\bigr),
    \qquad s\in K.
\]
Here \(s_n^j=0\) in the centered case, while \(s_n^j\to+\infty\) in the escaping
case. Then
\[
    \partial_{ss}U_n^j-\Delta_yU_n^j
    +\alpha\lambda_n^j\partial_sU_n^j
    =
    |U_n^j|^{\frac4{D-2}}U_n^j.
\]
Since \(\alpha\lambda_n^j\to0\), the small-scale nonlinear profile
approximation in Proposition~\ref{prop:nonlinear-profile} gives
\begin{equation}
    \partial_sU_n^j\rightharpoonup \partial_sU^j
    \qquad
    \text{weakly in }L^2(K\times\mathbb R^D).
    \label{eq:claim-small-weak}
\end{equation}
Indeed, in the variables of the \(j\)-th profile, the corresponding nonlinear
profile gives the limit \(U^j\), all other profiles vanish weakly by parameter
orthogonality, and the linear remainder has no weak profile. The damping error is
\[
    \alpha\lambda_n^j\partial_sU_n^j,
\]
which is \(o_n(1)\) in the perturbative norm, as in the small-scale part of
Proposition~\ref{prop:nonlinear-profile}.
By the assumption of the claim, the physical time interval $\lambda_n^j(s_n^j+K)$
is contained in \([\sigma_n,\tau_n]\) for all large \(n\). Hence
\eqref{eq:claim-global-kinetic-small} gives
\begin{equation}
\begin{aligned}
    \frac1{|K|}
    \int_K
    \|\partial_sU_n^j(s)\|_{L^2}^2\,ds
    =
    \frac1{\lambda_n^j|K|}
    \int_{\lambda_n^j(s_n^j+K)}
    \|\partial_t\widetilde u_n(t)\|_{L^2}^2\,dt
  \to0.
\end{aligned}
    \label{eq:claim-small-kinetic}
\end{equation}
Combining \eqref{eq:claim-small-weak} and \eqref{eq:claim-small-kinetic}, we
obtain
\[
    \partial_sU^j=0
    \qquad
    \text{in }L^2(K\times\mathbb R^D).
\]
Thus \(U^j\) is stationary on \(K\).
The fixed-scale damped profile is treated in the same way, without rescaling. If
\(U^0\) denotes the fixed-scale nonlinear damped profile, then the fixed-scale
part of Proposition~\ref{prop:nonlinear-profile} gives
\[
    \partial_t\widetilde u_n(t_n+t)
    \rightharpoonup
    \partial_tU^0(t)
    \qquad
    \text{weakly in }L^2(K\times\mathbb R^D),
\]
for every compact \(K\) contained in the lifespan of \(U^0\). By
\eqref{eq:claim-global-kinetic-small},
\[
    \frac1{|K|}
    \int_K
    \|\partial_t\widetilde u_n(t_n+t)\|_{L^2}^2\,dt
    \to0.
\]
Therefore \(\partial_tU^0=0\) on \(K\). This proves the claim. \end{proof}

\textbf{Step 5. Identification of the centered profiles.}
By Step~4, all nonzero small-scale profiles in
\eqref{eq:step3-truncated-profile-decomp} are centered:
\[
    \lambda_n^j\to0,\qquad s_n^j=0.
\]
The fixed-scale damped profile, if present, is also centered in the sense that no
time translation is involved. Applying Claim~\ref{claim:profile-stationarity} to
each remaining nonlinear profile, we obtain that all of them are stationary.
Hence, 
we may relabel the remaining nonzero profiles so that, for every fixed \(J\),
\begin{equation*}
    \vec{\widetilde u}_n(t_n)
    =
    \sum_{j=1}^{J}
    \bigl(
        \iota_j W_{\lambda_{j,n}},0
    \bigr)
    +
    (\widetilde w_{0,n}^J,\widetilde w_{1,n}^J)
    +
    o_n(1)
    \qquad\text{in }H^1\times L^2.
\end{equation*}
Here the scales are asymptotically orthogonal:
\[
    \frac{\lambda_{j,n}}{\lambda_{k,n}}
    +
    \frac{\lambda_{k,n}}{\lambda_{j,n}}
    \to+\infty,
    \qquad j\neq k.
\]
For the small-scale profiles this follows from the scale orthogonality in
Proposition~\ref{LPF for DNLW}; the fixed-scale bubble, if present, is
included by taking \(\lambda_{j,n}\equiv \mu_j\).
In the small-scale case, the high-frequency cut-off in
Proposition~\ref{LPF for DNLW} does not affect the bubble:
\[
    D_{\lambda_n^j}P_n^jQ^j
    =
    D_{\lambda_n^j}Q^j+o_{H^1}(1),
    \qquad
    \dot D_{\lambda_n^j}P_n^j\partial_tU^j(0)=0,
\]
since \(P_n^j=P_{>(\lambda_n^j)^\theta}\to I\) on the profile \(Q^j\) and
\(Q^j\in H^1\). 

\textbf{Step 6. Vanishing of the remaining linear part.}
Let
\[
    \vec{\widetilde u}_n(t_n)
    =
    \sum_{j=1}^{J}
    \bigl(\iota_jW_{\lambda_{j,n}},0\bigr)
    +
    (\widetilde w_{0,n},\widetilde w_{1,n})
    +
    o_n(1)
    \quad\text{in }H^1\times L^2
\]
be the decomposition obtained in Step~5, after all nonzero bubbles have been
extracted. We choose \(r_n\to+\infty\) such that
\[
    r_n\leq r_{2,n},\qquad r_n\leq R_{1,n},
\]
and, after passing to a subsequence,
\begin{equation}
    \|\vec u_n(t_n)\|_{\mathcal E(A_n^{-1}\rho_nr_n,A_n\rho_nr_n)}
    \to0
    \label{eq:step6-annulus-small}
\end{equation}
for some \(1\ll A_n\ll r_n\),
\begin{equation}
    \rho_nr_n
    \int_0^{\rho_nR_n}|\partial_tu_n(t_n,r)|^2r^{D-1}\,dr
    \to0,
    \label{eq:step6-weighted-kinetic}
\end{equation}
and
\begin{equation}
    \frac{\lambda_{j,n}}{\rho_nr_n}\to0
    \qquad
    \text{for every }1\leq j\leq J.
    \label{eq:step6-radius-dominates-bubbles}
\end{equation}
This follows from the same annular pigeonhole argument as in Step~3, the kinetic
smallness at the good time, and the fact that there are only finitely many
nonzero bubbles.
Set
\[
    K_R(v):=
    \int_0^\infty
    \left(
        |\partial_rv(r)|^2-|v(r)|^{\frac{2D}{D-2}}
    \right)
    \chi_R(r)r^{D-1}\,dr.
\]
Applying the localized virial inequality from Step~2 with
\(\widetilde R_n=r_n\), and integrating by parts as in Step~8 of
\cite{JL}, we obtain
\begin{equation}
    \limsup_{n\to\infty}
    \left[
        K_{\rho_nr_n}(u_n(t_n))
        -
        \alpha\int_0^\infty
        \partial_tu_n(t_n)\underline{\Lambda} u_n(t_n)
        \chi_{\rho_nr_n}r^{D-1}\,dr
    \right]\leq0.
    \label{eq:step6-virial-after-ibp}
\end{equation}
The boundary terms produced by the integration by parts vanish by
\eqref{eq:step6-annulus-small}. The damping term in
\eqref{eq:step6-virial-after-ibp} also vanishes. Indeed, by Cauchy--Schwarz, the
uniform energy bound, and \eqref{eq:step6-weighted-kinetic},
\[
\left|
    \int_0^\infty
    \partial_tu_n(t_n)\underline{\Lambda} u_n(t_n)
    \chi_{\rho_nr_n}r^{D-1}\,dr
\right|
\lesssim
\left(
    \rho_nr_n
    \int_0^{\rho_nR_n}|\partial_tu_n(t_n,r)|^2r^{D-1}\,dr
\right)^{1/2}
\to0.
\]
Hence
\begin{equation}
    \limsup_{n\to\infty}K_{\rho_nr_n}(u_n(t_n))\leq0.
    \label{eq:step6-K-nonpositive}
\end{equation}
Since \(r_n\leq r_{2,n}\), we have
\[
    \widetilde u_n(t_n,r)=u_n(t_n,r)
    \qquad\text{for }r\leq\rho_nr_n.
\]
Using \eqref{eq:step6-radius-dominates-bubbles}, the scale orthogonality, and
the identity $K(W)=0,$
we get
\[
    K_{\rho_nr_n}
    \left(
        \sum_{j=1}^{J}\iota_jW_{\lambda_{j,n}}
    \right)
    \to0,
\]
and the cross terms between distinct bubbles vanish. Therefore
\eqref{eq:step6-K-nonpositive} and the decomposition in Step~5 imply
\begin{equation}
    \limsup_{n\to\infty}
    K_{\rho_nr_n}(\widetilde w_{0,n})\leq0.
    \label{eq:step6-K-remainder}
\end{equation}
The final remainder has no nonzero profile left; by the refined Sobolev estimate
used in the proof of Proposition~\ref{LPF for DNLW}, 
\[
    \|\widetilde w_{0,n}\|_{L^{\frac{2D}{D-2}}}\to0.
\]
Thus \eqref{eq:step6-K-remainder} yields
\[
    \int_0^{\rho_nr_n}|\partial_r\widetilde w_{0,n}(r)|^2r^{D-1}\,dr\to0.
\]
Moreover, from \eqref{eq:step3-global-velocity-small} and the fact that all
profiles in Step~5 have zero velocity,
\[
    \int_0^{\rho_nr_n}|\widetilde w_{1,n}(r)|^2r^{D-1}\,dr\to0.
\]
Consequently,
\[
    \|(\widetilde w_{0,n},\widetilde w_{1,n})\|_{\mathcal E(r\leq\rho_nr_n)}
    \to0.
\]
Combining this with the bubble decomposition in Step~5, the scale orthogonality,
and the identity \(u_n(t_n)=\widetilde u_n(t_n)\) on
\(r\leq\rho_nr_n\), we conclude that
\[
    \delta_{\rho_nr_n}(u_n(t_n))\to0.
\]
This completes the proof in Case II.

For Case I, the only modification is in Steps 1 and 2. From \eqref{eq:seq-comp-kinetic-finite}, after
choosing \(R_{1,n}\to+\infty\) sufficiently slowly, one can choose
\[
    \sigma_n\in[0,\rho_n/3],
    \qquad
    \tau_n\in[2\rho_n/3,\rho_n],
\]
so that the endpoint kinetic energies satisfy the same estimates as in
\eqref{eq:average-virial-small}. The localized virial identity and the selection lemma then give
the analogues of \eqref{eq:seq-good-time-kinetic},
\eqref{eq:seq-maximal-kinetic}, and \eqref{eq:seq-good-time-virial}. Once these
three conclusions are obtained, Steps 3--6 are unchanged.
\end{proof}

	\subsection{Extraction of the radiation}
We now extract the radiation term. We first consider the finite-time case.
\begin{prop}[Radiation in the finite-time blow-up case]
\label{radiation in finite-time case}
Let \(\boldsymbol u(t)\in\mathcal E\) be a solution to \eqref{DNLW} on
\([0,T)\), \(T<\infty\), satisfying the type-II bound \eqref{Type 2}. Then there
exists \(\boldsymbol u_0^*\in\mathcal E\) such that
\[
    \boldsymbol u(t)\rightharpoonup \boldsymbol u_0^*
    \quad\text{weakly in }\mathcal E
    \qquad\text{as }t\to T,
\]
and, for every \(\phi\in C_0^\infty(0,\infty)\),
\[
    \left\|
        \phi(\boldsymbol u(t)-\boldsymbol u_0^*)
    \right\|_{\mathcal E}\to0
    \qquad\text{as }t\to T.
\]
Moreover, let \(\boldsymbol u^*(t)\) be the solution of \eqref{DNLW} with
terminal data
\[
    \boldsymbol u^*(T)=\boldsymbol u_0^* .
\]
Then \(\boldsymbol u^*(t)\) is defined on \([T-T_0,T]\) for some \(T_0>0\), and
\[
    \boldsymbol u(t,r)=\boldsymbol u^*(t,r)
    \qquad
    \text{for } r\ge T-t,\quad t\in[T-T_0,T).
\]
Finally,
\[
    \lim_{t\to T}E(\boldsymbol u(t)-\boldsymbol u^*(t))
    =
    \lim_{t\to T}E(\boldsymbol u(t))-E(\boldsymbol u_0^*).
\]
\end{prop}
    \begin{proof}
The proof is the same as the radiation extraction in
\cite[Theorem~3.2]{DKM 2011}, and we recall why it applies to the
damped equation.
The argument in \cite{DKM 2011} uses only the type-II bound in the
energy space, weak compactness, local well-posedness, finite speed of
propagation, and the decoupling of the energy. All these ingredients remain
valid for equation (\ref{DNLW}).
Indeed, the damping term is lower order for the local theory and does not affect
finite speed of propagation. Moreover, the energy identity
\[
    \frac{d}{dt}E(u(t))=-\alpha\|\partial_tu(t)\|_{L^2}^2
\]
gives the existence of the energy limit as \(t\to T\).
Thus, by the same compactness argument as in \cite[Theorem~3.2]{DKM 2011},
there exists $\boldsymbol u^{\,*}\in\mathcal E$
such that
\[
    \boldsymbol u(t)\rightharpoonup \boldsymbol u^{\,*}
    \qquad\text{weakly in }\mathcal E
    \quad\text{as }t\to T.
\]
Furthermore, for every \(\phi\in C_0^\infty(0,\infty)\),
\[
    \phi\bigl(\boldsymbol u(t)-\boldsymbol u^{\,*}\bigr)\to0
    \qquad\text{strongly in }\mathcal E .
\]
Let \(\boldsymbol u^*(t)\) be the solution of the damped equation with terminal
data \(\boldsymbol u^*(T)=\boldsymbol u^{\,*}\).
By local well-posedness, \(u^*\) is defined on \([T-T_0,T]\) for some
\(T_0>0\). The strong convergence away from the origin and finite speed of
propagation imply
\[
    u(t,r)=u^*(t,r)
    \qquad
    \text{for } r\geq T-t,\quad t\in[T-T_0,T).
\]
Finally, the energy decoupling follows from the weak convergence in the quadratic
part and the standard decoupling of the nonlinear potential:
\[
    \lim_{t\to T}E\bigl(u(t)-u^*(t)\bigr)
    =
    \lim_{t\to T}E(u(t))-E(u^*(T)).
\]
This proves the proposition.
\end{proof}

 In contrast with the finite-time case, the damping eliminates the global
radiation: the exterior energy of a global type-II solution tends to zero as
\(t\to+\infty\).
\begin{prop}[Vanishing of the radiation term in the global case]
\label{global case radiation}
Let \(\boldsymbol u(t)\in\mathcal E\) be a solution to \eqref{DNLW} on
\([T,\infty)\), \(T\ge0\), satisfying the type-II bound \eqref{Type 2}. Then, for
every \(R>0\),
\[
    \lim_{t\to\infty}
    \int_{t-R}^{\infty}
    \left[
        |\partial_tu(t,r)|^2
        +
        |\partial_ru(t,r)|^2
        +
        \frac{|u(t,r)|^2}{r^2}
    \right]r^{D-1}\,dr
    =0.
\]
\end{prop}
\begin{proof}
By the energy identity and the type-II bound,
\[
    \int_T^\infty \|\partial_tu(t)\|_{L^2}^2\,dt<\infty .
\]
Let $ h(t):=\|\partial_tu(t)\|_{L^2}^2 $.
By the one-sided Hardy--Littlewood maximal inequality, we can choose a sequence
\(s_n\to+\infty\) such that
\begin{equation*}
    \|\partial_tu(s_n)\|_{L^2}\to0,
\end{equation*}
and
\begin{equation}
    \sup_{\lambda>0}\frac1\lambda
    \int_{s_n}^{s_n+\lambda}
    \|\partial_tu(t)\|_{L^2}^2\,dt
    \to0.
    \label{eq:global-good-time-maximal}
\end{equation}
We apply Proposition~\ref{LPF for DNLW} to the bounded sequence
\[
    \vec u(s_n)\in H^1_{\rm rad}\times L^2_{\rm rad}.
\]
Using the profile-stationarity argument of Claim~\ref{claim:profile-stationarity},
with \([\sigma_n,\tau_n]\) replaced by \([s_n,\infty)\), every nonzero profile is
stationary. Indeed, the proof of that claim only uses the averaged kinetic
smallness on the corresponding profile windows, which is supplied here by
\eqref{eq:global-good-time-maximal}. Thus, after discarding zero profiles and
relabeling, we obtain
\begin{equation*}
    \boldsymbol u(s_n)
    =
    \sum_{j=1}^{J}
    \bigl(
        \iota_jW_{\lambda_{j,n}},0
    \bigr)
    +
    (w_{0,n},w_{1,n})
    +
    o_n(1)
    \quad\text{in }H^1\times L^2,
\end{equation*}
where \(J<\infty\), \(\iota_j\in\{-1,1\}\), and the scales are asymptotically
orthogonal. Moreover,
\begin{equation}
    \left\|
        \pi_1S_\alpha(t)(w_{0,n},w_{1,n})
    \right\|_{S_D([0,\infty))}
    \to0.
    \label{eq:global-linear-rem-small}
\end{equation}
Since all scales \(\lambda_{j,n}\) are either fixed or tend to \(0\), we have
\begin{equation}
    \frac{\lambda_{j,n}}{s_n}\to0
    \qquad
    \text{for every }1\leq j\leq J.
    \label{eq:global-bubbles-inside-cone}
\end{equation}
Set
\[
    B_n(x):=\sum_{j=1}^{J}\iota_jW_{\lambda_{j,n}}(x),
\]
and let $\boldsymbol z_n(t):=S_\alpha(t)(w_{0,n},w_{1,n})$, $ z_n:=\pi_1\boldsymbol z_n$.
Define the nonlinear error
\[
    r_n(t,x):=u(s_n+t,x)-B_n(x)-z_n(t,x).
\]
Then 
\begin{equation}
    \boldsymbol r_n(0)=o_n(1)
    \quad\text{in }H^1\times L^2.
    \label{eq:global-r-initial-small} 
\end{equation}
Furthermore \(r_n\) solves
\[
    \partial_{tt}r_n-\Delta r_n+\alpha\partial_tr_n
    =
    \mathcal N_n,
\]
where
\[
    \mathcal N_n
    =
    f(B_n+z_n+r_n)-\sum_{j=1}^{J}f(\iota_jW_{\lambda_{j,n}}).
\]
We claim that for every fixed \(R>0\),
\begin{equation}
    \sup_{t\ge0}
    \|\boldsymbol r_n(t)\|_{\mathcal E(|x|>s_n+t-R)}
    \to0.
    \label{eq:global-exterior-r-small}
\end{equation}
Let $\Omega_{n,R}
    :=
    \{(t,x): t\ge0,\ |x|>s_n+t-R\}.$
For \(T_1>0\), denote
\[
    \Omega_{n,R}^{T_1}:=\Omega_{n,R}\cap([0,T_1]\times\mathbb R^D).
\]
By the finite speed of propagation and the Strichartz estimates for the damped
linear flow,
\begin{equation}
\begin{aligned}
    &\|r_n\|_{X_D(\Omega_{n,R}^{T_1})}
    +
    \sup_{0\le t\le T_1}
    \|\boldsymbol r_n(t)\|_{\mathcal E(|x|>s_n+t-R)}
    \\
    &\hspace{2cm}
    \lesssim
    \|\boldsymbol r_n(0)\|_{H^1\times L^2}
    +
    \|\mathcal N_n\|_{N_D(\Omega_{n,R}^{T_1})}.
\end{aligned}
    \label{eq:global-exterior-strichartz}
\end{equation}
We decompose
\[
\begin{aligned}
    \mathcal N_n
    &=
    \bigl[
        f(B_n+z_n+r_n)-f(B_n+z_n)
    \bigr]
   +
    \bigl[
        f(B_n+z_n)-f(B_n)
    \bigr]
    \\
    &\quad+
    \biggl[
        f(B_n)-\sum_{j=1}^{J}f(\iota_jW_{\lambda_{j,n}})
    \biggr].
\end{aligned}
    \label{eq:global-nonlinear-decomp}
\]
We first record the smallness of the terms independent of \(r_n\). From
\eqref{eq:global-bubbles-inside-cone}, the decay of \(W\), and the scale
orthogonality, we have
\begin{equation}
    \|B_n\|_{X_D(\Omega_{n,R})}\to0,
    \qquad
    \left\|
        f(B_n)-\sum_{j=1}^{J}f(\iota_jW_{\lambda_{j,n}})
    \right\|_{N_D(\Omega_{n,R})}\to0.
    \label{eq:global-bubble-tail-small}
\end{equation}
Moreover, by \eqref{eq:global-linear-rem-small} and the nonlinear estimates in
the \(X_D\)-\(N_D\) framework,
\begin{equation}
    \|f(B_n+z_n)-f(B_n)\|_{N_D(\Omega_{n,R})}\to0.
    \label{eq:global-z-interaction-small}
\end{equation}
Finally, the same nonlinear difference estimate gives, uniformly in \(T_1\),
\begin{equation}
\begin{aligned}
    &
    \|f(B_n+z_n+r_n)-f(B_n+z_n)\|_{N_D(\Omega_{n,R}^{T_1})}
    \\
    &\qquad
    \le
    o_n(1)\,
    \|r_n\|_{X_D(\Omega_{n,R}^{T_1})}
    +
    C\|r_n\|_{X_D(\Omega_{n,R}^{T_1})}^{1+\frac4{D-2}}.
\end{aligned}
    \label{eq:global-r-perturbative}
\end{equation}
Combining
\eqref{eq:global-exterior-strichartz}--\eqref{eq:global-r-perturbative}, and
using \eqref{eq:global-r-initial-small}, we obtain
\[
\begin{aligned}
    &\|r_n\|_{X_D(\Omega_{n,R}^{T_1})}
    +
    \sup_{0\le t\le T_1}
    \|\vec r_n(t)\|_{\mathcal E(|x|>s_n+t-R)}
    \\
    &\qquad
    \le
    o_n(1)
    +
    o_n(1)\,
    \|r_n\|_{X_D(\Omega_{n,R}^{T_1})}
    +
    C\|r_n\|_{X_D(\Omega_{n,R}^{T_1})}^{1+\frac4{D-2}} .
\end{aligned}
\]
A standard continuity argument, independent of \(T_1\), yields
\[
    \|r_n\|_{X_D(\Omega_{n,R})}
    +
    \sup_{t\ge0}
    \|\boldsymbol r_n(t)\|_{\mathcal E(|x|>s_n+t-R)}
    \to0.
\]
This proves \eqref{eq:global-exterior-r-small}.
We now finish the proof of the proposition. Fix \(R>0\) and \(\eta>0\). Choose
\(n\) sufficiently large so that
\[
    \sup_{t\ge0}
    \|\boldsymbol r_n(t)\|_{\mathcal E(|x|>s_n+t-R)}
    <\eta.
\]
For \(t\ge0\), using
\[
    u(s_n+t)=B_n+z_n(t)+r_n(t),
\]
we estimate the exterior energy in the region \(|x|>s_n+t-R\). By
\eqref{eq:global-bubbles-inside-cone}, for this fixed \(n\),
\[
    \lim_{t\to\infty}
    \|B_n\|_{\mathcal E(|x|>s_n+t-R)}=0.
\]
Also, since \(z_n\) solves the homogeneous damped wave equation,
\[
    \lim_{t\to\infty}
    \|\boldsymbol z_n(t)\|_{H^1\times L^2}=0.
\]
Therefore
\[
    \limsup_{t\to\infty}
    \|\boldsymbol u(s_n+t)\|_{\mathcal E(|x|>s_n+t-R)}
    \lesssim \eta.
\]
Since \(\eta>0\) is arbitrary, and \(s_n+t\to+\infty\), we conclude that
\[
    \lim_{t\to\infty}
    \int_{t-R}^{\infty}
    \left[
        |\partial_tu(t,r)|^2
        +
        |\partial_ru(t,r)|^2
        +
        \frac{|u(t,r)|^2}{r^2}
    \right]r^{D-1}\,dr
    =0.
\]
This proves the proposition.
\end{proof}

\subsection{The sequential decomposition}\label{The sequential decomposition}
We first record the non-concentration of energy in the self-similar region.  The
corresponding statements for the undamped equation are proved by localized
energy identities and finite speed of propagation.  These arguments are stable
under the addition of the damping term: in the finite-time case the damping
contribution is a lower-order error on a shrinking time interval, while in the
global case it is controlled by the dissipation identity
\[
    \int_T^\infty \|\partial_tu(t)\|_{L^2}^2\,dt<\infty .
\]
Thus the proofs from \cite{CT 1993,JK 2017,JL} apply with only these harmless
modifications.
\begin{prop}[No self-similar concentration for finite-time blow-up solutions]
\label{prop:no-self-similar-finite}
Let \(u(t)\in\mathcal E\) be a solution to \eqref{DNLW} defined on
\([0,T)\), \(T<\infty\), and satisfying the type-II bound. Then, for any
\(\lambda\in(0,1)\),
\[
    \lim_{t\to T^-}
    \int_{\lambda(T-t)}^{T-t}
    \left[
        |\partial_tu(t,r)|^2
        +
        |\partial_ru(t,r)|^2
        +
        \frac{|u(t,r)|^2}{r^2}
    \right]r^{D-1}\,dr
    =0.
\]
\end{prop}
\begin{prop}[No self-similar concentration for global solutions]
\label{prop:no-self-similar-global}
Let \(u(t)\in\mathcal E\) be a solution to \eqref{DNLW} defined on
\([T,\infty)\), \(T\geq0\), and satisfying the type-II bound. Then, for any
\(\gamma\in(0,1)\),
\[
    \lim_{R\to\infty}\limsup_{t\to\infty}
    \int_{\gamma t}^{t-R}
    \left[
        |\partial_tu(t,r)|^2
        +
        |\partial_ru(t,r)|^2
        +
        \frac{|u(t,r)|^2}{r^2}
    \right]r^{D-1}\,dr
    =0.
\]
\end{prop}
\begin{Rmk}
For completeness, we recall the only point in which the damped equation differs
from the undamped one.  In the localized multiplier identities used in the
self-similar region, the additional term is always of the form
\[
    \alpha\int \partial_tu\,\mathcal M[u],
\]
where \(\mathcal M[u]\) is a localized first-order expression controlled by the
energy norm.  In the finite-time case, Cauchy--Schwarz, Hardy's inequality, and
the type-II bound show that this term is \(o(1)\) on the shrinking interval
\([t,T)\).  In the global case, the same estimate is controlled by the
integrability of \(\|\partial_tu(t)\|_{L^2}^2\).  Hence the localized flux
arguments proving the two propositions for the wave equation carry over without
any further change.
\end{Rmk}
The self-similar non-concentration will be used through the following averaged
virial consequence inside the backward light cone.

\begin{prop}[Time-averaged vanishing of kinetic energy]
\label{prop:avg-kinetic-finite}
Let \(u(t)\in\mathcal E\) be a type-II solution to \eqref{DNLW} on
\([0,T)\), \(T<\infty\). Then
\begin{equation}
    \lim_{t\to T^-}
    \frac1{T-t}
    \int_t^T
    \int_{|x|\leq T-s}
    |\partial_tu(s,x)|^2\,dxds
    =0,
    \label{eq:avg-kinetic-finite}
\end{equation}
and
\begin{equation}
    \lim_{t\to T^-}
    \frac1{T-t}
    \int_t^T
    \int_{|x|\leq T-s}
    \left(
        |u(s,x)|^{\frac{2D}{D-2}}
        -
        |\nabla u(s,x)|^2
    \right)\,dxds
    =0.
    \label{eq:avg-K-finite}
\end{equation}
\end{prop}
\begin{proof}
Let $R(t):=T-t$,
and choose a radial cut-off \(\phi\in C_0^\infty(B_1)\) such that
\(\phi=1\) on \(B_{1/2}\). Set
\[
    \phi_t(x):=\phi\left(\frac{x}{R(t)}\right),
    \qquad
    \Lambda u:=x\cdot\nabla u+\frac{D-2}{2}u.
\]
We first prove the averaged kinetic estimate. Define
\[
    \mathcal M(t)
    :=
    e^{\alpha t}
    \int_{\mathbb R^D}
    \partial_tu(t,x)\Lambda u(t,x)\phi_t(x)\,dx .
\]
A direct computation using the equation of $u$
and the critical scaling identity gives
\begin{equation}
    \frac{d}{dt}\mathcal M(t)
    =
    -e^{\alpha t}
    \int_{\mathbb R^D}
    |\partial_tu(t,x)|^2\phi_t(x)\,dx
    +
    e^{\alpha t}\operatorname{Err}_1(t),
    \label{eq:avg-kinetic-identity}
\end{equation}
where \(\operatorname{Err}_1(t)\) is supported in the self-similar annulus
\[
    \frac{R(t)}2\leq |x|\leq R(t)
\]
and satisfies
\begin{equation}
    |\operatorname{Err}_1(t)|
    \lesssim
    \int_{\frac{R(t)}2\leq |x|\leq R(t)}
    \left(
        |\partial_tu|^2+|\nabla u|^2+\frac{|u|^2}{|x|^2}
        +|u|^{\frac{2D}{D-2}}
    \right)(t,x)\,dx.
    \label{eq:avg-error1-bound}
\end{equation}
The factor \(e^{\alpha t}\) is inserted exactly to cancel the damping term
\(-\alpha\partial_tu\) in the derivative of the localized scaling functional.
We next record the endpoint estimate
\begin{equation}
    \frac{|\mathcal M(t)|}{R(t)}\to0
    \qquad\text{as }t\to T^-.
    \label{eq:avg-M-endpoint}
\end{equation}
Indeed, for any \(\lambda\in(0,1/2)\), split the integral defining
\(\mathcal M(t)\) into the regions
\[
    |x|\leq \lambda R(t),
    \qquad
    \lambda R(t)\leq |x|\leq R(t).
\]
On the first region, Cauchy--Schwarz, Hardy's inequality, and the type-II bound
give
\[
    \frac1{R(t)}
    \int_{|x|\leq \lambda R(t)}
    |\partial_tu|\,\bigl(|x||\nabla u|+|u|\bigr)\,dx
    \lesssim
    \lambda.
\]
On the second region, the same estimate together with
Proposition~\ref{prop:no-self-similar-finite} gives a term tending to \(0\) as
\(t\to T^-\), for every fixed \(\lambda\). Letting then \(\lambda\to0\) proves
\eqref{eq:avg-M-endpoint}.
Integrating \eqref{eq:avg-kinetic-identity} from \(t\) to \(T\), dividing by
\(R(t)\), and using \eqref{eq:avg-M-endpoint},
\eqref{eq:avg-error1-bound}, and Proposition~\ref{prop:no-self-similar-finite},
we obtain
\[
    \lim_{t\to T^-}
    \frac1{R(t)}
    \int_t^T
    \int_{\mathbb R^D}
    |\partial_tu(s,x)|^2\phi_s(x)\,dxds
    =0.
\]
Since \(\phi_s=1\) on \(|x|\leq R(s)/2\), and the remaining annulus
\(R(s)/2\leq |x|\leq R(s)\) is controlled by
Proposition~\ref{prop:no-self-similar-finite}, this proves
\eqref{eq:avg-kinetic-finite}.
It remains to prove \eqref{eq:avg-K-finite}. Define
\[
    \mathcal P(t)
    :=
    e^{\alpha t}
    \int_{\mathbb R^D}
    \partial_tu(t,x)u(t,x)\phi_t(x)\,dx .
\]
Another direct computation gives
\begin{equation}
\begin{aligned}
    \frac{d}{dt}\mathcal P(t)
    &=
    e^{\alpha t}
    \int_{\mathbb R^D}
    \left(
        |\partial_tu|^2
        -
        |\nabla u|^2
        +
        |u|^{\frac{2D}{D-2}}
    \right)(t,x)\phi_t(x)\,dx
    +
    e^{\alpha t}\operatorname{Err}_2(t),
\end{aligned}
    \label{eq:avg-K-identity}
\end{equation}
where \(\operatorname{Err}_2(t)\) is supported in $ \frac{R(t)}2\leq |x|\leq R(t)$
and satisfies the same type of bound as \(\operatorname{Err}_1(t)\):
\[
    |\operatorname{Err}_2(t)|
    \lesssim
    \int_{\frac{R(t)}2\leq |x|\leq R(t)}
    \left(
        |\partial_tu|^2+|\nabla u|^2+\frac{|u|^2}{|x|^2}
        +|u|^{\frac{2D}{D-2}}
    \right)(t,x)\,dx.
\]
Moreover,
\[
    \frac{|\mathcal P(t)|}{R(t)}\to0
    \qquad\text{as }t\to T^-,
\]
by the same inner-region and self-similar-annulus estimate used for
\(\mathcal M(t)\).
Integrating \eqref{eq:avg-K-identity} from \(t\) to \(T\), dividing by
\(R(t)\), and using the already proved kinetic estimate
\eqref{eq:avg-kinetic-finite}, the endpoint estimate for \(\mathcal P\), and
Proposition~\ref{prop:no-self-similar-finite}, we get
\[
    \lim_{t\to T^-}
    \frac1{R(t)}
    \int_t^T
    \int_{\mathbb R^D}
    \left(
        |u(s,x)|^{\frac{2D}{D-2}}
        -
        |\nabla u(s,x)|^2
    \right)\phi_s(x)\,dxds
    =0.
\]
Finally, the part of the cone not covered by \(\phi_s\) is again contained in
the self-similar annulus and is controlled by
Proposition~\ref{prop:no-self-similar-finite}. This proves
\eqref{eq:avg-K-finite}.
\end{proof}

In the global setting we do not need a separate analogue of
Proposition~\ref{prop:avg-kinetic-finite}: the dissipation identity gives the
stronger estimate
\[
    \int_{T_0}^{\infty}\|\partial_tu(t)\|_{L^2}^2\,dt<\infty .
\]
We now complete the proof of the sequential decomposition. 

First, we deduce Propositions~\ref{Properties of the radiation} and
\ref{Properties of the radiation global case}. In the finite-time case, let
\(R(t):=T_+-t\). Proposition~\ref{radiation in finite-time case} gives a regular part
\(u^*(t)\) and identifies \(u(t)\) with \(u^*(t)\) in the exterior region
\(r\geq R(t)\). Proposition~\ref{prop:no-self-similar-finite} gives, for every
fixed \(\lambda\in(0,1)\),
\[
    \|\boldsymbol u(t)\|_{\mathcal E(\lambda R(t),R(t))}\to0
    \qquad\text{as }t\to T_+.
\]
Since \(u^*(t)\) is regular up to \(t=T_+\), its energy in balls of radius
\(o(1)\) tends to zero. Hence, by a diagonal choice of \(\lambda\downarrow0\),
there exists \(\rho(t)\ll R(t)\) such that
\[
    \|\boldsymbol u(t)-\boldsymbol u^*(t)\|_{\mathcal E(\rho(t),\infty)}\to0.
\]
The remaining assertion
\[
    \|\boldsymbol u^*(t)\|_{\mathcal E(0,\gamma R(t))}\to0,\qquad 0<\gamma<1,
\]
follows from the same regularity of \(u^*\). This proves
Proposition~\ref{Properties of the radiation}.
In the global case, Propositions
~\ref{global case radiation} and\ref{prop:no-self-similar-global} imply that, for every fixed
\(\gamma\in(0,1)\),
\[
    \limsup_{t\to\infty}
    \|\boldsymbol u(t)\|_{\mathcal E(\gamma t,\infty)}=0.
\]
Indeed, the part \(\gamma t<r<t-R\) is controlled by the self-similar
non-concentration, and the part \(r>t-R\) by the vanishing of the global
radiation. A diagonal choice of \(\gamma\downarrow0\) gives a function
\(\rho(t)\ll t\) such that
\[
    \|\boldsymbol u(t)\|_{\mathcal E(\rho(t),\infty)}\to0,
\]
which proves Proposition~\ref{Properties of the radiation global case}. 
We now prove Theorem \ref{Sequential soliton resolution}.
\begin{proof}
Let \(\rho(t)\)
denote the auxiliary scale appearing in Proposition~\ref{Properties of the radiation}
in the finite-time case, and in Proposition~\ref{Properties of the radiation global case} in the
global case. Thus \(\rho(t)\ll T_+-t\) in the finite-time case, and
\(\rho(t)\ll t\) in the global case.
We first consider the finite-time blow-up case. Let \(T=T_+\). From
Proposition~\ref{prop:avg-kinetic-finite} and the elementary selection argument used in \cite{JL},
there exists a sequence \(\tau_n\to T\) such that
\begin{equation}
    \lim_{n\to\infty}
    \sup_{0<\sigma<T-\tau_n}
    \frac1{\sigma}
    \int_{\tau_n}^{\tau_n+\sigma}
    \int_0^{T-t}
    |\partial_tu(t,r)|^2r^{D-1}\,drdt
    =0.
    \label{eq:finite-maximal-selection}
\end{equation}
Choose a sequence \(\ell_n>0\) such that
\[
    \sup_{t\in[\tau_n,\tau_n+\ell_n]}\rho(t)\ll \ell_n
    \qquad\text{and}\qquad
    \ell_n\ll T-\tau_n.
\]
Set $R_n:=\frac{T-\tau_n-\ell_n}{\ell_n}.$
After replacing \(R_n\) by a slower sequence, we may assume
\[
    R_n\to+\infty,
    \qquad
    \ell_nR_n\ll T-\tau_n.
\]
Define
\[
    u_n(s,r):=u(\tau_n+s,r),
    \qquad
    0\leq s\leq\ell_n.
\]
Then \eqref{eq:finite-maximal-selection} implies
\[
    \frac1{\ell_n}
    \int_0^{\ell_n}
    \int_0^{\ell_nR_n}
    |\partial_su_n(s,r)|^2r^{D-1}\,drds
    \to0.
\]
Therefore Lemma~\ref{lem:sequential-compactness}, Case I, applies. Hence, after
passing to a subsequence, there exist
\[
    s_n\in[0,\ell_n],
    \qquad
    1\ll r_n\leq R_n,
\]
such that $  \delta_{\ell_nr_n}(u_n(s_n))\to0.$
Set $t_n:=\tau_n+s_n$ and $ L_n:=\ell_nr_n.$
By construction,
\[
    \rho(t_n)\ll L_n\ll T-t_n.
\]
Returning to the original variables gives
\[
    \delta_{L_n}(u(t_n))\to0.
\]
Together with the finite-time radiation reduction in
Proposition~\ref{Properties of the radiation}, this yields the desired sequential
decomposition in the finite-time case.
We now treat the global case. By the dissipation identity, choose
\(\tau_n\to+\infty\) and \(\ell_n\to+\infty\) such that
\[
    \sup_{t\in[\tau_n,\tau_n+\ell_n]}\rho(t)\ll \ell_n,
    \qquad
    \ell_n\ll\tau_n,
\]
and
\[
    \int_{\tau_n}^{\tau_n+\ell_n}
    \|\partial_tu(t)\|_{L^2}^2\,dt
    \to0.
\]
Set $R_n:=\frac{\tau_n}{\ell_n}.$
Again replacing \(R_n\) by a slower sequence if necessary, we assume
\[
    R_n\to+\infty,
    \qquad
    \ell_nR_n\ll \tau_n.
\]
Define
\[
    u_n(s,r):=u(\tau_n+s,r),
    \qquad
    0\leq s\leq\ell_n.
\]
Then
\[
    \int_0^{\ell_n}
    \int_0^{\ell_nR_n}
    |\partial_su_n(s,r)|^2r^{D-1}\,drds
    \to0.
\]
Thus Lemma~\ref{lem:sequential-compactness}, Case II, applies and gives
\[
    s_n\in[0,\ell_n],
    \qquad
    1\ll r_n\leq R_n,
\]
with
\[
    \delta_{\ell_nr_n}(u(\tau_n+s_n))\to0.
\]
Setting $t_n:=\tau_n+s_n$ and $L_n:=\ell_nr_n,$
we have
\[
    \rho(t_n)\ll L_n\ll t_n.
\]
By the global radiation reduction in Proposition~\ref{Properties of the radiation global case},
there is no exterior radiation term. Hence the last display yields the desired
sequential decomposition in the global case.
\end{proof}

\section{From sequential to full soliton resolution}
In this section we pass from the sequential soliton resolution proved in the
previous section to the full-time convergence stated in Theorem~\ref{soliton resolution}. The argument is a
no-return argument. Indeed, the sequential result implies that the solution enters
arbitrarily small neighborhoods of the \(N\)-bubble manifold. If the full-time
convergence failed, then the solution would have to leave such a neighborhood and
return to it along a sequence of time intervals. These intervals are the collision
intervals.

The proof has two parts. First, using the collision-interval reduction of \cite{JL}, we isolate the bubbles which are actually involved in the
collision and introduce modulation coordinates for them. This part is geometric:
it uses the exterior radiation estimate, finite speed of propagation, and the
static multi-bubble modulation lemma, and is unaffected by the damping term. The
damping enters only in the second part, where we verify the localized virial
no-return estimate. In the finite-time case the damping term is absorbed by an
exponential weight, while in the global case it is controlled by the dissipation
and by the vanishing of the radiation.
\subsection{Collision intervals and exterior-interior decomposition}
We begin with the notation used to separate the exterior bubbles from the bubbles
which may be involved in a collision. Set $ R_+(t):=T_+-t$ when $T_+<\infty$ and $R_+(t):=t$ when $T_+=\infty.$
In the finite-time case, let \(\boldsymbol u^*(t)\) be the radiation term given
by Proposition~\ref{Properties of the radiation}; in the global case we set
\(\boldsymbol u^*(t)\equiv0\).
By Propositions~\ref{Properties of the radiation} and
\ref{Properties of the radiation global case}, there exists a function
\(\rho:I_*\to(0,\infty)\) such that
\begin{equation}
      \lim_{t\to T_+}
    \left[
        \left(\frac{\rho(t)}{R_+(t)}\right)^{\frac{D-2}{2}}
        +
        \|\boldsymbol u(t)-\boldsymbol{u}^*(t)\|_{\mathcal E(\rho(t),\infty)}
    \right]
    =0.
    \label{eq:radiation-scale-section5}
\end{equation}
Moreover, in the finite-time case, for every \(\gamma\in(0,1)\),
\[
    \|\boldsymbol u^*(t)\|_{\mathcal E(0,\gamma R_+(t))}\to0
    \qquad\text{as } t\to T_+.
    \label{eq:regular-part-small-section5}
\]
By Theorem~\ref{Sequential soliton resolution}, there exist an integer
\(N\ge0\), a sequence \(t_n\to T_+\), signs, and scales such that
\(\boldsymbol u(t_n)-\boldsymbol u^*(t_n)\) converges to an \(N\)-bubble
configuration. We fix this \(N\) throughout the rest of the proof.
The case \(N=0\) contains no collision and is treated by the standard no-bubble
argument; hence we assume \(N\ge1\). We use the convention
\[
    \lambda_{N+1}(t):=R_+(t).
\]
\begin{Def}[Exterior proximity]
Let \(K\in\{0,\ldots,N\}\), \(t\in I_*\), and \(\rho\ge0\), with \(\rho>0\) if
\(K\ge1\). We define
\[
    d_K(t;\rho)
    :=
    \inf_{\boldsymbol\iota,\boldsymbol\lambda}
    \left(
        \left\|
            \boldsymbol u(t)-\boldsymbol{u}^*(t)
            -
            \sum_{j=K+1}^{N}\iota_j\boldsymbol W_{\lambda_j}
        \right\|_{\mathcal E(\rho,\infty)}^2
        +
        \sum_{j=K}^{N}
        \left(\frac{\lambda_j}{\lambda_{j+1}}\right)^{\frac{D-2}{2}}
    \right)^{1/2}.
    \label{eq:exterior-proximity}
\]
Here
\[
    \boldsymbol\iota=(\iota_{K+1},\ldots,\iota_N)\in\{-1,1\}^{N-K},
    \qquad
    \boldsymbol\lambda=(\lambda_{K+1},\ldots,\lambda_N)\in(0,\infty)^{N-K},
\]
and we use the convention \(\lambda_K:=\rho\). If \(K=N\), the sum of bubbles is
empty and the infimum is void. For \(K=0\) we only use \(\rho=0\), and set
\[
    d(t):=d_0(t;0).
\]
\end{Def}
With this notation, the sequential soliton resolution gives
\begin{equation*}
      \liminf_{t\to T_+}d(t)=0.
\end{equation*}
The full soliton resolution is equivalent to
\begin{equation}
     \lim_{t\to T_+}d(t)=0.
    \label{eq:d-goes-zero}
\end{equation}
We argue by contradiction and assume that \eqref{eq:d-goes-zero} fails. The next
definition records the intervals on which the solution leaves a small neighborhood
of the full \(N\)-bubble manifold, while the exterior \(N-K\) bubbles remain
well-described.
\begin{Def}[Collision intervals]
Let \(K\in\{0,\ldots,N\}\) and \(0<\varepsilon<\eta\). A compact interval
\([a,b]\subset I_*\) is called a \(K\)-collision interval with parameters
\((\varepsilon,\eta)\) if
\[
    d(a)\le\varepsilon,\qquad d(b)\le\varepsilon,
\]
there exists \(c\in(a,b)\) such that 
\[
    d(c)\ge\eta,
\]
and there exists a function \(\rho_K:[a,b]\to(0,\infty)\) such that
\[
    d_K(t;\rho_K(t))\le\varepsilon
    \qquad\text{for all }t\in[a,b].
\]
In this case we write $[a,b]\in\mathcal C_K(\varepsilon,\eta).$ 
\end{Def}
The following proposition selects the number of bubbles which are genuinely
involved in the collision and separates the remaining exterior bubbles.
\begin{prop}[Collision reduction and exterior decomposition]
\label{prop:collision-reduction}
Assume that \eqref{eq:d-goes-zero} fails. Let \(K\) be the smallest
non-negative integer with the following property: there exist a number
\(\eta>0\), a sequence \(\varepsilon_n\to0\), and disjoint compact intervals
\[
    I_n=[a_n,b_n]\subset I_*,
    \qquad a_n,b_n\to T_+,
\]
such that
\begin{equation*}
    I_n\in\mathcal C_K(\varepsilon_n,\eta)
    \qquad\text{for all }n.
\end{equation*}
Then \(K\) is well-defined and \(K\in\{1,\ldots,N\}\).
Fix \(K\), \(\eta\), \(\varepsilon_n\), and \(I_n=[a_n,b_n]\) as above. After
passing to a subsequence, there exists a Lipschitz function 
\[
    \nu_n:I_n\to(0,\infty)
\]
such that
\begin{equation}
    \sup_{t\in I_n}
    \left(
        d_K(t;\nu_n(t))
        +
        \|\boldsymbol u(t)-\boldsymbol u^*(t)\|_{\mathcal E(\nu_n(t),2\nu_n(t))}
    \right)
    \to0,
    \label{eq:nu-separation}
\end{equation}
and
\begin{equation}
    \sup_{t\in I_n}|\nu_n'(t)|\to0.
    \label{eq:nu-slow}
\end{equation}
Furthermore, there exist signs
\[
    \boldsymbol\sigma=(\sigma_{K+1},\ldots,\sigma_N)
    \in\{-1,1\}^{N-K},
\]
scales
\[
    \boldsymbol\mu_n(t)
    =
    (\mu_{K+1,n}(t),\ldots,\mu_{N,n}(t))
    \in C^1(I_n;(0,\infty)^{N-K}),
\]
and an exterior error \(\boldsymbol h_n(t)\in\mathcal E\) such that, for
\(t\in I_n\),
\begin{equation}
    (1-\chi_{\nu_n(t)})(\boldsymbol u(t)-\boldsymbol u^*(t))
    =
    \sum_{j=K+1}^{N}\sigma_j\boldsymbol W_{\mu_{j,n}(t)}
    +
    \boldsymbol h_n(t),
    \label{eq:exterior-decomposition}
\end{equation}
where \(\chi_\nu(r):=\chi(r/\nu)\), and \(\chi\) is a fixed smooth cut-off equal
to \(1\) on \(r\le1\) and \(0\) on \(r\ge2\). With the convention
\[
    \mu_{N+1,n}(t):=R_+(t),
\]
we have
\begin{equation}
    \sup_{t\in I_n}
    \left(
        \|\boldsymbol h_n(t)\|_{\mathcal E}^2
        +
        \left(\frac{\nu_n(t)}{\mu_{K+1,n}(t)}\right)^{\frac{D-2}{2}}
        +
        \sum_{j=K+1}^{N}
        \left(\frac{\mu_{j,n}(t)}{\mu_{j+1,n}(t)}\right)^{\frac{D-2}{2}}
    \right)
    \to0.
    \label{eq:exterior-smallness}
\end{equation}
If \(K=N\), the exterior sum is empty and the convention
\(\mu_{N+1,n}(t)=R_+(t)\) is used in \eqref{eq:exterior-smallness}.
\end{prop}

\begin{proof}
Since \(\liminf_{t\to T_+}d(t)=0\) and \eqref{eq:d-goes-zero} fails, the
continuity of \(d(t)\) gives collision intervals with \(K=N\); hence the above
minimal integer is well-defined. The case \(K=0\) is excluded by the same
argument as in the wave equation: if \(K=0\), then the whole \(N\)-bubble
configuration remains controlled on the interval by the exterior distance, which
contradicts the existence of a point where \(d(t)\ge\eta\). This argument uses
only the exterior smallness \eqref{eq:radiation-scale-section5}, finite speed of
propagation, and the continuity of the distance functions. Thus \(K\in\{1,\ldots,N\}\).

For this minimal \(K\), the exterior-interior separation follows from the
standard annular selection argument. Namely, using \(d_K(t;\rho_K(t))\le
\varepsilon_n\) on \(I_n\), one chooses a separating scale \(\nu_n(t)\) between
the interior \(K\) bubbles and the exterior \(N-K\) bubbles so that the annular
energy on \((\nu_n(t),2\nu_n(t))\) is \(o_n(1)\); after the usual Lipschitz
regularization, \(\sup_{t\in I_n}|\nu_n'(t)|\to0\). This gives
\eqref{eq:nu-separation} and \eqref{eq:nu-slow}. The exterior decomposition
\eqref{eq:exterior-decomposition} and the smallness
\eqref{eq:exterior-smallness} then follow by applying the static multi-bubble
modulation lemma in the region \(r\ge \nu_n(t)\).
This reduction is purely geometric: it uses only finite speed of propagation,
the exterior estimate \eqref{eq:radiation-scale-section5}, annular pigeonholing,
and the static multi-bubble modulation lemma. Hence the damping term does not
enter at this stage.
\end{proof}

\subsection{Interior modulation and corrected parameters}
We now work on the collision intervals \(I_n=[a_n,b_n]\) given by
Proposition~\ref{prop:collision-reduction}. The exterior \(N-K\) bubbles have
been separated by the scale \(\nu_n(t)\), and the remaining analysis concerns the
interior \(K\) bubbles. Since \(\nu_n(t)\ll R_+(t)\) and \(\boldsymbol u^*(t)\)
is negligible in \(r\lesssim R_+(t)\) in the finite-time case, while
\(\boldsymbol u^*\equiv0\) in the global case, the interior modulation may be
written for the localized solution \(\chi_{\nu_n(t)}\boldsymbol u(t)\) itself.
Here and below \(\chi_\nu(r):=\chi(r/\nu)\).

We also fix the convention for the small errors produced by this localization.
After passing to a subsequence and increasing the errors if necessary, we denote
by \(\zeta_n\to0\) a sequence which controls, uniformly for \(t\in I_n\), the
exterior error in Proposition~\ref{prop:collision-reduction}, the annular energy
on \((\nu_n(t),2\nu_n(t))\), the terms involving \(\nu_n'(t)\), and, in the
finite-time case, the contribution of \(\boldsymbol u^*(t)\) in the interior
region. Thus estimates identical to the corresponding wave-equation estimates
will be used below with an additional error \(\zeta_n\), or
\(\zeta_n/\lambda_j(t)\) after projecting an equation at scale \(\lambda_j(t)\).

The following lemma gives the basic coordinates near the interior \(K\)-bubble
manifold. It is the static part of the modulation analysis, together with the
first-order estimates needed later.
\begin{lemma}[Basic interior modulation]
\label{lem:basic-interior-modulation}
Assume \(D\ge6\). There exist constants \(C_0>0\) and \(\eta_0>0\) such that,
after enlarging the error sequence \(\zeta_n \to 0\) fixed above if necessary, the following
holds. Let \(J\subset I_n\) be an open interval such that
$d(t)\le\eta_0$ for all $t\in J$.
Then there exist signs
\[
    \boldsymbol\iota=(\iota_1,\ldots,\iota_K)\in\{-1,1\}^K,
\]
independent of \(t\in J\), \(C^1\) modulation parameters
\[
    \boldsymbol\lambda(t)=(\lambda_1(t),\ldots,\lambda_K(t))
    \in C^1(J;(0,\infty)^K),
\]
and a remainder $\boldsymbol g(t)=(g(t),\dot g(t))\in\mathcal E$
such that, for all \(t\in J\),
\begin{equation}
    \chi_{\nu_n(t)}\boldsymbol u(t)
    =
    \sum_{j=1}^{K}\iota_j\boldsymbol W_{\lambda_j(t)}
    +
    \boldsymbol g(t),
    \qquad
    \langle Z_{\lambda_j(t)},g(t)\rangle=0
    \quad 1\le j\le K.
    \label{eq:basic-modulation-decomposition}
\end{equation}
Define the stable and unstable components by
\begin{equation*}
    a_j^\pm(t)
    :=
    \langle \boldsymbol\alpha_{\lambda_j(t)}^\pm,\boldsymbol g(t)\rangle,
    \qquad 1\le j\le K,
\end{equation*}
where \(\boldsymbol\alpha_\lambda^\pm\) are defined in \eqref{eq:alpha-pm}.
Then, for all \(t\in J\),
\begin{equation}
    C_0^{-1}d(t)-\zeta_n
    \le
    \|\boldsymbol g(t)\|_{\mathcal E}
    +
    \sum_{j=1}^{K-1}
    \left(\frac{\lambda_j(t)}{\lambda_{j+1}(t)}\right)^{\frac{D-2}{4}}
    \le
    C_0d(t)+\zeta_n.
    \label{eq:d-basic-equivalence}
\end{equation}
Moreover, if $ S:=\{j\in\{1,\ldots,K-1\}:\iota_j=\iota_{j+1}\}$,
then
\begin{equation}
    \|\boldsymbol g(t)\|_{\mathcal E}
    +
    \sum_{j\notin S}
    \left(\frac{\lambda_j(t)}{\lambda_{j+1}(t)}\right)^{\frac{D-2}{4}}
    \le
    C_0
    \max_{j\in S}
    \left(\frac{\lambda_j(t)}{\lambda_{j+1}(t)}\right)^{\frac{D-2}{4}}
    +
    C_0\max_{\substack{1\le i\le K\\ \pm}}|a_i^\pm(t)|
    +
    \zeta_n .
    \label{eq:coercive-basic-modulation}
\end{equation}
The scale parameters satisfy the rough derivative estimate
\begin{equation}
    |\lambda_j'(t)|
    \le
    C_0\|\dot g(t)\|_{L^2}
    +
    \zeta_n,
    \qquad 1\le j\le K.
    \label{eq:lambda-rough-derivative}
\end{equation}
Finally, for \(1\le j\le K\),
\begin{equation}
    \left|
        \frac{d}{dt}a_j^\pm(t)
        \mp
        \frac{\kappa}{\lambda_j(t)}a_j^\pm(t)
        \pm
        \frac{\alpha}{2}
        \langle Y_{\underline{\lambda}_j(t)},\dot g(t)\rangle
    \right|
    \le
    \frac{C_0}{\lambda_j(t)}d(t)^2
    +
    \frac{\zeta_n}{\lambda_j(t)}.
    \label{eq:aj-damped-dynamics}
\end{equation}
\end{lemma}
\begin{proof}
We only indicate the points where the damped equation enters. The existence of
the decomposition \eqref{eq:basic-modulation-decomposition}, the orthogonality
conditions, and the estimates \eqref{eq:d-basic-equivalence} and
\eqref{eq:coercive-basic-modulation} are consequences of the static modulation
lemma near a \(K\)-bubble configuration. This part uses only the elliptic
multi-bubble geometry and is identical to the wave case. The exterior bubbles and
the cut-off errors are absorbed into the uniform error \(\zeta_n\) by
Proposition~\ref{prop:collision-reduction}.
Differentiating the orthogonality conditions
\[
    \langle Z_{\lambda_j(t)},g(t)\rangle=0
\]
gives the usual modulation system for the parameters \(\lambda_j(t)\). Since the
matrix of this system is a small perturbation of a diagonal one, and since the
terms supported in the annulus \((\nu_n(t),2\nu_n(t))\) are \(o_n(1)\), we obtain
\eqref{eq:lambda-rough-derivative}. This argument is again the same as in the
undamped case.
It remains to record the evolution of the stable and unstable components. We
write
\[
    \boldsymbol W(t):=\sum_{i=1}^K\iota_i\boldsymbol W_{\lambda_i(t)}
\]
and set $ \boldsymbol w(t):=\chi_{\nu_n(t)}\boldsymbol u(t)
    =\boldsymbol W(t)+\boldsymbol g(t)$.
Since \(\boldsymbol u\) solves \(\partial_t\boldsymbol u=\widetilde JDE(\boldsymbol u)\),
we have the exact identity
\begin{equation*}
    \partial_t\boldsymbol w
    =
    \widetilde JDE(\boldsymbol w)
    +
    \boldsymbol\Phi_n(t),
    \label{eq:w-localized-equation}
\end{equation*}
where the localization error is
\begin{equation*}
    \boldsymbol\Phi_n(t)
    :=
    \chi_{\nu_n(t)}\widetilde JDE(\boldsymbol u(t))
    -
    \widetilde JDE(\chi_{\nu_n(t)}\boldsymbol u(t))
    -
    \frac{\nu_n'(t)}{\nu_n(t)}
    (r\partial_r\chi)_{\nu_n(t)}\boldsymbol u(t).
    \label{eq:localization-error}
\end{equation*}
Consequently,
\[
    \partial_t\boldsymbol g
    =
    \widetilde JDE(\boldsymbol W+\boldsymbol g)
    -
    \partial_t\boldsymbol W
    +
    \boldsymbol\Phi_n(t).
\]
We decompose the first term into its linearized part and the remaining interaction
terms:
\begin{equation}
    \partial_t\boldsymbol g
    =
    \widetilde J D^2E(\boldsymbol W(t))\boldsymbol g
    -
    \partial_t\boldsymbol W(t)
    +
    \boldsymbol{\mathcal R}(t),
    \label{eq:g-schematic-basic}
\end{equation}
where
\begin{equation*}
    \boldsymbol{\mathcal R}(t)
    :=
    \widetilde J\Big(
        DE(\boldsymbol W+\boldsymbol g)
        -
        D^2E(\boldsymbol W)\boldsymbol g
    \Big)
    +
    \boldsymbol\Phi_n(t).
\end{equation*}
Thus \(\boldsymbol{\mathcal R}\) contains the static interaction of the bubbles,
the terms at least quadratic in \(\boldsymbol g\), and the localization errors.
By the adjacent-bubble interaction estimates, the smallness of
\(\boldsymbol g\), and Proposition~\ref{prop:collision-reduction}, for all
\(1\le j\le K\),
\begin{equation}
    \left|
        \langle \boldsymbol\alpha_{\lambda_j(t)}^\pm,
        \boldsymbol{\mathcal R}(t)\rangle
    \right|
    \le
    \frac{C}{\lambda_j(t)}d(t)^2
    +
    \frac{\zeta_n}{\lambda_j(t)}.
    \label{eq:R-basic-bound}
\end{equation}
We now project \eqref{eq:g-schematic-basic} onto
\(\boldsymbol\alpha_{\lambda_j(t)}^\pm\). Differentiating
\(a_j^\pm(t)=\langle\boldsymbol\alpha_{\lambda_j(t)}^\pm,\boldsymbol g(t)\rangle\)
gives
\[
    \frac{d}{dt}a_j^\pm(t)
    =
    \langle\boldsymbol\alpha_{\lambda_j(t)}^\pm,
    \widetilde J D^2E(\boldsymbol W(t))\boldsymbol g(t)\rangle
    +\mathcal E_j^\pm(t),
\]
where \(\mathcal E_j^\pm\) contains the derivative of
\(\boldsymbol\alpha_{\lambda_j(t)}^\pm\), the term \(-\partial_t\boldsymbol W(t)\),
the difference between the full multi-bubble linearized operator and the
one-bubble operator at \(W_{\lambda_j(t)}\), and the remainder
\(\boldsymbol{\mathcal R}(t)\). By \eqref{eq:lambda-rough-derivative},
\eqref{eq:d-basic-equivalence}, scale separation, and
\eqref{eq:R-basic-bound},
\begin{equation}
    |\mathcal E_j^\pm(t)|
    \le
    \frac{C}{\lambda_j(t)}d(t)^2
    +
    \frac{\zeta_n}{\lambda_j(t)} .
    \label{eq:aj-error-bound}
\end{equation}
The principal term is exactly the one-bubble linearized dynamics. The wave part
gives the usual eigenvalue contribution, while the damping part contributes only
through the velocity component:
\begin{equation}
    \left\langle
        \boldsymbol\alpha_{\lambda_j(t)}^\pm,
        \widetilde J D^2E(\boldsymbol W_{\lambda_j(t)})\boldsymbol g(t)
    \right\rangle
    =
    \pm\frac{\kappa}{\lambda_j(t)}a_j^\pm(t)
    \mp
    \frac{\alpha}{2}
    \langle Y_{\underline{\lambda}_j(t)},\dot g(t)\rangle .
    \label{eq:damped-eigen-projection}
\end{equation}
Combining \eqref{eq:aj-error-bound} and
\eqref{eq:damped-eigen-projection} yields \eqref{eq:aj-damped-dynamics}.
\end{proof}

We next recall the localized virial correction developed in \cite{JL 2018, JL 2019}. It is designed to obtain a closed estimate for the derivative of the scale velocity. If one
uses the naive velocity
\[
    b_j^0(t)
    :=
    -
    \frac{\iota_j}{\|\Lambda W\|_{L^2}^2}
    \langle \Lambda W_{\lambda_j(t)},\dot g(t)\rangle ,
\]
then differentiating \(b_j^0\) gives the expected adjacent-bubble force, but also
a quadratic virial term in \(g\), of size \(d(t)^2/\lambda_j(t)\), coming from
the scaling direction. This term is not perturbative.
The correction is to add a localized scaling term
\[
    -
    \frac{1}{\|\Lambda W\|_{L^2}^2}
    \langle A(\lambda_j)g,\dot g\rangle .
\]
The point is twofold. First, \(A(\lambda)\) is uniformly bounded
\(\dot H^1\to L^2\), hence
\[
    |\langle A(\lambda_j)g,\dot g\rangle|
    \lesssim
    \|g\|_{\dot H^1}\|\dot g\|_{L^2}
    \lesssim d(t)^2,
\]
so the correction does not change the leading meaning of the scale velocity.
Second, \(A(\lambda)\) agrees with \(\lambda^{-1}\Lambda\) on the annulus
\(r\sim\lambda\). Therefore, when the correction is differentiated, it produces
the localized virial form which cancels the bad scaling contribution and is then
controlled by the localized coercivity estimates below.
To construct \(A(\lambda)\), fix \(c_0>0\). Choose \(c>0\) small and \(R>1\)
large, and let \(q=q_{c,R}\) be the function constructed in
\cite[Lemma 5.14]{JL}. Thus
\[
    q(r)=\frac12 r^2
    \qquad\text{for }r\in[R^{-1},R],
\]
while \(q\) is constant near \(0\) and near infinity, and satisfies the derivative
bounds and sign conditions stated there. For \(\lambda>0\), define
\begin{equation*}
    A(\lambda)g(r)
    :=
    q'\left(\frac r\lambda\right)\partial_rg(r)
    +
    \frac{D-2}{2D\lambda}
    \Delta q\left(\frac r\lambda\right)g(r),
\end{equation*}
and
\begin{equation*}
    \underline A(\lambda)g(r)
    :=
    q'\left(\frac r\lambda\right)\partial_rg(r)
    +
    \frac{1}{2\lambda}
    \Delta q\left(\frac r\lambda\right)g(r),
\end{equation*}
where \(\Delta=\partial_r^2+\frac{D-1}{r}\partial_r\). On the annulus
\(R^{-1}\lambda\le r\le R\lambda\), one has
\[
    A(\lambda)g
    =
    \frac1\lambda
    \left(r\partial_rg+\frac{D-2}{2}g\right),
    \qquad
    \underline A(\lambda)g
    =
    \frac1\lambda
    \left(r\partial_rg+\frac D2g\right).
\]
We shall use the following localized virial estimates from
\cite[Lemma 5.16]{JL}.
After choosing \(c>0\) sufficiently small and \(R>1\)
sufficiently large, for all \(g\) in the energy space,
\begin{equation}
    \langle \underline A(\lambda)g,-\Delta g\rangle
    \ge
    -\frac{c_0}{\lambda}\|g\|_{\mathcal E}^2
    +
    \frac1\lambda
    \int_{R^{-1}\lambda}^{R\lambda}
        |\partial_rg|^2r^{D-1}\,dr,
    \label{eq:A-coercivity-free}
\end{equation}
and, for every admissible \(Z\) with
\(\langle Z,\Lambda W\rangle>0\), \(\langle Z,Y\rangle=0\), if
\(\langle g,Z_\lambda\rangle=0\), then
\begin{equation}
    \frac1\lambda
    \int_{R^{-1}\lambda}^{R\lambda}
        |\partial_rg|^2r^{D-1}\,dr
    -
    \frac1\lambda
    \int_0^\infty
        \frac1D
        \Delta q\left(\frac r\lambda\right)
        f'(W_\lambda)g^2r^{D-1}\,dr
    \ge
    \frac{c_0}{\lambda}\|g\|_{\mathcal E}^2
    -
    \frac{C}{\lambda}a_\lambda^2 .
    \label{eq:A-coercivity-linearized}
\end{equation}
Here \(a_\lambda\) denotes the projection of \(g\) onto the negative mode at
scale \(\lambda\).
In the applications below this term is controlled by the stable and
unstable coefficients \(a_j^\pm\).
For \(1\le j\le K-1\), define the corrected scale
\begin{equation}
    \xi_j(t)
    :=
    \begin{cases}
        \lambda_j(t), & D\ge7,\\[0.4em]
        \displaystyle
        \lambda_j(t)
        -
        \frac{\iota_j}{\|\Lambda W\|_{L^2}^2}
        \left\langle
            \chi\left(\frac{\cdot}{L\lambda_j(t)}\right)
            \Lambda W_{\lambda_j(t)},g(t)
        \right\rangle,
        & D=6,
    \end{cases}
    \label{eq:xi-definition}
\end{equation}
where \(L\gg1\) will be chosen large. We also define the corrected velocity
\begin{equation*}
    \beta_j(t)
    :=
    -
    \frac{\iota_j}{\|\Lambda W\|_{L^2}^2}
    \langle \Lambda W_{\lambda_j(t)},\dot g(t)\rangle
    -
    \frac{1}{\|\Lambda W\|_{L^2}^2}
    \langle A(\lambda_j(t))g(t),\dot g(t)\rangle .
\end{equation*}
The first term is the usual scale velocity; the second one is the localized virial
correction. Its purpose is to absorb the virial error in the scale dynamics.
\begin{lemma}[Corrected modulation estimates]
\label{lem:corrected-modulation}
Let \(D\ge6\) and \(c_0>0\). There exist constants
\(\eta_0>0\), \(L_0>0\), \(c>0\), \(R>1\), and \(C_0>0\), with the following
property. Let \(J\subset I_n\) be an open interval on which
\[
    d(t)\le\eta_0
    \qquad\text{for all }t\in J.
\]
After choosing \(L\ge L_0\) in \eqref{eq:xi-definition}, the quantities
\(\xi_j\) and \(\beta_j\) satisfy, for all \(1\le j\le K-1\) and \(t\in J\),
\begin{equation}
    \left|\frac{\xi_j(t)}{\lambda_j(t)}-1\right|
    \le c_0,
    \label{eq:xi-close-lambda}
\end{equation}
and
\begin{equation}
    |\xi_j'(t)-\beta_j(t)|
    \le c_0 d(t)+\zeta_n.
    \label{eq:xi-beta-close}
\end{equation}
Moreover,
\begin{align}
    \beta_j'(t)+\alpha\beta_j(t)
    \ge\;&
    (\iota_j\iota_{j+1}\omega^2-c_0)
    \frac1{\lambda_j(t)}
    \left(\frac{\lambda_j(t)}{\lambda_{j+1}(t)}\right)^{\frac{D-2}{2}}
    \notag\\
    &+
    (-\iota_j\iota_{j-1}\omega^2-c_0)
    \frac1{\lambda_j(t)}
    \left(\frac{\lambda_{j-1}(t)}{\lambda_j(t)}\right)^{\frac{D-2}{2}}
    \notag\\
    &-
    \frac{c_0}{\lambda_j(t)}d(t)^2
    -
    \frac{C_0}{\lambda_j(t)}
    \left((a_j^+(t))^2+(a_j^-(t))^2\right)
    -
    \frac{\zeta_n}{\lambda_j(t)},
    \label{eq:beta-damped-ineq}
\end{align}
where, by convention, \(\lambda_0(t)=0\), \(\lambda_{K+1}(t)=+\infty\), and
\[
    \omega^2
    :=
    \frac{D-2}{2D}
    (D(D-2))^{\frac D2}
    \|\Lambda W\|_{L^2}^{-2}>0.
    \label{eq:omega-definition}
\]
\end{lemma}
\begin{proof}
The estimates \eqref{eq:xi-close-lambda} and \eqref{eq:xi-beta-close} are the
same corrected-scale estimates as in \cite[Lemma 5.19]{JL}. They use only the
definition of \(\xi_j\), the orthogonality conditions in
\eqref{eq:basic-modulation-decomposition}, the rough estimate
\eqref{eq:lambda-rough-derivative}, and the localized boundedness of \(A(\lambda)\).
The errors caused by the exterior cutoff are \(o_n(1)\) by
Proposition~\ref{prop:collision-reduction}, and are absorbed into \(\zeta_n\).

It remains to prove the differential inequality for \(\beta_j\). We write
\[
    A_j:=A(\lambda_j(t)),
    \qquad
    \underline A_j:=\underline A(\lambda_j(t)),
    \qquad
    W_j:=W_{\lambda_j(t)} .
\]
Differentiating the definition of \(\beta_j\), using the equation for the second
component of \(\boldsymbol g\), and separating the terms which contain the
damping, one obtains
\begin{equation}
    \|\Lambda W\|_{L^2}^2\beta_j'
    =
    I_j+Q_j+\mathcal E_j
    +
    \alpha\iota_j\langle \Lambda W_j,\dot g\rangle
    +
    \alpha\langle A_jg,\dot g\rangle .
    \label{eq:beta-prime-decomposition}
\end{equation}
Here \(I_j\) is the contribution of the static interaction of the bubbles,
\[
    I_j
    :=
    -\frac{\iota_j}{\lambda_j}
    \left\langle
        \Lambda W_j,\,
        f\left(\sum_{i=1}^K\iota_i W_{\lambda_i}\right)
        -
        \sum_{i=1}^K\iota_i f(W_{\lambda_i})
    \right\rangle ,
\]
and \(Q_j\) is the localized virial quadratic form
\begin{equation*}
    Q_j
    :=
    \langle \underline A_jg,-\Delta g\rangle
    -
    \frac1{\lambda_j}
    \int_0^\infty
        \frac1D
        \Delta q\left(\frac r{\lambda_j}\right)
        f'(W_j)g^2 r^{D-1}\,dr .
\end{equation*}
The remainder \(\mathcal E_j\) contains the nonlinear terms at least cubic in
\(g\), the scale-separated interaction errors, the terms involving
\(\lambda_i'\), and the localization errors. By the estimates used in the
undamped case, together with \eqref{eq:lambda-rough-derivative} and the definition
of \(\zeta_n\),
\begin{equation}
    |\mathcal E_j|
    \le
    \|\Lambda W\|_{L^2}^2
    \left(
        \frac{c_0}{\lambda_j}d(t)^2
        +
        \frac{\zeta_n}{\lambda_j}
    \right),
    \label{eq:beta-remainder-bound}
\end{equation}
provided \(\eta_0\) is chosen sufficiently small and \(n\) sufficiently large.
This is precisely the part of the computation which is unchanged from
\cite[Lemma 5.19]{JL}; the exterior cut-off errors are absorbed into
\(\zeta_n\).
The damping contribution in \eqref{eq:beta-prime-decomposition} is explicit. By
the definition of \(\beta_j\),
\[
    \alpha\iota_j\langle \Lambda W_j,\dot g\rangle
    +
    \alpha\langle A_jg,\dot g\rangle
    =
    -\alpha\|\Lambda W\|_{L^2}^2\beta_j .
\]
Hence
\begin{equation}
    \|\Lambda W\|_{L^2}^2(\beta_j'+\alpha\beta_j)
    =
    I_j+Q_j+\mathcal E_j .
    \label{eq:beta-alpha-main}
\end{equation}
We now estimate the two principal terms. The standard adjacent-bubble computation
gives
\begin{align}
    \frac{I_j}{\|\Lambda W\|_{L^2}^2}
    \ge\;&
    (\iota_j\iota_{j+1}\omega^2-c_0)
    \frac1{\lambda_j}
    \left(\frac{\lambda_j}{\lambda_{j+1}}\right)^{\frac{D-2}{2}}
    \notag\\
    &+
    (-\iota_j\iota_{j-1}\omega^2-c_0)
    \frac1{\lambda_j}
    \left(\frac{\lambda_{j-1}}{\lambda_j}\right)^{\frac{D-2}{2}}
    -
    \frac{c_0}{\lambda_j}d(t)^2
    -
    \frac{\zeta_n}{\lambda_j}.
    \label{eq:Ij-estimate}
\end{align}
Here the convention is \(\lambda_0=0\), \(\lambda_{K+1}=+\infty\). The error \(c_0d(t)^2/\lambda_j\)
comes from the non-adjacent interactions and from taking \(\eta_0\) small.
For \(Q_j\), we use the localized virial coercivity estimates
\eqref{eq:A-coercivity-free},\eqref{eq:A-coercivity-linearized}. Since
\(\langle Z_{\lambda_j},g\rangle=0\), the negative direction is controlled by the
stable and unstable coefficients, and we obtain
\begin{equation}
    \frac{Q_j}{\|\Lambda W\|_{L^2}^2}
    \ge
    -
    \frac{c_0}{\lambda_j}d(t)^2
    -
    \frac{C_0}{\lambda_j}
    \left((a_j^+(t))^2+(a_j^-(t))^2\right).
    \label{eq:Qj-estimate}
\end{equation}
Combining \eqref{eq:beta-alpha-main}, \eqref{eq:beta-remainder-bound},
\eqref{eq:Ij-estimate}, and \eqref{eq:Qj-estimate} yields
\eqref{eq:beta-damped-ineq}.
\end{proof}
As a consequence of the modulation estimates above, we also record the following
localized virial bound, which will be used in the no-return argument.
\begin{coro}[Localized virial control]
\label{cor:localized-virial-control}
There exist constants \(C_0>0\), \(\eta_0>0\), and a sequence
\(\delta_n\downarrow0\), with $\frac{\zeta_n}{\delta_n}\to0,$
such that the following holds. Let \(J\subset I_n\) be an open interval on which
\[
    \delta_n\le d(t)\le\eta_0
    \qquad\text{for all }t\in J.
\]
Let \(\rho:J\to(0,\infty)\) be a \(C^1\) function satisfying
\[
    \rho(t)\le \nu_n(t),
    \qquad
    |\rho'(t)|\le1 .
\]
Then, for all \(t\in J\),
\begin{equation}
    \left|
        \Omega_{1,\rho(t)}(\boldsymbol u(t))
        +
        \frac{D-2}{2}\Omega_{2,\rho(t)}(\boldsymbol u(t))
    \right|
    \le
    C_0 d(t).
    \label{eq:localized-virial-control}
\end{equation}
\end{coro}
\begin{proof}
Since \(\rho(t)\le\nu_n(t)\), the quantities
\(\Omega_{1,\rho(t)}\) and \(\Omega_{2,\rho(t)}\) only see the interior region.
Using the decomposition
\[
    \chi_{\nu_n(t)}\boldsymbol u(t)
    =
    \sum_{j=1}^K\iota_j\boldsymbol W_{\lambda_j(t)}
    +
    \boldsymbol g(t),
\]
the same computation as in the wave equation gives
\[
    \left|
        \Omega_{1,\rho(t)}(\boldsymbol u(t))
        +
        \frac{D-2}{2}\Omega_{2,\rho(t)}(\boldsymbol u(t))
    \right|
    \lesssim
    \|\boldsymbol g(t)\|_{\mathcal E}
    +
    \sum_{j=1}^{K-1}
    \left(\frac{\lambda_j(t)}{\lambda_{j+1}(t)}\right)^{\frac{D-2}{4}}
    +
    \zeta_n .
\]
Here the pure multi-bubble contribution cancels by the scaling identity, while
the terms produced by the exterior cut-off, the annular region, and the radiation
are absorbed into \(\zeta_n\). Applying \eqref{eq:d-basic-equivalence} and using
\(\zeta_n\le o(1)\delta_n\le o(1)d(t)\), after increasing \(C_0\) if necessary,
yields \eqref{eq:localized-virial-control}. The same argument, without using the lower bound \(d(t)\ge\delta_n\), gives the
weaker estimate
\[
\left|
    \Omega_{1,\rho(t)}(u(t))
    +\frac{D-2}{2}\Omega_{2,\rho(t)}(u(t))
\right|
\le C_0 d(t)+\zeta_n
\]
whenever \(d(t)\le\eta_0\), \(\rho(t)\le\nu_n(t)\), and \(|\rho'(t)|\le1\).
\end{proof}

\subsection{Scale control and interval decomposition}
We first introduce the auxiliary scale which measures the size of the interior
\(K\)-bubble cluster. This scale will be used later to construct moving cut-offs
between the interior bubbles and the exterior region.
Fix \(\kappa_1>0\) sufficiently small. For \(t\in I_n\), define
\[
    \mu_n(t)
    :=
    \sup\left\{
        r\le \nu_n(t):
        \|\boldsymbol u(t)\|_{\mathcal E(r,\nu_n(t))}=\kappa_1
    \right\}.
\]
For \(n\) large, this number is well-defined and satisfies
\(\mu_n(t)<\nu_n(t)\). We then define its Lipschitz regularization by
\begin{equation}
    \mu_{*,n}(t)
    :=
    \inf_{s\in I_n}\bigl(4\mu_n(s)+|s-t|\bigr),
    \qquad t\in I_n.
    \label{eq:mu-star-definition}
\end{equation}
When no confusion is possible we write simply \(\mu_*(t)\).
\begin{lemma}[Auxiliary scale]
\label{lem:auxiliary-scale}
There exist constants \(\eta_0>0\), \(C_0>0\), and \(\kappa_2>0\) such that,
after taking \(n\) sufficiently large, the following properties hold.
\begin{enumerate}
    \item The function \(\mu_*\) is \(1\)-Lipschitz on \(I_n\), and
    \[
        \mu_*(t)\le 4\mu_n(t)
        \qquad\text{for all }t\in I_n.
    \]
    \item If \(t\in I_n\) and \(d(t)\le\eta_0\), then
    \begin{equation}
        \kappa_2\lambda_K(t)
        \le
        \mu_*(t)
        \le
        \kappa_2^{-1}\lambda_K(t).
        \label{eq:mu-star-lambdaK}
    \end{equation}
    \item Let \(t_n\in I_n\). Suppose that \(d(t_n)\le\eta_0\), and that there
    exists a sequence \(R_n\to\infty\) such that
    \[
        R_n\mu_*(t_n)\le \nu_n(t_n),
        \qquad
        \frac{R_n\mu_*(t_n)}{\nu_n(t_n)}\to0,
    \]
    and $   \|\boldsymbol u(t_n)\|_{\mathcal E(R_n\mu_*(t_n),\,\nu_n(t_n))}
        \to0 .$
    Then \(d(t_n)\to0\).
\end{enumerate}
\end{lemma}
\begin{proof}
The construction is the same as in the wave equation. The first property follows
directly from the definition \eqref{eq:mu-star-definition}. If \(d(t)\le\eta_0\),
the interior decomposition from Lemma~\ref{lem:basic-interior-modulation} shows
that the outermost interior bubble is located at scale \(\lambda_K(t)\), while
the exterior region starts beyond \(\nu_n(t)\). Choosing \(\kappa_1\) small and
\(\eta_0\) small gives \eqref{eq:mu-star-lambdaK}. The last assertion is the
standard finite-speed consequence: if there is no energy in an annulus separating
the scale \(\mu_*(t_n)\) from the exterior scale \(\nu_n(t_n)\), then the
interior cluster is already separated from the exterior region, and the static
modulation lemma implies \(d(t_n)\to0\). The damping term does not affect this
argument.
\end{proof}
The next consequence gives the lower bound on the length of a genuine excursion
away from the multi-bubble manifold.
\begin{lemma}[Length of excursions]
\label{lem:excursion-length}
For every \(0<\varepsilon<\eta<\eta_0\) there exists \(C_\eta>0\) such that the
following holds for all sufficiently large \(n\). Let
\([c,d]\subset I_n\) satisfy $d(c)\le\varepsilon,$ $d(d)\le\varepsilon,$
and suppose that there exists \(t_0\in[c,d]\) with $d(t_0)\ge\eta .$
Then
\begin{equation*}
    d-c
    \ge
    C_\eta^{-1}\max\bigl(\mu_*(c),\mu_*(d)\bigr).
\end{equation*}
\end{lemma}
\begin{proof}
This is again the finite-speed argument of \cite{JL}. If
\(d-c\ll\max(\mu_*(c),\mu_*(d))\), then the energy distribution on the relevant
annuli cannot change enough between the endpoints and the point \(t_0\). Using
Lemma~\ref{lem:auxiliary-scale}, one obtains an annular region separating the
interior cluster from the exterior scale with vanishing energy, which forces
\(d(t_0)\to0\), contradicting \(d(t_0)\ge\eta\). The proof uses only finite speed
of propagation and the exterior-interior decomposition, and is unchanged by the
damping term.
\end{proof}

The next proposition is the point where the corrected modulation estimates are
used. It controls the solution on intervals on which the distance to the
multi-bubble manifold stays small.
\begin{prop}[Control on modulation intervals]
\label{prop:modulation-interval-control}
There exist constants \(\eta_0>0\), \(C_0>0\), and a sequence
\(\delta_n\downarrow0\), with $\frac{\zeta_n}{\delta_n^2}\to0,$
such that the following holds. Let \([t_1,t_2]\subset I_n\) be an interval on
which
\[
    \delta_n\le d(t)\le\eta_0
    \qquad\text{for all }t\in[t_1,t_2].
\]
Then, for \(n\) sufficiently large,
\begin{equation}
    \sup_{t\in[t_1,t_2]}\lambda_K(t)
    \le
    \frac43
    \inf_{t\in[t_1,t_2]}\lambda_K(t),
    \label{eq:lambdaK-control}
\end{equation}
and
\begin{equation}
    \int_{t_1}^{t_2}d(t)\,dt
    \le
    C_0
    \left(
        d(t_1)^{\frac4{D-2}}\lambda_K(t_1)
        +
        d(t_2)^{\frac4{D-2}}\lambda_K(t_2)
    \right).
    \label{eq:modulation-interval-integral}
\end{equation}
\end{prop}
\begin{proof}
We first consider the finite-time case. Set
\[
    \widetilde\beta_j(t)
    :=
    e^{\alpha(t-T_+)}\beta_j(t).
\]
Then
\[
    \widetilde\beta_j'(t)
    =
    e^{\alpha(t-T_+)}
    \bigl(\beta_j'(t)+\alpha\beta_j(t)\bigr).
\]
Since \(t\to T_+\) on the collision intervals, the factor
\(e^{\alpha(t-T_+)}\) is uniformly comparable to \(1\). Hence
Lemma~\ref{lem:corrected-modulation} gives the same differential inequality for
\(\widetilde\beta_j\) as in the wave equation, up to errors bounded by
\(\zeta_n/\lambda_j(t)\). Because \(d(t)\ge\delta_n\) and
\(\zeta_n/\delta_n^2\to0\), these errors are absorbed into the
\(c_0d(t)^2/\lambda_j(t)\) term.
We now follow the finite-dimensional argument of \cite{JL}. Let $S:=\{j\in\{1,\ldots,K-1\}:\iota_j=\iota_{j+1}\}.$
For \(C_1>0\) sufficiently large, define
\begin{equation*}
    \Phi(t)
    :=
    \sum_{j\in S}2^{-j}\xi_j(t)\widetilde\beta_j(t)
    -
    C_1\sum_{j=1}^{K}\lambda_j(t)(a_j^-(t))^2
    +
    C_1\sum_{j=1}^{K}\lambda_j(t)(a_j^+(t))^2 .
\end{equation*}
We claim that
\begin{equation}
    \Phi'(t)\ge c\,d(t)^2
    \qquad\text{for all }t\in[t_1,t_2],
    \label{eq:Phi-monotone}
\end{equation}
where \(c>0\) depends only on \(D\) and \(N\).
Indeed, differentiating \(\Phi\), using
\(\xi_j'=\beta_j+O(c_0d+\zeta_n)\), and recalling that
\(e^{\alpha(t-T_+)}\simeq1\) on \([t_1,t_2]\), gives
\[
\begin{aligned}
    \Phi'(t)
    \ge\;&
    c\sum_{j\in S}(\beta_j(t))^2
    +
    \sum_{j\in S}2^{-j}\lambda_j(t)\widetilde\beta_j'(t)  \\
    &+
    C_1\kappa\sum_{j=1}^{K}\big((a_j^-(t))^2+(a_j^+(t))^2\big)
    -
    \mathcal E_{\rm damp}(t)
    -
    c_0d(t)^2
    -
    \zeta_n .
\end{aligned}
\]
Here the terms containing \(\lambda_j'(a_j^\pm)^2\) are absorbed into
\(c_0d(t)^2\), using \eqref{eq:lambda-rough-derivative} and the smallness of
\(\eta_0\). The only new term compared with the undamped wave equation is
\[
    \mathcal E_{\rm damp}(t)
    :=
    C_1\alpha
    \sum_{j=1}^K
    \lambda_j(t)
    |\langle Y_{\underline{\lambda}_j(t)},\dot g(t)\rangle|
    \bigl(|a_j^-(t)|+|a_j^+(t)|\bigr).
\]
Since we are in the finite-time case and \(t_1\) is sufficiently close to
\(T_+\), the scales satisfy \(\lambda_j(t)\lesssim T_+-t_1\) on
\([t_1,t_2]\). Hence, by Cauchy--Schwarz and
\(\|\dot g(t)\|_{L^2}\lesssim d(t)+\zeta_n\),
\[
    \mathcal E_{\rm damp}(t)
    \le
    \frac12 C_1\kappa
    \sum_{j=1}^{K}\big((a_j^-(t))^2+(a_j^+(t))^2\big)
    +
    c_0 d(t)^2
    +
    \zeta_n ,
\]
after taking \(t_1\) closer to \(T_+\) and then \(n\) large.
Next, by \eqref{eq:beta-damped-ineq},
\[
\begin{aligned}
    \sum_{j\in S}2^{-j}\lambda_j\widetilde\beta_j'
    \ge\;&
    \omega^2 e^{\alpha(t-T_+)}
    \sum_{j\in S}2^{-j}
    \left(
        \iota_j\iota_{j+1}
        \left(\frac{\lambda_j}{\lambda_{j+1}}\right)^{\frac{D-2}{2}}
        -
        \iota_j\iota_{j-1}
        \left(\frac{\lambda_{j-1}}{\lambda_j}\right)^{\frac{D-2}{2}}
    \right)  \\
    &-
    C\sum_{j=1}^K\big((a_j^+)^2+(a_j^-)^2\big)
    -
    c_0d(t)^2
    -
    \zeta_n .
\end{aligned}
\]
The standard weighted-sum rearrangement gives
\[
    \sum_{j\in S}2^{-j}
    \left(
        \left(\frac{\lambda_j}{\lambda_{j+1}}\right)^{\frac{D-2}{2}}
        -
        \iota_j\iota_{j-1}
        \left(\frac{\lambda_{j-1}}{\lambda_j}\right)^{\frac{D-2}{2}}
    \right)
    \ge
    c
    \sum_{j\in S}
    \left(\frac{\lambda_j}{\lambda_{j+1}}\right)^{\frac{D-2}{2}} .
\]
Combining this with \eqref{eq:coercive-basic-modulation} and taking \(C_1\)
large, \(c_0\) small, and \(n\) large, yields \eqref{eq:Phi-monotone}.
We now turn to the size of \(\Phi\) at the endpoints. From
\eqref{eq:xi-close-lambda}, \eqref{eq:xi-beta-close}, the definition of
\(\beta_j\), and the boundedness of \(A(\lambda_j)\), we have
\[
    |\xi_j(t)|\lesssim \lambda_j(t),
    \qquad
    |\widetilde\beta_j(t)|\lesssim d(t),
    \qquad
    |a_j^\pm(t)|\lesssim d(t),
\]
for \(n\) large. Therefore, using
\(\lambda_j/\lambda_K\le \lambda_j/\lambda_{j+1}\) for \(j<K\), we obtain
\begin{equation}
    \frac{|\Phi(t)|}{\lambda_K(t)}
    \lesssim
    \sum_{j\in S}
    \frac{\lambda_j(t)}{\lambda_K(t)}|\widetilde\beta_j(t)|
    +
    \sum_{j=1}^{K}
    \frac{\lambda_j(t)}{\lambda_K(t)}
    \big((a_j^-(t))^2+(a_j^+(t))^2\big)
    \lesssim
    d(t)^{\frac{D+2}{D-2}} .
    \label{eq:Phi-upper}
\end{equation}
Indeed, the first term is bounded by
\[
    d(t)\sum_{j\in S}\frac{\lambda_j(t)}{\lambda_{j+1}(t)}
    \lesssim
    d(t)^{1+\frac4{D-2}},
\]
while the stable/unstable contribution is \(O(d(t)^2)\), which is bounded by the
right-hand side since \(D\ge6\) and \(d(t)\le\eta_0\).
Combining \eqref{eq:Phi-monotone} and \eqref{eq:Phi-upper} in the standard way
yields \eqref{eq:modulation-interval-integral}. Once
\eqref{eq:modulation-interval-integral} is known, the rough bound
\[
    |\lambda_K'(t)|\lesssim d(t)
\]
from Lemma~\ref{lem:basic-interior-modulation}, together with the smallness of
\(\eta_0\), gives \eqref{eq:lambdaK-control}. This proves the finite-time case.

We turn to the global case. If the lengths \(t_2-t_1\) are uniformly bounded,
the same proof applies with the local integrating factor
\[
    \widetilde\beta_j(t):=e^{\alpha(t-t_1)}\beta_j(t),
\]
since \(e^{\alpha(t-t_1)}\) is then uniformly bounded above and below. It remains
to consider the case where the lengths are not uniformly bounded.
Set
\[
    \gamma_n:=
    \int_{\inf I_n}^{\infty}\|\partial_tu(t)\|_{L^2}^2\,dt .
\]
Since \(I_n\to\infty\) in the global case, \(\gamma_n\to0\). We choose
\(\delta_n\downarrow0\) so slowly that, in addition to
\(\zeta_n/\delta_n^2\to0\),
\[
    \frac{\gamma_n^{1/2}}{\delta_n^2}\to0 .
\]
Suppose, toward a contradiction, that either \eqref{eq:lambdaK-control} or
\eqref{eq:modulation-interval-integral} fails for a sequence of intervals
\([t_{1,n},t_{2,n}]\subset I_n\) such that
\[
    t_{2,n}-t_{1,n}\to\infty,
    \qquad
    \delta_n\le d(t)\le\eta_0
    \quad\text{for all }t\in[t_{1,n},t_{2,n}].
\]
By the definition of \(\gamma_n\), and by the localization errors already
absorbed in \(\zeta_n\), we have
\[
    \int_{t_{1,n}}^{t_{2,n}}\|\dot g(t)\|_{L^2}^2\,dt
    \lesssim
    \gamma_n+\zeta_n .
\]
Hence, by the Cauchy--Schwarz inequality,
\[
\begin{aligned}
    \int_{t_{1,n}}^{t_{2,n}}\|\dot g(t)\|_{L^2}\,dt
    &\lesssim
    (t_{2,n}-t_{1,n})^{1/2}\gamma_n^{1/2}
    +
    \zeta_n(t_{2,n}-t_{1,n}) .
\end{aligned}
\]
On the other hand,
\[
    \int_{t_{1,n}}^{t_{2,n}}d(t)^2\,dt
    \ge
    \delta_n^2(t_{2,n}-t_{1,n}).
\]
By the choice of \(\delta_n\), it follows that
\begin{equation}
    \int_{t_{1,n}}^{t_{2,n}}\|\dot g(t)\|_{L^2}\,dt
    =
    o_n(1)
    \int_{t_{1,n}}^{t_{2,n}}d(t)^2\,dt .
    \label{eq:global-velocity-error-small}
\end{equation}
We use the unweighted functional
\[
    \Phi(t)
    :=
    \sum_{j\in S}2^{-j}\xi_j(t)\beta_j(t)
    -
    C_1\sum_{j=1}^{K}\lambda_j(t)(a_j^-(t))^2
    +
    C_1\sum_{j=1}^{K}\lambda_j(t)(a_j^+(t))^2 .
\]
Differentiating \(\Phi\) gives the same expression as in the finite-time case,
except that the damping is no longer absorbed by an exponential factor. More
precisely, using \eqref{eq:beta-damped-ineq} for
\(\beta_j'+\alpha\beta_j\) and then moving the additional
\(-\alpha\beta_j\)-contribution to the error side, we get
\[
\begin{aligned}
    \Phi'(t)\ge\;&
    c\sum_{j\in S}(\beta_j(t))^2
    +
    c\sum_{j\in S}
    \left(\frac{\lambda_j(t)}{\lambda_{j+1}(t)}\right)^{\frac{D-2}{2}}
    +
    c\sum_{j=1}^{K}\big((a_j^-(t))^2+(a_j^+(t))^2\big)    \\
    &-
    c_0d(t)^2
    -
    \zeta_n
    -
    \mathcal D_{\rm glob}(t),
\end{aligned}
\]
where the only new global error is
\[
    \mathcal D_{\rm glob}(t)
    :=
    C\sum_{j\in S}\lambda_j(t)|\beta_j(t)|
    +
    C\sum_{j=1}^{K}
    \lambda_j(t)
    |\langle Y_{\underline{\lambda_j}(t)},\dot g(t)\rangle|
    \bigl(|a_j^-(t)|+|a_j^+(t)|\bigr).
\]
The scales are uniformly bounded in the global case. Moreover, by the definition
of \(\beta_j\) and the boundedness of \(A(\lambda_j)\colon\dot H^1\to L^2\),
\[
    |\beta_j(t)|
    \lesssim
    \|\dot g(t)\|_{L^2},
\]
provided \(\eta_0\) is small. Also
\[
    |\langle Y_{\underline{\lambda}_j(t)},\dot g(t)\rangle|
    \lesssim
    \|\dot g(t)\|_{L^2},
    \qquad
    |a_j^\pm(t)|\lesssim d(t).
\]
Therefore
\[
    \mathcal D_{\rm glob}(t)
    \lesssim
    \|\dot g(t)\|_{L^2}
    +
    \|\dot g(t)\|_{L^2}d(t).
\]
Using \eqref{eq:global-velocity-error-small} and Cauchy--Schwarz, we obtain
\[
    \int_{t_{1,n}}^{t_{2,n}}\mathcal D_{\rm glob}(t)\,dt
    \le
    o_n(1)
    \int_{t_{1,n}}^{t_{2,n}}d(t)^2\,dt .
\]
Hence, after taking \(c_0\) small and \(n\) large, and using
\eqref{eq:coercive-basic-modulation}, we arrive at the integrated monotonicity
estimate
\begin{equation}
    \int_{t_{1,n}}^{t_{2,n}}\Phi'(t)\,dt
    \ge
    c\int_{t_{1,n}}^{t_{2,n}}d(t)^2\,dt .
    \label{eq:Phi-global-integrated}
\end{equation}
The endpoint bound \eqref{eq:Phi-upper} holds for this unweighted \(\Phi\) as
well. Combining it with \eqref{eq:Phi-global-integrated} gives
\eqref{eq:modulation-interval-integral}. Finally, \eqref{eq:lambdaK-control}
follows from \eqref{eq:modulation-interval-integral} and
\[
    |\lambda_K'(t)|\lesssim d(t)+\zeta_n,
\]
as in the finite-time case.
\end{proof}
We now record the interval decomposition which will be used in the localized
virial argument. This is the same decomposition as in \cite{JL}: the
collision interval is divided into small-modulation pieces, transition pieces,
and compactness pieces. The only input from the present equation is
Proposition~\ref{prop:modulation-interval-control}, which gives the estimate on
the small-modulation pieces.
\begin{prop}[Interval decomposition]
\label{prop:interval-decomposition}
There exist constants \(\theta_0>0\), \(\varepsilon_*>0\), and \(C_0>0\) with the
following property. Let \(\theta_n\downarrow0\) be any sequence such that
\[
    \max\{\varepsilon_n,\delta_n\}\le \theta_n\le \theta_0 ,
\]
where \(\varepsilon_n\) is the endpoint parameter of the collision interval
\(I_n=[a_n,b_n]\). Then, after passing to a subsequence, for all sufficiently
large \(n\) there exist an integer \(N_n\ge1\) and a partition
\[
\begin{aligned}
    a_n
    &=
    e_{0,n}^{L}
    <
    e_{0,n}^{R}
    \le
    c_{0,n}^{R}
    \le
    d_{0,n}^{R}
    \le
    f_{0,n}^{R}
    \le
    f_{1,n}^{L}
    \le
    d_{1,n}^{L}
    \le
    c_{1,n}^{L}
    \le
    e_{1,n}^{L}
    <
    e_{1,n}^{R}
    \le \cdots                                      \\
    &\qquad\qquad\qquad\qquad
    \cdots
    \le
    c_{N_n,n}^{L}
    \le
    e_{N_n,n}^{L}
    <
    e_{N_n,n}^{R}
    =
    b_n ,
\end{aligned}
\]
with the following properties.
\begin{enumerate}
    \item For every \(m=0,\ldots,N_n\), $ d(t)\le\eta_0$ for all $t\in[e_{m,n}^{L},e_{m,n}^{R}],$
    and
    \begin{equation}
        \int_{e_{m,n}^{L}}^{e_{m,n}^{R}} d(t)\,dt
        \le
        C_0\theta_n^{\frac4{D-2}}
        \min\bigl(\mu_*(e_{m,n}^{L}),\mu_*(e_{m,n}^{R})\bigr).
        \label{eq:small-modulation-piece}
    \end{equation}
    \item On the transition pieces one has $d(t)\ge\theta_n .$
    More precisely, for \(m=0,\ldots,N_n-1\), $d(t)\ge\theta_n$
    on
    \[
        [e_{m,n}^{R},c_{m,n}^{R}]
        \cup
        [f_{m,n}^{R},f_{m+1,n}^{L}]
        \cup
        [c_{m+1,n}^{L},e_{m+1,n}^{L}].
    \]
    \item On the compactness pieces one has $ d(t)\ge\varepsilon_* .$
    Namely, for \(m=0,\ldots,N_n-1\), $d(t)\ge\varepsilon_*$
    on $[c_{m,n}^{R},f_{m,n}^{R}]
        \cup
        [f_{m+1,n}^{L},c_{m+1,n}^{L}].$
    \item The intervals really leave the small-modulation regime:
    \[
        d(d_{m,n}^{R})\ge\eta_0,
        \qquad
        d(d_{m+1,n}^{L})\ge\eta_0,
        \qquad
        0\le m\le N_n-1.
    \]
    \item The endpoints \(c_{m,n}^{R}\) and \(c_{m,n}^{L}\) are threshold
    crossings:
    \[
        d(c_{m,n}^{R})=\theta_n,
        \qquad
        d(c_{m,n}^{L})=\theta_n .
    \]
    \item For every \(m=0,\ldots,N_n-1\), either $d(t)\ge\varepsilon_*$ for all $t\in[c_{m,n}^{R},c_{m+1,n}^{L}],$ or $d(f_{m,n}^{R})=d(f_{m+1,n}^{L})=\varepsilon_* .$
    
    \item On the pieces connecting the modulation intervals to the compactness
    region, the auxiliary scale is comparable:
    \begin{equation*}
        \sup_{t\in[e_{m,n}^{L},c_{m,n}^{R}]}\mu_*(t)
        \le
        C_0
        \inf_{t\in[e_{m,n}^{L},c_{m,n}^{R}]}\mu_*(t),
    \end{equation*}
    and
    \begin{equation*}
        \sup_{t\in[c_{m,n}^{L},e_{m,n}^{R}]}\mu_*(t)
        \le
        C_0
        \inf_{t\in[c_{m,n}^{L},e_{m,n}^{R}]}\mu_*(t),
    \end{equation*}
    whenever the intervals are defined.
\end{enumerate}
\end{prop}
\begin{proof}
The decomposition is obtained by the stopping-time construction of
\cite{JL}. Starting from \(a_n\), one follows the solution as long as it
remains in the small-modulation region \(d(t)\le\eta_0\); this gives the
intervals \([e_{m,n}^{L},e_{m,n}^{R}]\). On these intervals,
\eqref{eq:small-modulation-piece} follows from
Proposition~\ref{prop:modulation-interval-control}, since
\(\theta_n\ge\delta_n\) and the pieces where \(d(t)\le\theta_n\) contribute only
to the right-hand side after increasing \(C_0\).
The points \(c_{m,n}^{R}\), \(d_{m,n}^{R}\), \(f_{m,n}^{R}\), and their left-hand
analogues are then defined as the first or last hitting times of the levels
\(\theta_n\), \(\eta_0\), and \(\varepsilon_*\), exactly as in the wave equation.
This immediately gives properties (2)--(6).

It remains only to justify the comparability of \(\mu_*\) in (7). If, for
example, \(\mu_*\) varied by a large factor on
\([e_{m,n}^{L},c_{m,n}^{R}]\), the Lipschitz property of \(\mu_*\) would give a
subinterval whose length is comparable to the smaller value of \(\mu_*\). Since
\(d(t)\ge\theta_n\) after leaving the small-modulation piece, Lemma
\ref{lem:excursion-length} would then force an excursion of length comparable to
\(\mu_*\), contradicting the way the stopping times were chosen. This is the
standard finite-speed argument; it uses only Lemmas~\ref{lem:auxiliary-scale} and
\ref{lem:excursion-length}. The damping term does not enter.
\end{proof}

\subsection{Localized virial estimates and the contradiction}

We now complete the no-return argument by integrating a localized virial identity
over the partition obtained in Proposition~\ref{prop:interval-decomposition}.
The finite-time and global cases share the same interval decomposition. The only
difference is the treatment of the damping term in the virial identity.

We next choose the moving cut-off used in the localized virial argument. Denote
\[
    \Omega_{\rho}(\boldsymbol u)
    :=
    \Omega_{1,\rho}(\boldsymbol u)
    +
    \frac{D-2}{2}\Omega_{2,\rho}(\boldsymbol u).
\]

\begin{lemma}[Choice of the moving cut-off]
\label{lem:moving-cutoff}
There exist \(\theta_0>0\) and, after passing to a subsequence, locally
Lipschitz functions $\rho_n:I_n\to(0,\infty)$
with the following properties.
\begin{enumerate}
    \item The cut-off scale separates the interior cluster from the exterior
    region:
    \begin{equation}
        \inf_{t\in I_n}\frac{\rho_n(t)}{\mu_*(t)}\to\infty,
        \qquad
        \sup_{t\in I_n}\frac{\rho_n(t)}{\nu_n(t)}\to0.
        \label{eq:rho-scale-separation}
    \end{equation}
    \item The endpoint contribution is negligible:
    \begin{equation}
        \rho_n(a_n)\|\partial_tu(a_n)\|_{L^2}
        +
        \rho_n(b_n)\|\partial_tu(b_n)\|_{L^2}
        =
        o_n(1)\max\{\mu_*(a_n),\mu_*(b_n)\}.
        \label{eq:rho-endpoint-small}
    \end{equation}
    \item If \(t_0\in I_n\) and \(d(t_0)\le\frac12\theta_0\), then $|\rho_n'(t)|\le1$
    for almost every \(t\) in a neighborhood of \(t_0\).
    \item The localized scaling error satisfies
\begin{equation}
    \sup_{t\in I_n}
    |\Omega_{\rho_n(t)}(\boldsymbol u(t))|
    \to0.
    \label{eq:Omega-rho-vanishes}
\end{equation}

    \item In the global case \(T_+=\infty\), the functions \(\rho_n\) can be
    chosen so that, on every small-modulation piece of the partition in
    Proposition~\ref{prop:interval-decomposition},
    \begin{equation}
        |\mathcal V_G(t)|
        \le C_0 d(t),
        \qquad
        \mathcal V_G(t)
        :=\lelan \partial_t u(t)\mid \chi_{\rho_n(t)}\left(r\partial_ru(t)+\frac{D-2}{2}u(t)\right) \rilan
        \label{eq:global-virial-d-small}
    \end{equation}
\end{enumerate}
\end{lemma}
\begin{proof}
The construction follows the cut-off selection of \cite{JL}. The scale
\(\mu_*\) is \(1\)-Lipschitz and measures the size of the interior cluster,
whereas \(\nu_n\) separates this cluster from the exterior region. By
Proposition~\ref{prop:collision-reduction} and Lemma~\ref{lem:auxiliary-scale},
we have $\frac{\mu_*(t)}{\nu_n(t)}\to0$
uniformly on \(I_n\). Hence one may choose a locally Lipschitz intermediate scale
\(\rho_n\) satisfying \eqref{eq:rho-scale-separation}, and then regularize it so
that \(|\rho_n'|\le1\) whenever \(d(t)\) is small. This gives (1) and (3).
The endpoint condition \eqref{eq:rho-endpoint-small} is obtained by choosing
\(\rho_n(a_n)\) and \(\rho_n(b_n)\) inside the interval $\mu_*(t)\ll \rho_n(t)\ll \nu_n(t)$
slowly enough. Since \(d(a_n),d(b_n)\to0\), the endpoint kinetic energy in the
region selected by \(\rho_n\) is negligible compared with the scale
\(\max\{\mu_*(a_n),\mu_*(b_n)\}\).

It remains to justify \eqref{eq:Omega-rho-vanishes}. The scale separation
\(\mu_*(t)\ll\rho_n(t)\ll\nu_n(t)\) places the cut-off in an annular region
between the interior \(K\)-bubble cluster and the exterior part. The pure
multi-bubble contribution cancels in
\(\Omega_{1,\rho}+\frac{D-2}{2}\Omega_{2,\rho}\) by the scaling identity, while
the exterior error, the annular energy, and the radiation term in the finite-time
case are \(o_n(1)\) by the definition of \(\zeta_n\). This proves
\eqref{eq:Omega-rho-vanishes}.

Finally assume \(T_+=\infty\). We prove \eqref{eq:global-virial-d-small}. It is
enough to show that on the small-modulation pieces where this estimate is used
the cut-off scale \(\rho_n(t)\) is uniformly bounded.
We first record the following consequence of the global dissipation. Let
\([c_n,d_n]\subset I_n\) be any sequence of subintervals on which the solution
makes a genuine excursion, namely
\[
    d(c_n)\le\theta_n,\qquad d(d_n)\le\theta_n,
    \qquad
    \sup_{t\in[c_n,d_n]}d(t)\ge\varepsilon_* .
\]
Since the collision intervals $I_n$ are pairwise disjoint, the intervals $[c_n,d_n]$ are pairwise disjoint and
\begin{equation}
    d_n-c_n\to0.
    \label{eq:global-excursion-length-zero}
\end{equation}
Indeed, otherwise the compactness lemma applied on such intervals gives
\[
    \int_{c_n}^{d_n}\|\partial_tu(t)\|_{L^2}^2\,dt
    \gtrsim d_n-c_n .
\]
Since the intervals are disjoint and $\int_0^\infty\|\partial_tu(t)\|_{L^2}^2\,dt<\infty,$
we obtain \eqref{eq:global-excursion-length-zero}. Combining this with
Lemma~\ref{lem:excursion-length} yields 
\begin{equation*}
    \max\{\mu_*(c_n),\mu_*(d_n)\}\to0.
\end{equation*}
In particular, by Lemma~\ref{lem:auxiliary-scale} and
Proposition~\ref{prop:modulation-interval-control}, on the adjacent
small-modulation pieces we have
\begin{equation}
    \sup \lambda_K(t)\to0 .
    \label{eq:lambdaK-small-global-pieces}
\end{equation}
We now distinguish two cases. If \(K<N\), the cut-off is chosen so that
\[
    \mu_*(t)\ll \rho_n(t)\ll \nu_n(t)\ll \mu_{K+1,n}(t).
\]
The exterior scales are uniformly bounded in the global case; in particular
\[
    \mu_{K+1,n}(t)\le \mu_{N,n}(t)\le C .
\]
Hence \(\rho_n(t)\le C\) on the relevant small-modulation pieces.
If \(K=N\), then \eqref{eq:lambdaK-small-global-pieces} allows us to choose an
auxiliary exterior scale \(\mu_{K+1,n}(t)\) satisfying
\[
    \lambda_K(t)\ll \mu_{K+1,n}(t)\ll1
\]
on these pieces. We then choose the cut-off so that
\[
    \mu_*(t)\ll \rho_n(t)\ll \mu_{K+1,n}(t),
\]
and again \(\rho_n(t)\le C\).
Thus in both cases the cut-off scale is uniformly bounded on every
small-modulation piece where \eqref{eq:global-virial-d-small} is used. Since
\(\boldsymbol u^*\equiv0\) in the global case, the definition of the distance
function gives
\[
    \|\partial_tu(t)\|_{L^2}\lesssim d(t).
\]
Therefore, by the Cauchy--Schwarz and Hardy's inequality we have
\[
\begin{aligned}
    |\mathcal V_G(t)|
    &\leq\left\|\partial_t u(t)\right\|_{L^2}\left(\rho_n(t)\left\|\partial_ru(t)\right\|_{L^2(0,2\rho_n)}+\left\|u(t)\right\|_{L^2(0,2\rho_n)}\right)  \\
    &\lesssim
    \|\partial_tu(t)\|_{L^2}\,
    \rho_n(t)\|u(t)\|_{\dot H^1}
    \lesssim d(t),
\end{aligned}
\]
which proves \eqref{eq:global-virial-d-small}.
\end{proof}

We now state the localized virial estimate on the pieces where the solution is
not too close to the multi-bubble manifold. Define
\[
    \mathcal V_n(t)
    :=
    \begin{cases}
    \displaystyle
    e^{\alpha(t-T_+)}
    \int_0^\infty
        \partial_tu(t,r)\left(r\partial_ru(t,r)
        +\frac{D-2}{2}u(t,r)\right)\chi_{\rho_n(t)}(r)\,r^{D-1}\,dr,
    & T_+<\infty,\\[1.2em]
    \displaystyle
   \int_0^\infty
        \partial_tu(t,r)\left(r\partial_ru(t,r)
        +\frac{D-2}{2}u(t,r)\right)\chi_{\rho_n(t)}(r)\,r^{D-1}\,dr,
    & T_+=\infty.
    \end{cases}
\]
\begin{prop}[Localized virial descent]
\label{prop:localized-virial-descent}
The following properties hold after passing to a subsequence.
\begin{enumerate}
\item For every \(\sigma>0\), after decreasing \(\eta_0>0\) if necessary,
there exists a sequence \(\theta_n\downarrow0\), chosen so slowly that
\[
    \max\{\varepsilon_n,\delta_n\}\le \theta_n\le\theta_0,
\]
such that the following holds.
Let \([\tilde a_n,\tilde b_n]\subset I_n\) be one of the subintervals obtained
from the elementary Lipschitz subdivision of \(\mu_*\), as in
\cite[Lemma~6.9]{JL}; in particular,
\[
    \tilde b_n-\tilde a_n
    \simeq
    \sup_{t\in[\tilde a_n,\tilde b_n]}\mu_*(t)
    \simeq
    \inf_{t\in[\tilde a_n,\tilde b_n]}\mu_*(t).
\]
Assume that $d(t)\ge \theta_n$ for all $t\in[\tilde a_n,\tilde b_n]$
and, in the global case, also \(d(t)\le\eta_0\) on this interval. Then
    \begin{equation}
        \mathcal V_n(\tilde b_n)
        \le
        \mathcal V_n(\tilde a_n)
        +
        \bigl(\sigma+o_n(1)\bigr)
        \sup_{t\in[\tilde a_n,\tilde b_n]}\mu_*(t).
        \label{eq:virial-almost-nonincrease}
    \end{equation}
    \item For every \(c>0\) and every \(\theta>0\), there exists
    \(\delta=\delta(c,\theta)>0\) such that, for all sufficiently large \(n\), if
    \([\tilde a,\tilde b]\subset I_n\) satisfies
    \[
        \tilde b-\tilde a
        \ge
        c\,\mu_*(\tilde a),
        \qquad
        d(t)\ge\theta
        \quad\text{for all }t\in[\tilde a,\tilde b],
    \]
    then
    \begin{equation}
        \mathcal V_n(\tilde b)-\mathcal V_n(\tilde a)
        \le
        -\delta
        \sup_{t\in[\tilde a,\tilde b]}\mu_*(t).
        \label{eq:virial-strict-descent}
    \end{equation}
    In the global case, after decreasing \(\delta\) if necessary, one also has
    \begin{equation}
        \mathcal V_n(\tilde b)-\mathcal V_n(\tilde a)
        \le
        -\delta(\tilde b-\tilde a).
        \label{eq:virial-strict-time-descent}
    \end{equation}
\end{enumerate}
\end{prop}
\begin{proof}
By the elementary subdivision lemma \cite[Lemma~6.9]{JL}, it suffices to prove
the first assertion on intervals for which
\[
    \tilde b-\tilde a\simeq \sup_{[\tilde a,\tilde b]}\mu_*
    \simeq \inf_{[\tilde a,\tilde b]}\mu_* .
\]
We record the localized virial identities. In the finite-time case, by the
definition of \(\mathcal V_n\) and Lemma~\ref{Virial identity},
\begin{equation}
    \mathcal V_n'(t)
    =
    -e^{\alpha(t-T_+)}
    \int_0^\infty
        (\partial_tu(t,r))^2\chi_{\rho_n(t)}(r)r^{D-1}\,dr
    +
    \mathcal R_n(t).
    \label{eq:finite-localized-virial}
\end{equation}
In the global case,
\begin{equation}
    \mathcal V_n'(t)
    =
    -
    \int_0^\infty
        (\partial_tu(t,r))^2\chi_{\rho_n(t)}(r)r^{D-1}\,dr
    -
    \alpha\mathcal V_n(t)
    +
    \mathcal R_n(t).
    \label{eq:global-localized-virial}
\end{equation}
Here \(\mathcal R_n(t)\) denotes the corresponding localized scaling error
\(\Omega_{\rho_n(t)}(\boldsymbol u(t))\), with the additional factor
\(e^{\alpha(t-T_+)}\) in the finite-time case. By Lemma~\ref{lem:moving-cutoff},
\[
    \sup_{t\in I_n}|\mathcal R_n(t)|=o_n(1).
\]
Hence, on every interval \([\tilde a,\tilde b]\) after the above subdivision,
\begin{equation}
    \int_{\tilde a}^{\tilde b}|\mathcal R_n(t)|\,dt
    \le
    o_n(1)
    \sup_{t\in[\tilde a,\tilde b]}\mu_*(t).
    \label{eq:virial-error-small}
\end{equation}
We now prove the first assertion. In the finite-time case, \eqref{eq:finite-localized-virial} and
\eqref{eq:virial-error-small} immediately
give
\[
    \mathcal V_n(\tilde b_n)-\mathcal V_n(\tilde a_n)
    \le
    o_n(1)
    \sup_{t\in[\tilde a_n,\tilde b_n]}\mu_*(t).
\]
In the global case, using \eqref{eq:global-virial-d-small}, \eqref{eq:global-localized-virial} and
\eqref{eq:virial-error-small}, we obtain
\[
    \mathcal V_n(\tilde b_n)-\mathcal V_n(\tilde a_n)
    \le
    C
    \int_{\tilde a_n}^{\tilde b_n}d(t)\,dt
    +
    o_n(1)
    \sup_{t\in[\tilde a_n,\tilde b_n]}\mu_*(t).
\]
Since \(d(t)\le\eta_0\) on this interval, Proposition
\ref{prop:modulation-interval-control} and Lemma~\ref{lem:auxiliary-scale} imply
\[
    \int_{\tilde a_n}^{\tilde b_n}d(t)\,dt
    \le
    C\eta_0^{\frac4{D-2}}
    \sup_{t\in[\tilde a_n,\tilde b_n]}\mu_*(t).
\]
Choosing \(\eta_0>0\) so small that
\(C\eta_0^{\frac4{D-2}}\le\sigma\) gives
\eqref{eq:virial-almost-nonincrease}. The lower bound
\(d(t)\ge\theta_n\) is included for later application to the transition pieces;
it is harmless here. The sequence \(\theta_n\downarrow0\) is chosen by the usual
diagonal argument from the corresponding fixed-threshold statements.
We now prove the strict descent estimate. We argue by contradiction. Suppose
that \eqref{eq:virial-strict-descent} fails for some fixed \(c>0\) and
\(\theta>0\). Then, after passing to a subsequence, there exist intervals
\([\tilde a_n,\tilde b_n]\subset I_n\) such that
\[
    \tilde b_n-\tilde a_n\ge c\,\mu_*(\tilde a_n),
    \qquad
    d(t)\ge\theta
    \quad\text{for all }t\in[\tilde a_n,\tilde b_n],
\]
and
\begin{equation}
    \mathcal V_n(\tilde b_n)-\mathcal V_n(\tilde a_n)
    \ge
    -o_n(1)
    \sup_{t\in[\tilde a_n,\tilde b_n]}\mu_*(t).
    \label{eq:failure-strict-descent}
\end{equation}
In the finite-time case, \eqref{eq:finite-localized-virial},
\eqref{eq:virial-error-small}, and \eqref{eq:failure-strict-descent} yield
\begin{equation}
    \int_{\tilde a_n}^{\tilde b_n}
    \int_0^\infty
        (\partial_tu(t,r))^2
        \chi_{\rho_n(t)}(r)r^{D-1}\,dr\,dt
    =
    o_n(1)
    \sup_{t\in[\tilde a_n,\tilde b_n]}\mu_*(t).
    \label{eq:localized-kinetic-small-average}
\end{equation}
In the global case we apply the same argument to the locally weighted functional $ e^{\alpha(t-\tilde a_n)}\mathcal V_n(t).$
Indeed,
\[
    \frac{d}{dt}
    \left(e^{\alpha(t-\tilde a_n)}\mathcal V_n(t)\right)
    =
    -e^{\alpha(t-\tilde a_n)}
    \int_0^\infty
        (\partial_tu(t,r))^2
        \chi_{\rho_n(t)}(r)r^{D-1}\,dr
    +
    e^{\alpha(t-\tilde a_n)}\mathcal R_n(t).
\]
Since \(d(t)\ge\theta\) on \([\tilde a_n,\tilde b_n]\), the global dissipation
implies $\tilde b_n-\tilde a_n\to0.$
Indeed, otherwise the compactness lemma on intervals where \(d(t)\ge\theta\)
would give a uniform lower bound for
\[
    \int_{\tilde a_n}^{\tilde b_n}\|\partial_tu(t)\|_{L^2}^2\,dt,
\]
contradicting the global dissipation and the fact that the intervals are
pairwise disjoint. Hence the local weight is uniformly comparable to \(1\), and
\eqref{eq:localized-kinetic-small-average} follows in the global case as well.
By the Lipschitz property of \(\mu_*\) and
\(\tilde b_n-\tilde a_n\ge c\mu_*(\tilde a_n)\), there exists
\(s_n\in[\tilde a_n,\tilde b_n]\) such that
\begin{equation}
    \int_0^{\frac12\rho_n(s_n)}
        (\partial_tu(s_n,r))^2r^{D-1}\,dr
    \to0.
    \label{eq:localized-kinetic-small-time}
\end{equation}
Moreover, by Lemma~\ref{lem:moving-cutoff},
    $\mu_*(s_n)\ll\rho_n(s_n)\ll\nu_n(s_n).$
Combining \eqref{eq:localized-kinetic-small-time} with finite speed of
propagation and Lemma~\ref{lem:auxiliary-scale}, we obtain $d(s_n)\to0.$
This contradicts \(d(s_n)\ge\theta\). Therefore
\eqref{eq:virial-strict-descent} holds.

Finally, in the global case the same argument, combined with the compactness
lemma on intervals where \(d(t)\ge\theta\), gives a fixed lower bound for the
localized kinetic energy per unit time unless \(d(t_n)\to0\) along a subsequence.
The latter is impossible by \(d(t)\ge\theta\). Hence, after decreasing
\(\delta\) if necessary,
\[
    \mathcal V_n(\tilde b)-\mathcal V_n(\tilde a)
    \le
    -\delta(\tilde b-\tilde a),
\]
which is \eqref{eq:virial-strict-time-descent}.
\end{proof}

We now finish the no-return argument. Recall that \(I_n=[a_n,b_n]\) is a
collision interval. Let \(\mathcal V_n\) be the localized virial functional
defined above, and set
\[
    M_n:=\max\{\mu_*(a_n),\mu_*(b_n)\}.
\]
We apply Proposition~\ref{prop:interval-decomposition} with the sequence
\(\theta_n\) given by Proposition~\ref{prop:localized-virial-descent}. Increasing \(\theta_n\) if necessary, we may assume $ \max\{\varepsilon_n,\delta_n\}\le \theta_n\le\theta_0.$
 For a
subinterval \(J\subset I_n\), write
\[
    \Delta_J\mathcal V_n
    :=
    \mathcal V_n(\sup J)-\mathcal V_n(\inf J).
\]
Let \(\mathcal S_n\), \(\mathcal T_n\), and \(\mathcal K_n\) denote respectively
the collections of small-modulation pieces, transition pieces, and compactness
pieces in the decomposition. After applying the elementary subdivision to the transition pieces, we still
denote by \(\mathcal T_n\) the resulting family of transition subintervals.
We first estimate the small-modulation pieces. Let
\(J=[e_{m,n}^{L},e_{m,n}^{R}]\in\mathcal S_n\). On \(J\) we have
\(d(t)\le\eta_0\), and, after decreasing \(\eta_0\) if necessary,
Lemma~\ref{lem:moving-cutoff} gives \(|\rho_n'(t)|\le1\) for a.e. \(t\in J\).
The localized virial identity, the weak form of Corollary~\ref{cor:localized-virial-control} recorded at the end of its proof, and
the estimate \(\|\partial_tu(t)\|_{L^2(0,\rho_n(t))}\lesssim d(t)+\zeta_n\)
give
\begin{equation*}
    \Delta_J\mathcal V_n
    \le
    C\int_J d(t)\,dt
    +
    o_n(1)\sup_{t\in J}\mu_*(t).
\end{equation*}
By Proposition~\ref{prop:interval-decomposition},
\[
    \int_J d(t)\,dt
    \le
    C\theta_n^{\frac4{D-2}}
    \min\{\mu_*(e_{m,n}^{L}),\mu_*(e_{m,n}^{R})\}.
\]
Since \(\theta_n\le\theta_0\), we obtain
\begin{equation}
    \Delta_J\mathcal V_n
    \le
    C\theta_0^{\frac4{D-2}}
    \min\{\mu_*(e_{m,n}^{L}),\mu_*(e_{m,n}^{R})\}
    +
    o_n(1)\sup_{t\in J}\mu_*(t).
    \label{eq:small-piece-bound}
\end{equation}
Next consider a transition piece. By the subdivision in
Proposition~\ref{prop:interval-decomposition} and \cite[Lemma~6.9]{JL}, each transition piece is divided into finitely many subintervals \(J\) such that
\[
    |J|\simeq \sup_{t\in J}\mu_*(t)\simeq \inf_{t\in J}\mu_*(t),
    \qquad
    d(t)\ge\theta_n \quad\text{for all }t\in J.
\]
Moreover, in the global case, \(d(t)\le\eta_0\) on the relevant transition pieces.
Applying Proposition~\ref{prop:localized-virial-descent}
\eqref{eq:virial-almost-nonincrease} to each such \(J\), we get
\begin{equation}
    \Delta_J\mathcal V_n
    \le
    \bigl(\sigma+o_n(1)\bigr)
    \sup_{t\in J}\mu_*(t).
    \label{eq:transition-piece-bound}
\end{equation}
Summing over the transition subintervals and using the comparability of
\(\mu_*\) from Proposition~\ref{prop:interval-decomposition}, we obtain
\begin{equation*}
    \sum_{J\in\mathcal T_n}\Delta_J\mathcal V_n
    \le
    \bigl(\sigma+o_n(1)\bigr)
    \sum_{J\in\mathcal K_n}\sup_{t\in J}\mu_*(t).
\end{equation*}
On each compactness piece \(J\in\mathcal K_n\), we have
\(d(t)\ge\varepsilon_*\). The length lower bound in
Proposition~\ref{prop:interval-decomposition} allows us to apply
Proposition~\ref{prop:localized-virial-descent}
\eqref{eq:virial-strict-descent}, and hence
\begin{equation}
    \Delta_J\mathcal V_n
    \le
    -\delta
    \sup_{t\in J}\mu_*(t),
    \label{eq:compact-piece-descent}
\end{equation}
where \(\delta>0\) is independent of \(n\). In the global case, the additional
time-descent estimate \eqref{eq:virial-strict-time-descent} only strengthens this
bound on long compactness pieces.
It remains to sum the estimates. The stopping-time construction and the
comparability of \(\mu_*\) in Proposition~\ref{prop:interval-decomposition} give
\begin{equation*}
    \sum_{J\in\mathcal S_n}
    \min\{\mu_*(\inf J),\mu_*(\sup J)\}
    +
    \sum_{J\in\mathcal T_n}
    \sup_{t\in J}\mu_*(t)
    \le
    C
    \sum_{J\in\mathcal K_n}
    \sup_{t\in J}\mu_*(t).
\end{equation*}
Therefore, summing \eqref{eq:small-piece-bound},
\eqref{eq:transition-piece-bound}, and \eqref{eq:compact-piece-descent}, we get
\[
    \mathcal V_n(b_n)-\mathcal V_n(a_n)
    \le
    \left(
        C\theta_0^{\frac4{D-2}}
        +C\sigma
        -\delta
        +o_n(1)
    \right)
    \sum_{J\in\mathcal K_n}
    \sup_{t\in J}\mu_*(t).
\]
Choose first \(\theta_0>0\) and \(\sigma>0\) so small that $ C\theta_0^{\frac4{D-2}}+C\sigma\le\frac{\delta}{2},$
and then take \(n\) sufficiently large. We obtain
\begin{equation*}
    \mathcal V_n(b_n)-\mathcal V_n(a_n)
    \le
    -c
    \sum_{J\in\mathcal K_n}
    \sup_{t\in J}\mu_*(t)
\end{equation*}
for some \(c>0\).
Since \(I_n\) is a collision interval, the partition contains at least one
compactness piece. Moreover, the stopping-time construction gives
\[
    \sum_{J\in\mathcal K_n}
    \sup_{t\in J}\mu_*(t)
    \gtrsim
    \max\{\mu_*(a_n),\mu_*(b_n)\}
    =
    M_n .
\]
Hence
\begin{equation}
    \mathcal V_n(b_n)-\mathcal V_n(a_n)
    \le
    -cM_n .
    \label{eq:virial-negative-Mn}
\end{equation}
On the other hand, by Lemma~\ref{lem:moving-cutoff}
\eqref{eq:rho-endpoint-small} and the boundedness of the energy,
\[
\begin{aligned}
    |\mathcal V_n(a_n)|+|\mathcal V_n(b_n)|
    \lesssim
    \rho_n(a_n)\|\partial_tu(a_n)\|_{L^2}
    +
    \rho_n(b_n)\|\partial_tu(b_n)\|_{L^2}  
    =
    o_n(1)M_n .
\end{aligned}
\]
Thus $ \mathcal V_n(b_n)-\mathcal V_n(a_n)
    \ge
    -o_n(1)M_n ,$
which contradicts \eqref{eq:virial-negative-Mn}. Hence the assumed collision
intervals cannot exist. Therefore the contradiction assumption
\eqref{eq:d-goes-zero} is false, and $\lim_{t\to T_+}d(t)=0.$
This proves the full-time soliton resolution and completes the proof of
Theorem~\ref{soliton resolution}.

 \section*{Acknowledgment}
The authors would like to thank Professor Kenji Nakanishi for many helpful discussions.  L. Zhao is supported by NSFC Grant of China No. 12271497, No. 12341102 and the National Key Research and Development Program of China No. 2020YFA0713100.

\end{document}